\RequirePackage{fix-cm,luatex85,shellesc, doi}

\documentclass[smallcondensed]{svjour3}

\def\makeheadbox{}

\usepackage[margin=1.8in,letterpaper]{geometry}

\usepackage{listings}

\smartqed{}

\usepackage{amsmath}
\usepackage{amssymb}
\usepackage{upgreek}
\usepackage{tikz}
\usetikzlibrary{calc}

\usepackage{pgfplots}
\pgfplotsset{compat=1.8}
\usepackage{booktabs}

\definecolor{red}{rgb}{0.7592, 0.3137, 0.3020}
\definecolor{blu}{rgb}{0.3098, 0.5059, 0.7412}
\definecolor{grn}{rgb}{0,0.4,0}
\definecolor{org}{rgb}{1.0,0.5,0.0}
\definecolor{prp}{rgb}{0.4,0.35,0.8}

\newcommand{\pd}[2]{\ensuremath{\frac{\partial #1}{\partial #2}} }
\newcommand{\fd}[2]{\ensuremath{\frac{d #1}{d #2}} }

\newcommand{\mat}[1]{{\boldsymbol{\mathrm{#1}} }}
\renewcommand{\vec}[1]{{\boldsymbol{\mathrm{#1}} }}

\newcommand{\EE}{\ensuremath{\mathbb{E}} }
\newcommand{\FF}{\ensuremath{\mathbb{F}} }
\newcommand{\MM}{\ensuremath{\mathbb{M}} }
\newcommand{\VV}{\ensuremath{V}}
\newcommand{\UU}{\ensuremath{U}}
\newcommand{\QQ}{\ensuremath{Q}}
\newcommand{\PP}{\ensuremath{P}}
\spnewtheorem{constraint}{Constraint}{\itshape}{\rmfamily}

\newcommand{\eref}[1]{(\ref{#1})}
\newcommand{\fref}[1]{Fig.~\ref{#1}}
\newcommand{\aref}[1]{Appendix~\ref{#1}}
\newcommand{\sref}[1]{\S\ref{#1}}

\usepackage{url}

\usepackage{xcolor}

\begin{document}
\title{An Energy Stable Approach for Discretizing Hyperbolic Equations with
  Nonconforming Discontinuous Galerkin Methods}
\titlerunning{Nonconforming DG}

\author{Jeremy E. Kozdon%
\and
Lucas C. Wilcox%
\thanks{The views expressed in this document are those of the authors and do not
reflect the official policy or position of the Department of Defense or the U.S.
Government.\\ Approved for public release; distribution unlimited}}
\institute{Jeremy~E.~Kozdon \and Lucas~C.~Wilcox \at{} Department of Applied
Mathematics, Naval Postgraduate School, Monterey, CA, USA \\
\email{\{jekozdon,lwilcox\}@nps.edu}}

\date{23 February 2018}

\maketitle
\begin{abstract}
  When nonconforming discontinuous Galerkin methods are implemented for
  hyperbolic equations using quadrature, exponential energy growth can result
  even when the underlying scheme with exact integration does not support such
  growth.  Using linear elasticity as a model problem, we propose a
  skew-symmetric formulation that has the same energy stability properties for
  both exact and inexact quadrature-based integration. These stability
  properties are maintained even when the material properties are variable and
  discontinuous, and the elements are non-affine (e.g., curved). Additionally,
  we show how the nonconforming scheme can be made conservative and constant
  preserving with variable material properties and curved elements. The analytic
  stability, conservation, and constant preserving results are confirmed through
  numerical experiments demonstrating the stability as well as the accuracy of
  the method.

\end{abstract}
\section{Introduction}
In this paper, we consider the energy stability of the (semi-discrete)
discontinuous Galerkin method on nonconforming meshes. Two key building
blocks for the work are the use of a skew-symmetric form of the governing
equations and the projection (or interpolation) of both the trial and test
functions to nonconforming mortar elements.

It has long been recognized that discretizing hyperbolic equations in
skew-symmetric form is advantageous; see for instance~\cite{Zang1991ANM}.
Recently, there has been a resurgence of interest in skew-symmetric formulations
to improve robustness of high-order methods; see~\cite{ChanWangModaveRemacleWarburton2016jcp,FisherCarpenterNordstromYamaleevSwanson2012jcp,Gassner2013,KoprivaGassner2014SISC,KozdonDunhamNordstrom2013,Nordstrom2006,Warburton2010}.
The skew-symmetric form decouples the stability of the discontinuous Galerkin
discretization of the equations into a volume and surface component. The
stability of the volume terms comes directly from the use of the weak derivative
(derivatives of the test functions) and the stability of the surface terms
comes from a suitably chosen flux.  This stability does not require
a discrete integration-by-parts property (i.e.,
summation-by-parts~\cite{KreissScherer1974}) nor a discrete
chain rule. We note that when the operators do have a summation-by-parts
property this can be used to flip the weak derivatives back to strong
derivatives without impacting the stability of the method (though in the
nonconforming method presented here, the surface integrals that result from this
procedure would be over an element face and not the corresponding mortar
element).

Skew-symmetry also ensures that the same treatment is used for nonconforming
faces in the primal and discrete adjoint equations. For example, this can be
useful when developing discretely exact discretizations for hyperbolic
optimization problems~\cite{WilcoxStadlerBuiThanhGhattas2015}.

In this work, we merge the ideas of skew-symmetry with nonconforming meshes.
Kopriva~\cite{Kopriva1996jcp} and Kopriva, Woodruff, and
Hussaini~\cite{KoprivaWoodruffHussaini2002} laid much of the groundwork for the
use of nonconforming discontinuous Galerkin methods for hyperbolic problems.
These methods were analyzed by Bui-Thanh and
Ghattas~\cite{BuiThanhGhattas2012sinum}, where it was shown that when inexact
quadrature is used, constant coefficient problems on affine meshes can have
energy growth that is not present in the method when exact integration is used.
Recently, Friedrich et al.~\cite{FriedrichEtAl2017} proposed a similar scheme for
nonconforming SBP operators. In their work, affine meshes with constant
coefficients are considered and the nonconforming characteristic of the
discretization is due to different approximation within the elements (i.e.,
hanging elements are not considered).  An important difference between this work
and Friedrich et al.\ is that here we consider the impact of nonconforming mesh
geometry and curvilinear coordinate transforms (i.e., non-affine meshes).

Here we expand upon the literature related to skew-symmetric discretizations by
showing that the skew-symmetric approach is of value for nonconforming methods
(either due to hanging nodes in the mesh or changes in element spaces). Two
critical ideas in this work are:
\begin{itemize}
  \item the use of a skew-symmetric form for linear elasticity so that
    integration-by-parts is not needed discretely; and
  \item evaluation of the skew-symmetric surface integrals in a mortar space so
    that all surface integrals are consistent even when variational crimes are
    present.
\end{itemize}
The first point is now well-known in the literature. The second
point is the core contribution of the work and has not been
discussed previously in the literature.  Unlike many previous skew-symmetric
formulations, for nonconforming meshes skew-symmetry is also of value for
constant coefficient problems on affine meshes to remove the potential
exponential growth of energy when inexact quadrature is used\footnote{Even in the
conforming case, the choice of the quadrature rule for some element types may
require a skew-symmetric form even for constant coefficient problems on affine
meshes (e.g., spectral element method quadrature for the quadrilateral face of
pyramids)~\cite{ChanWangModaveRemacleWarburton2016jcp}.}.

In addition to stability, it is often desirable that schemes be both
conservative and constant preserving. These properties can be lost due to a lack of
a discrete product rule if care is not taken in the computation of the metric
terms~\cite{Kopriva2006jsc}. With nonconforming meshes, additional challenges
arise as the mesh may become discretely discontinuous due to different aliasing
errors being incurred across nonconforming faces. By enforcing a set of accuracy
constraints on the discrete operators as well as requiring a discrete version of
the divergence theorem, we are able to show that our proposed
discretization is both conservative and constant preserving. For isoparametric
hexahedral elements this will require that the mesh is made discretely
continuous, or watertight, after a global coordinate transform and that aliasing
errors in the calculation of metric terms be made consistent across
nonconforming faces.

Throughout this paper, we take variable coefficient elastodynamics as a model
problem, though the approach is straightforward to generalize to other linear
wave problems that can be written in skew-symmetric form. In the results section
we consider isoparametric hexahedral elements but the stability analysis
applies to other element types.

\section{Continuous Problem}\label{sec:continuous}
We consider a velocity-stress formulation of time-dependent linear elasticity in
the domain $\Omega\subset \mathbb{R}^d$:
\begin{align}
  \label{eqn:continuous:pde}
  \rho \pd{v_{i}}{t} &= \pd{\sigma_{ij}}{x_{j}},&
  \pd{\sigma_{ij}}{t} &= \frac{1}{2}C_{ijkl}
  \left(\pd{v_{k}}{x_{l}} + \pd{v_{l}}{x_{k}}\right),
\end{align}
where $d = 2$ or $d = 3$. Unless otherwise noted, summation over
$1, 2, 3$ is implied for terms with twice repeated subscripts; free subscripts
can take any of the values $1, 2, 3$; and in the case of $d=2$ the
derivatives with respect to $x_{3}$ are taken to be $0$. Here $v_{i}$ is the
particle velocity in the $x_{i}$ direction and $\sigma_{ij}$ are
the components of the symmetric stress tensor such that $\sigma_{ij} =
\sigma_{ji}$. The scalar $\rho$ is the density of the material
and $C_{ijkl}$ are the components of the fourth-order stiffness tensor that has
the symmetries: $C_{ijkl} = C_{klij} = C_{jikl} = C_{ijlk}$. In the results
\sref{sec:results}, isotropic elasticity is considered where
\begin{align}
  C_{ijkl} = \lambda \delta_{ij}\delta_{kl} + \mu\left(\delta_{ik}\delta_{jl} +
  \delta_{il}\delta_{jk}\right),
\end{align}
with $\lambda$ and $\mu$ denoting Lam\'{e}'s first and second parameters ($\mu$ is
also known as the shear modulus), and
$\delta_{ij}$ denoting the Kronecker delta that takes a value of $1$ if $i=j$ and
$0$ otherwise. Both $\rho$ and $C_{ijkl}$ are allowed to be spatially dependent
and may include jump discontinuities.

Since the focus of this work is the treatment of nonconforming mesh interfaces,
only the traction-free boundary condition on $\partial \Omega$ is
considered. That is, if $n_{i}$ is a component of the outward pointing normal
vector to $\partial\Omega$ and $T_{i} = \sigma_{ij}n_{j}$ are components of the
traction vector, then the boundary condition is $T_{i} = 0$ on $\partial
\Omega$.

Critical to the stability analysis that follows is the existence of an energy
norm in which the energy of the semi-discrete numerical scheme is
non-increasing. This is motivated by the fact that the continuous problem with
the traction-free boundary condition does not support energy growth, where the
energy in the solution is defined as
\begin{align}
  \label{eqn:energy}
  \mathcal{E} &=
  \int_{\Omega}
  \left(
  \frac{\rho}{2} v_{i}v_{i}
  +
  \frac{1}{2}
  \sigma_{ij}
  S_{ijkl}
  \sigma_{kl}
  \right).
\end{align}
Here $S_{ijkl}$ denotes the components of the fourth-order compliance tensor
which is the inverse of the stiffness tensor, i.e.,
$s_{ij}C_{ijkl}S_{klnm}s_{nm} = s_{ij}s_{ij}$ for all symmetric second-order
tensors with components $s_{ij}$.  The energy equation~\eref{eqn:energy} is a
well-defined norm if $\rho > 0$ and the compliance tensor is positive definite,
i.e., $s_{ij}S_{ijkl}s_{kl} > 0$ for all non-zero, symmetric second-order
tensors with components $s_{ij}$ (e.g., see~\cite{Slaughter2002}).
In the case of isotropic elasticity the components of the compliance tensor are
\begin{align}
  S_{ijkl} =
  - \frac{\lambda}{2\mu(2\mu + 3\lambda)} \delta_{ij}\delta_{kl}
  +
  \frac{1}{4\mu}\left(\delta_{ik}\delta_{jl} + \delta_{il}\delta_{jk}\right),
\end{align}
and the compliance tensor is positive definite if $\mu > 0$ and $K = \lambda +
2\mu/3 > 0$; $K$ is known as the bulk modulus of the material.

To see that the traction-free boundary condition does not lead to energy growth,
the time derivative of the energy equation~\eref{eqn:energy} is considered:
\begin{align}
  \label{eqn:energy:diss1}
  \fd{\mathcal{E}}{t}
  &=
  \int_{\Omega}
  \left(
  \rho v_{i}\pd{v_{i}}{t}
  +
  \sigma_{ij}
  S_{ijkl}
  \pd{\sigma_{kl}}{t}
  \right)
  =
  \int_{\Omega}
  \left(
  v_{i}\pd{\sigma_{ij}}{x_{j}}
  +
  \sigma_{ij}
  \pd{v_{i}}{x_{j}}
  \right)
\end{align}
where~\eref{eqn:continuous:pde} has been used to change time derivatives into
spatial derivatives. By applying the divergence theorem and substituting in the
traction-free boundary condition, the energy rate of change is then
\begin{align}
  \label{eqn:energy:diss}
  \fd{\mathcal{E}}{t}
  &=
  \int_{\partial\Omega}
  v_{i}\sigma_{ij}n_{j}
  =
  \int_{\partial\Omega}
  v_{i}T_{i} = 0.
\end{align}
We formalize this in the following theorem.
\begin{theorem}\label{thm:cont:wellposed}
  Problem~\eref{eqn:continuous:pde} with the traction-free boundary condition
  $T_{i} = 0$ on $\partial\Omega$ satisfies the energy estimate $\mathcal{E}(t) =
  \mathcal{E}(0)$.
\end{theorem}
\begin{proof}
  Integrating~\eref{eqn:energy:diss} gives the result $\mathcal{E}(t) =
  \mathcal{E}(0)$.
  \qed{}
\end{proof}

\section{Notation for the Discontinuous Galerkin Method}
\subsection{Mesh and Geometry Transformation}\label{sec:mesh}
In this work, the finite element mesh is defined in two steps. First, the domain
is partitioned into a set of non-overlapping elements,
the union of which completely covers the domain. It is assumed that there is an
exact transformation between the physical elements and a set of reference
elements. After this, approximation errors are allowed for in the mappings
between the physical and reference elements which may result in gaps and
overlaps in the mesh and produce a set of elements whose union is no longer
equal to the domain. This could arise if one used an isoparametric approximation
for the geometry on a nonconforming mesh. Other approximations for the geometry
are possible as long as the quadrature rules introduced satisfy the constraints
given in \sref{sec:notation:quad}.  Though the scheme is stable with gaps in the
computational mesh, it is not necessarily conservative and constant preserving,
and in \sref{sec:constant} and \aref{app:div} we show how these properties can
be ensured.

Initially, we let $\Omega$ be partitioned into a finite set of non-overlapping,
possibly nonconforming, $d$-dimensional, curved volume elements. Let
$\EE$ denote the set of all elements and $|\EE|$ be the
total number of volume elements. At this initial stage, we require that
$\bigcup_{e\in\EE} e = \Omega$. These requirements on the partitioning of
$\Omega$ imply that before approximation errors are introduced the mesh has no
gaps. In the computational results \sref{sec:results}, $d = 3$ is considered
with curvilinear hexahedral elements, though the stability analysis is more
general.

\begin{figure}
  \centering
  \begin{tikzpicture}[font=\scriptsize,scale=0.75]
    \begin{scope}[xshift=-1cm]
      \draw (1, 0) -- (3, 0) -- (3, 2) -- (1, 2) -- (1, 0);
      \node at (2, 1) {$e^{1}$};
      \draw (3, 0) -- (4, 0) -- (4, 1) -- (3, 1) -- (3, 0);
      \node at (3.5, 0.5) {$e^{2}$};
      \draw (3, 1) -- (4, 1) -- (4, 2) -- (3, 2) -- (3, 1);
      \node at (3.5, 1.5) {$e^{3}$};
    \end{scope}
    \draw (5, 0) -- (7, 0) -- (7, 2) -- (5, 2) -- (5, 0);
    \node at (6, 1) {$e^{1}$};

    \draw (7.5, 0) -- (7.5, 2);
    \draw (7.25, 0) -- (7.75, 0);
    \draw[->] (7.5, 1) -- (7.75, 1);
    \draw (7.25, 2) -- (7.75, 2);
    \node at (7.5, 2.2) {$m^{2}$};

    \draw (4.5, 0) -- (4.5, 2);
    \draw (4.25, 0) -- (4.75, 0);
    \draw[->] (4.5, 1) -- (4.25, 1);
    \draw (4.25, 2) -- (4.75, 2);
    \node at (4.5, 2.2) {$m^{3}$};

    \draw (5, 2.5) -- (7, 2.5);
    \draw (5, 2.25) -- (5, 2.75);
    \draw[->] (6,2.5) -- (6,2.75);
    \draw (7, 2.25) -- (7, 2.75);
    \node at (6, 3) {$m^{4}$};

    \draw (5, -0.5) -- (7, -0.5);
    \draw (5, -0.25) -- (5, -0.75);
    \draw[->] (6,-0.5) -- (6,-0.75);
    \draw (7, -0.25) -- (7, -0.75);
    \node at (6, -1) {$m^{5}$};

    \draw (8, 0) -- (8.75, 0) -- (8.75, 0.75) -- (8, 0.75) -- (8, 0);
    \node at (8.3750, 0.3750) {$e^{2}$};

    \draw (8, 1.25) -- (8.75, 1.25) -- (8.75, 2) -- (8, 2) -- (8, 1.25);
    \node at (8.3750, 1.6250) {$e^{3}$};

    \draw (8, 1) -- (8.75, 1);
    \draw (8, 1.125) -- (8, 0.875);
    \draw (8.75, 1.125) -- (8.75, 0.875);
    \draw[->] (8.375, 1) -- (8.375, 1.125);
    \node at (9.125, 1) {$m^{1}$};

    \begin{scope}[yshift=0.25cm]
    \draw (8, 2) -- (8.75, 2);
    \draw (8, 2.125) -- (8, 1.875);
    \draw (8.75, 2.125) -- (8.75, 1.875);
    \draw[->] (8.375, 2) -- (8.375, 2.125);
    \node at (8.375, 2.375) {$m^{6}$};
    \end{scope}

    \begin{scope}[yshift=-2.25cm]
    \draw (8, 2) -- (8.75, 2);
    \draw (8, 2.125) -- (8, 1.875);
    \draw (8.75, 2.125) -- (8.75, 1.875);
    \draw[->] (8.375, 2) -- (8.375, 1.875);
    \node at (8.375, 1.75) {$m^{7}$};
    \end{scope}

    \begin{scope}[xshift=-3.5cm]
    \draw (12.5, 1.25) -- (12.5, 2);
    \draw (12.375, 1.25) -- (12.625, 1.25);
    \draw (12.375, 2) -- (12.625, 2);
    \draw[->] (12.5, 1.625) -- (12.75, 1.625);
    \node at (13.125, 1.625) {$m^{8}$};

    \begin{scope}[yshift=-1.25cm]
    \draw (12.5, 1.25) -- (12.5, 2);
    \draw (12.375, 1.25) -- (12.625, 1.25);
    \draw (12.375, 2) -- (12.625, 2);
    \draw[->] (12.5, 1.625) -- (12.75, 1.625);
    \node at (13.125, 1.625) {$m^{9}$};
    \end{scope}
    \end{scope}
    \node at (7, 3.5) {full-side mortar};
    \begin{scope}[xshift=1.5cm]
    \draw (10, 0) -- (12, 0) -- (12, 2) -- (10, 2) -- (10, 0);
    \node at (11, 1) {$e^{1}$};

    \draw (9.5, 0) -- (9.5, 2);
    \draw (9.25, 0) -- (9.75, 0);
    \draw[->] (9.5, 1) -- (9.25, 1);
    \draw (9.25, 2) -- (9.75, 2);
    \node at (9.5, 2.2) {$m^{4}$};

    \draw (10, 2.5) -- (12, 2.5);
    \draw (10, 2.25) -- (10, 2.75);
    \draw[->] (11,2.5) -- (11,2.75);
    \draw (12, 2.25) -- (12, 2.75);
    \node at (11, 3) {$m^{5}$};

    \draw (10, -0.5) -- (12, -0.5);
    \draw (10, -0.25) -- (10, -0.75);
    \draw[->] (11,-0.5) -- (11,-0.75);
    \draw (12, -0.25) -- (12, -0.75);
    \node at (11, -1) {$m^{6}$};

    \begin{scope}[xshift=5cm]
    \begin{scope}[yshift=0.25cm]
    \draw (8, 2) -- (8.75, 2);
    \draw (8, 2.125) -- (8, 1.875);
    \draw (8.75, 2.125) -- (8.75, 1.875);
    \draw[->] (8.375, 2) -- (8.375, 2.125);
    \node at (8.375, 2.375) {$m^{7}$};
    \end{scope}

    \begin{scope}[yshift=-2.25cm]
    \draw (8, 2) -- (8.75, 2);
    \draw (8, 2.125) -- (8, 1.875);
    \draw (8.75, 2.125) -- (8.75, 1.875);
    \draw[->] (8.375, 2) -- (8.375, 1.875);
    \node at (8.375, 1.75) {$m^{8}$};
    \end{scope}
    \end{scope}

    \draw (12.5, 1.25) -- (12.5, 2);
    \draw (12.25, 1.25) -- (12.75, 1.25);
    \draw (12.25, 2) -- (12.75, 2);
    \draw[->] (12.5, 1.625) -- (12.75, 1.625);
    \node at (12.5, 2.2) {$m^{2}$};

    \begin{scope}[xshift=1.5cm]
    \draw (12.5, 1.25) -- (12.5, 2);
    \draw (12.375, 1.25) -- (12.625, 1.25);
    \draw (12.375, 2) -- (12.625, 2);
    \draw[->] (12.5, 1.625) -- (12.75, 1.625);
    \node at (13.125, 1.625) {$m^{9}$};
    \end{scope}

    \begin{scope}[yshift=-1.25cm]
    \begin{scope}[xshift=1.5cm]
    \draw (12.5, 1.25) -- (12.5, 2);
    \draw (12.375, 1.25) -- (12.625, 1.25);
    \draw (12.375, 2) -- (12.625, 2);
    \draw[->] (12.5, 1.625) -- (12.75, 1.625);
    \node at (13.125, 1.625) {$m^{10}$};
    \end{scope}
    \end{scope}

    \draw (12.5, 0) -- (12.5, 0.75);
    \draw (12.25, 0) -- (12.75, 0);
    \draw (12.25, 0.75) -- (12.75, 0.75);
    \draw[->] (12.5, 0.3750) -- (12.75, 0.3750);
    \node at (12.5, -0.2) {$m^{3}$};

    \draw (13, 0) -- (13.75, 0) -- (13.75, 0.75) -- (13, 0.75) -- (13, 0);
    \node at (13.3750, 0.3750) {$e^{2}$};
    \draw (13, 1.25) -- (13.75, 1.25) -- (13.75, 2) -- (13, 2) -- (13, 1.25);
    \node at (13.3750, 1.6250) {$e^{3}$};
    \draw (13, 1) -- (13.75, 1);
    \draw (13, 1.125) -- (13, 0.875);
    \draw (13.75, 1.125) -- (13.75, 0.875);
    \draw[->] (13.375, 1) -- (13.375, 1.125);
    \node at (14.125, 1) {$m^{1}$};
    \end{scope}
    \node at (13.5, 3.5) {split-side mortar};
  \end{tikzpicture}
  \caption{(left) Example of a nonconforming mesh. (center) Example of a mortar
  decomposition of the mesh, where the mortar elements are conforming to the
  larger element across the nonconforming faces; we refer to this type of mortar
  element as a full-side mortar. (right) Example of a mortar decomposition of
  the mesh, where the mortar elements conform to the smaller elements across the
  nonconforming faces; we refer to this type of mortar elements as a
  split-side mortar. (center and right) The arrows on the mortar faces represent
  the direction of the canonical (and arbitrary) mortar element
  normals.}\label{fig:mortars}
\end{figure}
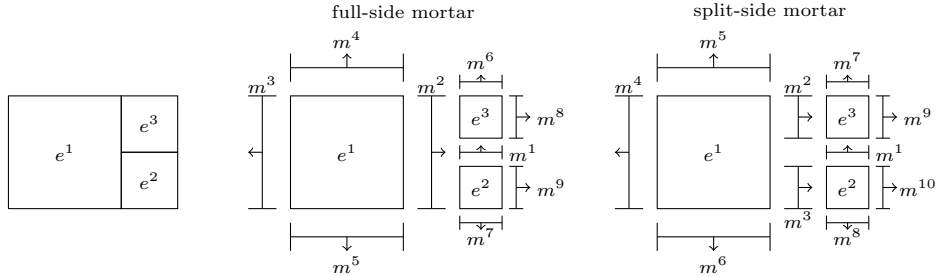
Let $\Gamma = \bigcup_{e\in\EE} \partial e$ where $\partial e$
is the boundary of element $e$; we call $\Gamma$ the mortar and it
contains both the internal mesh interfaces and the outer boundary. The mortar is
partitioned into a finite set of non-overlapping, $(d-1)$-dimensional mortar
elements; the set of mortar elements is denoted by $\MM$. The number of
mortar elements is $|\MM|$ and we require that $\bigcup_{m\in\MM}
m = \Gamma$.
\fref{fig:mortars} contains an example volume mesh and two possible mortar
meshes. The center and right panels show two possible approaches to
partition the internal mortar for the volume mesh in the left panel. When
the mortar elements span the entire nonconforming interface between elements we
call this a \emph{full-side} mortar, and when the mortar elements conform to the
smallest volume faces we call this a \emph{split-side} mortar.

The set $\MM^{e}$ is defined to be the set of mortar elements that volume element
$e \in \EE$ connects to:
\begin{align}
  \MM^{e} =
  \left\{m\in \MM \middle| m \cap \partial e \ne \emptyset \right\}.
\end{align}
Similarly the set $\EE^{m}$ is defined to be the set of volume elements that mortar
element $m \in \MM$ connects to:
\begin{align}
  \EE^{m} =
  \left\{e\in \EE \middle| m \cap \partial e \ne \emptyset \right\}.
\end{align}
It is useful to further partition $\EE^{m}$ into two subsets depending on which
side of the mortar each element resides. To do this, each mortar element $m \in
\MM$ is given a canonical orientation defined by a unit normal (the orientation
of which is arbitrary); the components of the unit normal for
$m$ are denoted $n_{i}^{m}$. In the method developed below, surface integrals
will be performed over the mortar elements and not the faces of the volume
elements, and thus the outward normal to the element may have a different
orientation (sign) than the mortar element normal.
If a volume element $e \in \EE^{m}$ is on the
side of the mortar towards which the normal points the element is said to the be
on the \emph{plus-side} of the mortar, otherwise it is said to be on the
\emph{minus-side} of the mortar. The sets of volume elements on the plus and
minus sides of mortar element $m$ are denoted by $\EE^{+m}$ and
$\EE^{-m}$, respectively, and $\EE^{m} = \EE^{+m} \bigcup \EE^{-m}$.  For the
example mesh shown in the center panel of \fref{fig:mortars} the above-defined
sets are:
\begin{equation}
  \begin{aligned}
  \MM^{e^{1}} &= \left\{m^{2},m^{3},m^{4},m^{5}\right\}, &
  \EE^{m^{1}} &= \left\{e^{2}, e^{3} \right\}, &
  \EE^{m^{2}} &= \left\{e^{1}, e^{2}, e^{3} \right\}, \\
  \MM^{e^{2}} &= \left\{m^{1}, m^{2}, m^{7}, m^{9}\right\}, &
  \EE^{+m^{1}} &= \left\{e^{3}\right\}, &
  \EE^{+m^{2}} &= \left\{e^{2}, e^{3} \right\}, \\
  \MM^{e^{3}} &= \left\{m^{1}, m^{2}, m^{6}, m^{8}\right\}, &
  \EE^{-m^{1}} &= \left\{e^{2} \right\}, &
  \EE^{-m^{2}} &= \left\{e^{1} \right\}.
  \end{aligned}
\end{equation}

Each element $e \in \EE$ is taken to have a reference element $\hat{e}$
where the discretization is specified. It is assume that
there exists a diffeomorphic mapping between the reference and physical elements.
That is, there exist differentiable functions $\vec{x}^{e}$ and $\vec{r}^{e}$
such that if $\vec{r} \in \hat{e}$ then $\vec{x}^{e}(\vec{r}) \in e$ and if
$\vec{x} \in e$ then $\vec{r}^{e}(\vec{x}) \in \hat{e}$. Similarly, for each
mortar element $m \in \MM$ it is assumed that there exists a reference
mortar element $\hat{m}$ along with a diffeomorphic mapping between the
reference and physical mortar elements.  Similar to volume elements, there exist
differentiable functions $\vec{x}^{m}$ and $\vec{r}^{m}$
such that if $\vec{r} \in \hat{m}$ then $\vec{x}^{m}(\vec{r}) \in m$ and if
$\vec{x} \in m$ then $\vec{r}^{m}(\vec{x}) \in \hat{m}$.

For the exact transformation, the Jacobian determinant for volume element $e \in
\EE$ is denoted $J^{e}$. For $d = 2$ the Jacobian determinant is
\begin{align}
  \label{eqn:J2d}
  J^{e} =
  \pd{x^{e}_{1}}{r_{1}}
  \pd{x^{e}_{2}}{r_{2}}
  -
  \pd{x^{e}_{1}}{r_{2}}
  \pd{x^{e}_{2}}{r_{1}},
\end{align}
and for $d = 3$
\begin{align}
  \label{eqn:J3d}
  J^{e} =
  \varepsilon_{ijk}
  \pd{x^{e}_{1}}{r_{i}}
  \pd{x^{e}_{2}}{r_{j}}
  \pd{x^{e}_{3}}{r_{k}},
\end{align}
with $\varepsilon_{ijk}$ being the Levi-Civita permutation symbol
\begin{align}
  \varepsilon_{ijk} =
  \begin{cases}
    +1, & \mbox{\ if\ } ijk \mbox{\ is\ } 123,\;312,\;\mbox{or}\;231,\\
    -1, & \mbox{\ if\ } ijk \mbox{\ is\ } 321,\;132,\;\mbox{or}\;213,\\
    \phantom{-}0, & \mbox{\ otherwise}.
  \end{cases}
\end{align}
The surface Jacobian determinant for mortar
element $m\in\MM$ is $S_{J}^{m}$.  For $d=2$, the surface Jacobian determinant for a mortar
element is
\begin{align}
  S^{m}_{J} = \sqrt{{\left(\fd{x^{m}_{1}}{\xi}\right)}^{2} +
  {\left(\fd{x^{m}_{2}}{\xi}\right)}^{2}},
\end{align}
where the parametric curves $(r^{m}_{1}(\xi), r^{m}_{2}(\xi))$
parameterize the mortar element and
\begin{align}
  \fd{x^{m}_{j}}{\xi} = \pd{x^{m}_{j}}{r_{1}} \fd{r^{m}_{1}}{\xi} + \pd{x^{m}_{j}}{r_{2}} \fd{r^{m}_{2}}{\xi}.
\end{align}
For $d=3$, the surface Jacobian determinant of a mortar element is
\begin{align}
  S_{J}^{m}
  =
  \sqrt{ \varepsilon_{ijk}
         \varepsilon_{inl}
         \pd{x^{m}_{j}}{\xi}
         \pd{x^{m}_{k}}{\eta}
         \pd{x^{m}_{n}}{\xi}
         \pd{x^{m}_{l}}{\eta}
         }
\end{align}
with the mortar element parameterized as $(r^{m}_{1}(\xi, \eta), r^{m}_{2}(\xi,
\eta), r^{m}_{3}(\xi, \eta))$ and
\begin{align}
  \pd{x^{m}_{k}}{\xi } &= \pd{x^{m}_{k}}{r_{i}} \pd{r^{m}_{i}}{\xi },&
  \pd{x^{m}_{k}}{\eta} &= \pd{x^{m}_{k}}{r_{i}} \pd{r^{m}_{i}}{\eta}.
\end{align}

In \aref{app:div} we discuss how we numerically evaluate~\eref{eqn:J3d} as well
as the other metric relations in order to ensure the scheme is conservative and
constant preserving; see also Kopriva~\cite{Kopriva2006jsc} for a discussion
concerning conforming meshes.

\subsection{Function Spaces}
The finite dimensional approximation space for $e \in \EE$ is defined on
the reference element $\hat{e}$, and is denoted by $\hat{\VV}^{e} \subset
\text{L}^{2}\left(\hat{e}\right)$ and has dimension $\dim \hat{\VV}^{e}$. For the
numerical results \sref{sec:results}, tensor product polynomials of degree
at most $N$ are used:
\begin{align}
  \hat{\QQ}^{N,d} := \left\{\prod_{i=1}^{d} r_{i}^{n_{i}}
  \;\middle|\;
  0\le n_{i} \le N, \ \forall i \in [1,d] \right\}.
\end{align}
A corresponding space $\VV^{e}$ for the physical element can be defined
as the space of all functions $q^{e}$ such that $q^{e}(\vec{x}) =
q^{e}(\vec{r}^{e}(\vec{x}))$ for some $q^{e} \in \hat{\VV}^{e}$.
Similar definitions are used for each mortar element $m \in
\MM$, with $\hat{\UU}^{m} \subset
\text{L}^{2}\left(\hat{m}\right)$ being the finite dimensional space defined
on the reference element $\hat{m}$ with dimension $\dim\hat{\UU}^{m}$; in the
results $\hat{\UU}^{m} = \hat{\QQ}^{N,d-1}$ where $N$ is the same as in the
volume approximation.

The operator $\mathcal{P}^{m,e}:  \hat{\VV}^{e} \rightarrow \hat{\UU}^{m}$ is
taken to be an operator with the property that if $q^{e} \in \hat{\VV}^{e}$ then
$\mathcal{P}^{m,e}{q^{e}} \in \hat{\UU}^{m}$. Though not required for
stability, the most natural way to define this operator is as an
L$^{2}$-projection. Namely if $q^{e} \in \hat{\VV}^{e}$ and $m \in \MM^{e}$,
then $\mathcal{P}^{m,e}{q^{e}}$ is constructed so that for all $\phi^{m}
\in \hat{\UU}^m$
\begin{align}
  \int_{\hat{m}} \phi^{m} \mathcal{P}^{m,e} q^{e} = \int_{\hat{m}^{e}}
  \phi^{m} q^{e},
\end{align}
where $\hat{m}^{e}$ is the portion of the reference mortar element $\hat{m}$
that corresponds to the intersection in physical space of $m$ and $\partial e$;
see \aref{app:project} for more details on the construction of
$\mathcal{P}^{m,e}$. Note that when the $m \cap \partial e = m$ the operator
$\mathcal{P}^{m,e}$ is an interpolation operator from the volume element to the
mortar element (assuming that the order of functions on the mortar is greater
than or equal to the volume element faces). For instance, if the mortar elements
between nonconforming elements were as shown in the center panel of
\fref{fig:mortars}, then $\mathcal{P}^{m^2,e^2}$ and $\mathcal{P}^{m^2,e^3}$
would be components of the L$^{2}$-projection operator between from $e^{2}$ and
$e^{3}$ to $m^{2}$, and $\mathcal{P}^{m^2,e^1}$ would be an interpolation
operator.  On the other hand, if the mortar elements between nonconforming
elements were as shown in the right panel of \fref{fig:mortars}, then
$\mathcal{P}^{m^2,e^1}$, $\mathcal{P}^{m^2,e^3}$, $\mathcal{P}^{m^3,e^1}$, and
$\mathcal{P}^{m^3,e^2}$ would all be interpolation operators.

If $\zeta^{e}_{n} \in \hat{\VV}^{e}$ for $n = 1, 2, \dots,
\dim\hat{\VV}^{e}$ are
linearly independent basis functions for $\hat{\VV}^{e}$, then
$q^{e} \in \hat{\VV}^{e}$ can be written as
\begin{align}
  \label{eqn:dof:basis}
  q^{e} &= \sum_{n=1}^{\dim\hat{\VV}^{e}} \mathrm{q}^{e}_{n} \zeta^{e}_{n},
\end{align}
where $\mathrm{q}_{n}^{e}$ are the scalar degree of freedom and are stored as the
vector
\begin{align}
  \vec{q}^{e} &=
  \begin{bmatrix}
    \mathrm{q}_{1}^{e} \\ \mathrm{q}_{2}^{e} \\ \vdots \\
    \mathrm{q}_{\dim\hat{\VV}^{e}}^{e}
  \end{bmatrix}.
\end{align}
Similar notation is used to represent functions on the reference mortar element
$\hat{m}$ with the vector $\vec{q}^{m}$ being the $\dim\hat{\UU}^{m}$ degrees
of freedom representing $q^{m} \in \hat{\UU}^{m}$.

\subsection{Quadrature}\label{sec:notation:quad}
To allow for a more general formulation, we allow the exact geometry
transformations to be approximated, such as by an isoparametric geometry
approximation, which leads to the definition of \emph{approximate} physical
elements.  Namely, it is assumed that there exist approximate
transformations $\vec{r}^{e}_{h}$ and $\vec{x}^{e}_{h}$ that transform between the
reference element $\hat{e}$ and an approximate physical element $e_{h} =
\vec{r}^{e}_{h}\left(\hat{e}\right)$ where $e_{h} \approx e$;
approximation of the mortar element transforms leads to approximate mortar
elements $m_{h} = r^{e}_{h}(\hat{m}) \approx m$.  The introduction of
the approximate physical elements $e_{h}$ means that it is possible that
$\bigcup_{e_{h}\in\EE_{h}}e_{h} = \Omega_{h} \ne \Omega$.  Similarly,
depending on how the approximate transformations are defined, the
computational mesh may now have gaps and overlap between neighboring
elements.\footnote{In principle one could introduce approximations of the
reference elements as well so that $\hat{e} \ne \hat{e}_{h}$, but since many
methods are specified using straight-sided reference elements this is not
considered here.} As noted above, allowing for gaps and overlaps has an impact
on the conservation and constant preserving properties of the method and this
will be addressed in more detail in~\sref{sec:constant}.

The $\kappa^{e}$-weighted inner product over $e\in\EE$ between $p^{e}
\in \hat{\VV}^{e}$ and $q^{e} \in \hat{\VV}^{e}$ is approximated as
\begin{align}
  \int_{\hat{e}}\kappa^{e} J^{e} p^{e} \; q^{e} \approx
  {\left(\vec{p}^{e}\right)}^{T}\mat{M}_{\kappa}^{e}\vec{q}^{e},
\end{align}
where $\mat{M}_{\kappa}^{e}$ is a symmetric matrix. If $\kappa^{e} > 0$ then it is
assumed that $\mat{M}_{\kappa}^{e}$ is positive definite. This positive definite
assumption on the mass matrix $\mat{M}^{e}_{\kappa}$ does not require any
particular assumptions concerning the approximation of the geometry
transformation. For example, the Jacobian determinant could be computed using the
exact transformation~\eref{eqn:J3d} or computed to respect the metric
identities~\cite{Kopriva2006jsc}.

If $S_{ijkl}$ is a component of a positive-definite, fourth-order tensor then it
is not required that $\mat{M}_{S_{ijkl}}^{e}$ be positive definite (since any
individual component of a positive-definite tensor need not be positive). That
said, it is required that
${\left(\vec{s}^{e}\right)}_{ij}^{T}\mat{M}_{S_{ijkl}}^{e}\vec{s}_{kl}^{e}
\ge 0$ for all symmetric second-order tensors whose components
$s_{ij}^{e}\in\hat{\VV}^{e}$ satisfy $s_{ij}^{e} = s_{ji}^{e}$. Defining
\begin{align}
  \label{eqn:mass:stress}
  \mat{\bar{M}}^{e}_{S}
  =
  \begin{bmatrix}
      \mat{M}^{e}_{S_{1111}} &   \mat{M}^{e}_{S_{1122}} &   \mat{M}^{e}_{S_{1133}} & 2 \mat{M}^{e}_{S_{1123}} & 2 \mat{M}^{e}_{S_{1113}} & 2 \mat{M}^{e}_{S_{1112}} \\
      \mat{M}^{e}_{S_{1122}} &   \mat{M}^{e}_{S_{2222}} &   \mat{M}^{e}_{S_{2233}} & 2 \mat{M}^{e}_{S_{2223}} & 2 \mat{M}^{e}_{S_{2213}} & 2 \mat{M}^{e}_{S_{2212}} \\
      \mat{M}^{e}_{S_{1133}} &   \mat{M}^{e}_{S_{2233}} &   \mat{M}^{e}_{S_{3333}} & 2 \mat{M}^{e}_{S_{3323}} & 2 \mat{M}^{e}_{S_{3313}} & 2 \mat{M}^{e}_{S_{3312}} \\
    2 \mat{M}^{e}_{S_{1123}} & 2 \mat{M}^{e}_{S_{2223}} & 2 \mat{M}^{e}_{S_{3323}} & 4 \mat{M}^{e}_{S_{2323}} & 4 \mat{M}^{e}_{S_{2313}} & 4 \mat{M}^{e}_{S_{2312}} \\
    2 \mat{M}^{e}_{S_{1113}} & 2 \mat{M}^{e}_{S_{2213}} & 2 \mat{M}^{e}_{S_{3313}} & 4 \mat{M}^{e}_{S_{2313}} & 4 \mat{M}^{e}_{S_{1313}} & 4 \mat{M}^{e}_{S_{1312}} \\
    2 \mat{M}^{e}_{S_{1112}} & 2 \mat{M}^{e}_{S_{2212}} & 2 \mat{M}^{e}_{S_{3312}} & 4 \mat{M}^{e}_{S_{2312}} & 4 \mat{M}^{e}_{S_{1312}} & 4 \mat{M}^{e}_{S_{1212}}
  \end{bmatrix},
\end{align}
the above restriction on $\mat{M}_{S_{ijkl}}^{e}$ can be restated as requiring
that $\mat{\bar{M}}^{e}_{S}$ be symmetric, positive definite.

Integrals involving spatial derivatives of the solution are approximated as
\begin{align}
  \int_{\hat{e}}J^{e} p^{e} \; \pd{q^{e}}{x_{j}}
  =
  \int_{\hat{e}}J^{e} p^{e} \; \pd{r_{k}^{e}}{x_{j}}\pd{q^{e}}{r_{k}}
  \approx
  {\left(\vec{p}^{e}\right)}^{T}\mat{S}_{j}^{e}\vec{q}^{e},
\end{align}
where we highlight the fact that the stiffness matrix $\mat{S}_{j}^{e}$ contains the
metric terms. One feature of our discretization is that no
summation-by-parts~\cite{KreissScherer1974} property between stiffness and mass
matrices is required for the scheme to be stable. In the results
\sref{sec:results} we use tensor product hexahedral elements, with collocation
of the integration and interpolation points, and the stiffness matrix
$\mat{S}_{j}^{e}$ is taken to be of the form
\begin{align}
  \mat{S}_{j}^{e} &= \mat{M}^{e}
  \mat{J}^{e}\mat{r}_{k,j}^{e}\mat{D}_{k}^{e},
\end{align}
where $\mat{M}^{e}$ is the diagonal matrix of tensor product quadrature weights,
$\mat{J}^{e}$ and $\mat{r}_{k,j}^{e}$ are diagonal matrices of the
approximations of $J^{e}$ and $\partial r_{k}^{e}/\partial x_{j}$, respectively,
at the nodes, and $\mat{D}_{k}^{e}$ is the differentiation matrix
with respect to the reference direction $r_{k}$.

Surface integrals over $m \in \MM$ are assumed to be
approximated using a primitive, positive weight, $n_{q}^{m}$-point quadrature
rule defined for the reference element $\hat{m}$. Thus, if $\omega_{n}^{m} > 0$
are the weights and $\vec{r}_{n}^{m}$ are the nodes of the rule (with
$n=1,2,\dots,n_{q}^{m}$) then $\mat{W}^{m}$ is the
diagonal matrix of quadrature weights and surface Jacobian determinant evaluated at the
quadrature nodes.  Namely, the diagonal elements of the matrix are
\begin{align}
  \mat{W}^{m}_{nn} &= \omega_{n}^{m} S_{J}^{m}\left(\vec{r}_{n}^{m}\right)&
  \mbox{(no summation over $n$)}.
\end{align}
If $\mat{L}^{m}$ is the interpolation matrix that goes from the degrees
of freedom of $\hat{\UU}^{m}$ to values at the quadrature nodes, then
inner products over the mortar $m \in \MM$ between $p^{m} \in
\hat{\UU}^{m}$ and $q^{m} \in \hat{\UU}^{m}$ are approximated as
\begin{align}
  \int_{\hat{m}} S_{J}^{m} p^{m} \; q^{m} \approx
  {\left(\vec{p}^{m}\right)}^{T}
  {\left(\mat{L}^{m}\right)}^{T} \mat{W}^{m} \mat{L}^{m}
  \vec{q}^{m}.
\end{align}
If the intersection of the boundary of volume element $e \in \EE$ and $m \in
\MM$ in non-zero, that is $\partial e \cap m \ne \emptyset$, then integrals
between $p^{m} \in \hat{\UU}^{m}$ and $q^{e} \in \hat{\VV}^{e}$ are
approximated as
\begin{align}
  \int_{\hat{m}} S_{J}^{m}\;p^{m}\;\mathcal{P}^{m,e}q^{e}
  \approx
  {\left(\vec{p}^{m}\right)}^{T}
  {\left(\mat{L}^{m}\right)}^{T} \mat{W}^{m} \mat{P}^{m,e}
  \vec{q}^{e}
  ,
\end{align}
where we note that $\mat{P}^{m,e}$ goes directly from the volume element to the
quadrature nodes, and thus includes both the projection operator
$\mathcal{P}^{m,e}$ (or its approximation) and the interpolation matrix
$\mat{L}^{m}$.
\section{Discontinuous Galerkin Method}
\subsection{Exact Integration}
A skew symmetric, discontinuous Galerkin formulation based on~\eref{eqn:continuous:pde} is:  For each $e \in \EE$,
find $\mathrm{v}_{j}^{e} \in \hat{\VV}^{e}$ and symmetric $\sigma_{ij}^{e} =
\sigma_{ji}^{e} \in \hat{\VV}^{e}$ such that for all $\mathrm{\phi}_{j}^{e} \in
\hat{\VV}^{e}$ and symmetric $\psi_{ij}^{e} = \psi_{ji}^{e} \in \hat{\VV}^{e}$
the following holds:
\begin{align}
  \label{eqn:dg:vel}
  \int_{\hat{e}}
  J^{e} \rho^{e}\; \phi_{i}^{e} \pd{v_{i}^{e}}{t}
  =&
  -
  \int_{\hat{e}}
  J^{e}
  \pd{\phi_{i}^{e}}{x_{j}}\sigma_{ij}^{e}
  +
  \sum_{m \in \MM^{e}}
  \int_{\hat{m}} S_{J}^{m}\phi_{i}^{m,e}T_{i}^{*m[e]},\\
  \label{eqn:dg:stress}
  \int_{\hat{e}}
  J^{e}
  \psi_{ij}^{e} S_{ijkl}^{e} \pd{\sigma_{kl}^{e}}{t}
  =&
  \int_{\hat{e}}
  \frac{J^{e}}{2} \psi_{ij}^{e} \left(\pd{v_{i}^{e}}{x_{j}} + \pd{v_{j}^{e}}{x_{i}}\right)
  \\\notag
  &+
  \sum_{m \in \MM^{e}}
  \int_{\hat{m}} S_{J}^{m} {\left(n_{j} \psi_{ij}\right)}^{m,e}
  \left(v_{i}^{*m} - v_{i}^{m[e]}\right).
\end{align}
Here the mortar projected \emph{test} velocity from element $e$ is defined as
$\phi_{i}^{m,e} = \mathcal{P}^{m,e}\phi_{i}$. The mortar projected \emph{test}
traction ${\left(n_{j} \psi_{ij}\right)}^{m,e}$ can be defined by
projecting the element face computed \emph{test} traction to the mortar,
${\left(n_{j} \psi_{ij}\right)}^{m,e} = \mathcal{P}^{m,e} n_{j}^{e}\psi_{ij}$
where $n_{j}^{e}$ is the outward normal for element $e$, or by projecting the
\emph{test} stress tensor to the mortar and then computing the \emph{test}
traction on the mortar, ${\left(n_{j} \psi_{ij}\right)}^{m,e} = n_{j}^{m,e}
\mathcal{P}^{m,e} \psi_{ij}$ where $n_{j}^{m,e}$ is the mortar normal that is
oriented outward to element $e$. Both approaches will result in a stable
numerical method, and in the results section the latter approach is used
(projecting the stresses).
The quantity $v_{i}^{m[e]} = v_{i}^{+m}$ if $e \in \EE^{+m}$ and
$v_{i}^{m[e]} = v_{i}^{-m}$ if $e \in \EE^{-m}$ where
\begin{align}
  \label{eqn:mortar:state:velocity}
  v_{i}^{\pm m}
  &=
  \sum_{e \in \EE^{\pm m}}
  \mathcal{P}^{m,e}
  v_{i}^{e},
\end{align}
are the velocities on the plus and minus side of the mortar.
The vectors $T_{i}^{*m[e]}$ and $v_{i}^{*m}$ are the numerical fluxes which
enforce continuity of traction and velocity across the mortar elements and the
physical boundary conditions. The choice of the numerical flux is critical to
ensure the consistency and stability of the method (as discussed below).  The
superscript $[e]$ in $T_{i}^{*m[e]}$ denotes the fact that this traction is
defined with respect the normal of element $e$. Namely, if elements $e^{1}$ and
$e^{2}$ are both connected to mortar $m$ then $T_{i}^{*m[e^{1}]} =
T_{i}^{*m[e^{2}]}$ if both elements are on the same side of the mortar element
and $T_{i}^{*m[e^{1}]} = -T_{i}^{*m[e^{2}]}$ if they are on opposite sides (due
to the equal but opposite normal vector).

An important feature of~\eref{eqn:dg:vel}--\eref{eqn:dg:stress} is that the
surface integrals are evaluated on the mortar elements, not the volume element
faces. This structure essentially decouples the volume and surface stability,
leading to a semi-discretely stable scheme even when quadrature (or
under-integration) is used.

The energy in element $e$ is defined as
\begin{align}
  \mathcal{E}^{e} &=
  \int_{\hat{e}}
  J^{e}
  \left(
  \frac{\rho}{2} v_{i}^{e}v_{i}^{e}
  +
  \frac{1}{2}
  \sigma_{ij}^{e}
  S_{ijkl}^{e}
  \sigma_{kl}^{e}
  \right),
\end{align}
with $\mathcal{E} = \sum_{e\in\EE}\mathcal{E}_{e}$ being the energy in the
entire domain; see~\eref{eqn:energy}. Since the continuous problem does not
support energy growth, it is desirable that the semi-discrete problem mimic
this property, e.g., $d\mathcal{E}/dt \le 0$. For each element, the energy rate
of change is
\begin{align}
  \label{eqn:elm:eng}
  \fd{\mathcal{E}^{e}}{t}
  &=
  \sum_{m \in \MM^{e}}
  \fd{\mathcal{E}^{m,e}}{t},\\
  \fd{\mathcal{E}^{m,e}}{t} =
  &
  \int_{\hat{m}}
  S_{J}^{m}
  \left(
  v_{i}^{m,e}T_{i}^{*m[e]} + v_{i}^{*m}T_{i}^{m,e} - v_{i}^{m[e]}T_{i}^{m,e}
  \right),
\end{align}
where $d\mathcal{E}^{m,e}/dt$ is the contribution to the energy rate of change
for mortar element $m \in \MM^{e}$ that comes from element $e$,
the mortar projected velocity from element $e$ is $v_{i}^{m,e} =
\mathcal{P}^{m,e}v_{i}^{e}$, and the mortar projected traction vector
$T_{i}^{m,e} = {\left(n_{j} \sigma_{ij}\right)}^{m,e}$ is either the projection
of the tractions to the mortar or the tractions defined from the projected
stresses; see discussion following~\eref{eqn:dg:stress}.

To complete the energy estimate, a single mortar element $m \in \MM$ is
considered and the contributions from all elements $e \in \EE^{m}$
are summed:
\begin{align}
  \fd{\mathcal{E}^{m}}{t}
  =&
  \sum_{e \in \EE^{m}}
  \fd{\mathcal{E}^{m,e}}{t}.
\end{align}
Using the~\eref{eqn:mortar:state:velocity} along with the mortar plus and minus
tractions
\begin{align}
  T_{i}^{\pm m}
  &=
  \mp \sum_{e \in \EE^{\pm m}}
  T_{i}^{m,e},
\end{align}
the mortar element energy rate of change can be written as
\begin{align}
  \label{eqn:mortar:energy}
  \fd{\mathcal{E}^{m}}{t}
  =&
  \int_{\hat{m}}
  S_{J}^{m}
  \left(
  \left(v_{i}^{-m}-v_{i}^{+m}\right)T_{i}^{*m} +
  v_{i}^{*m}\left(T_{i}^{-m}-T_{i}^{+m}\right)\right.\\
  &\notag
  \qquad
  \qquad
  \quad
  \left.- v_{i}^{-m}T_{i}^{-m}
  + v_{i}^{+m}T_{i}^{+m}
  \right).
\end{align}
Here, the traction component of the numerical flux $T^{*m}_{i} = T^{*m[e]}$ if
$e\in\EE^{-m}$ and $T^{*m}_{i} = -T^{*m[e]}$ if $e\in\EE^{+m}$.  For mortar
elements on the physical boundary, the energy rate of change is
\begin{align}
  \label{eqn:mortar:energy:bc}
  \fd{\mathcal{E}^{m}}{t}
  =&
  \int_{\hat{m}}
  S_{J}^{m}
  \left(
  v_{i}^{-m}T_{i}^{*m} +
  v_{i}^{*m}T_{i}^{-m}
  - v_{i}^{-m}T_{i}^{-m}
  \right).
\end{align}

If the numerical flux is defined such that the integrands of~\eref{eqn:mortar:energy} and~\eref{eqn:mortar:energy:bc} are non-positive for
all $v_{i}^{\pm m}$ and $T_{i}^{\pm m}$, then the following theorem results.
\begin{theorem}\label{thm:exact:stable}
  If there exists a numerical flux such that the integrands of
 ~\eref{eqn:mortar:energy} and~\eref{eqn:mortar:energy:bc} are non-positive,
  then discontinuous Galerkin method~\eref{eqn:dg:vel}--\eref{eqn:dg:stress}
  satisfies the energy estimate $\mathcal{E}(t) \le \mathcal{E}(0)$ and is
  energy stable.
\end{theorem}
\begin{proof}
  Taking the derivative of the energy gives,
  \begin{align}
    \fd{\mathcal{E}}{t} = \sum_{e\in\EE} \fd{\mathcal{E}^{e}}{t}
    = \sum_{m \in \MM} \fd{\mathcal{E}^{m}}{t}.
  \end{align}
  If across every face the numerical fluxes have the property that
 ~\eref{eqn:mortar:energy} and~\eref{eqn:mortar:energy:bc} are non-positive it
  follows that
  \begin{align}
    \fd{\mathcal{E}}{t} \le 0,
  \end{align}
  and $\mathcal{E}(t) \le \mathcal{E}(0)$ results upon integration.
  \qed{}
\end{proof}

\subsection{Numerical Flux for Isotropic Elasticity}
The critical question then becomes: can numerical fluxes be defined so
that~\eref{eqn:mortar:energy} and~\eref{eqn:mortar:energy:bc} are non-positive?
For isotropic elasticity, one approach is to use a flux defined as
\begin{align}
  \label{eqn:flux:full}
  T_{i}^{*m} &= n^{m}_{i} T^{*m}_{\parallel} + T^{*m}_{i\bot},&
  v_{i}^{*m} &= n^{m}_{i} v^{*m}_{\parallel} + v^{*m}_{i\bot},
\end{align}
where $T^{*m}_{\parallel}$ and $v^{*m}_{\parallel}$ are the mortar parallel
traction and velocity, and $T^{*m}_{i\bot}$ and $v^{*m}_{i\bot}$ are the
mortar perpendicular components. To define these terms, it is necessary to first
define the parallel and perpendicular plus and minus states:
\begin{align}
  T_{\parallel}^{\pm m} &= n_{i} T_{i}^{\pm m}, &
  T_{i\bot}^{\pm m} &= T_{i}^{\pm m}-n_{i} T_{\parallel}^{\pm m}, \\
  v_{\parallel}^{\pm m} &= n_{i} v_{i}^{\pm m},&
  v_{i\bot}^{\pm m} &= v_{i}^{\pm m}-n_{i} v_{\parallel}^{\pm m}.
\end{align}
With this, the terms in flux~\eref{eqn:flux:full} can be defined as
\begin{align}
  \label{eqn:flux:start}
  T_{\parallel}^{*m} &=
  k_{p}^{m}
  \left(
    Z_{p}^{+m} T_{\parallel}^{-m} + Z_{p}^{-m} T_{\parallel}^{+m} -
    \alpha Z_{p}^{-m}Z_{p}^{+m}\left(v_{\parallel}^{-m} - v_{\parallel}^{+m}\right)
    \right)
  \\
  v_{\parallel}^{* m} &=
  k_{p}^{m}
  \left(
    Z_{p}^{-m} v_{\parallel}^{-m} + Z_{p}^{+m} v_{\parallel}^{+m}
    - \alpha \left(T_{\parallel}^{-m} - T_{\parallel}^{+m}\right)
    \right),
  \\
  T_{i\bot}^{*m} &=
  k_{s}^{m}
  \left(
    Z_{s}^{+m} T_{i\bot}^{-m} + Z_{s}^{-m} T_{i\bot}^{+m} -
    \alpha Z_{s}^{-m}Z_{s}^{+m}\left(v_{i\bot}^{-m} - v_{i\bot}^{+m}\right)
  \right),
  \\
  \label{eqn:flux:end}
  v_{i\bot}^{* m} &=
  k_{s}^{m}
  \left(
    Z_{s}^{-m} v_{i\bot}^{-m} + Z_{s}^{+m} v_{i\bot}^{+m}
    - \alpha \left(T_{i\bot}^{-m} - T_{i\bot}^{+m}\right)
  \right),
\end{align}
with the material properties entering the flux definition through the following
relationships
\begin{align}
  \label{eqn:flux:material}
  Z_{s}^{\pm m} &= \sqrt{\rho^{\pm m} \mu^{\pm m}}, &
  Z_{p}^{\pm m} &= \sqrt{\rho^{\pm m} \left(\lambda^{\pm m} + 2 \mu^{\pm m}\right)},\\
  k_{s}^{m} &= \frac{1}{Z_{s}^{-m} + Z_{s}^{+m}}, &
  k_{p}^{m} &= \frac{1}{Z_{p}^{-m} + Z_{p}^{+m}}.
\end{align}
Here the parameter $\alpha \ge 0$ controls the amount of dissipation that occurs
through the flux, with $\alpha = 1$ being the upwind
flux~\cite{WilcoxStadlerBursteddeEtAl2010} and $\alpha=0$ being a central-like
flux which results in no-energy dissipation across the interface. To enforce the
physical boundary condition $T_{i} = 0$ we set $v_{i}^{+m} = v_{i}^{-m}$,
$T_{i}^{+m} = -T_{i}^{-m}$, $Z_{s}^{+m} = Z_{s}^{-m}$, and $Z_{p}^{+m} =
Z_{p}^{-m}$ which leads to
\begin{align}
  T_{\parallel}^{*m} &= 0,
  &
  v_{\parallel}^{* m} &= v_{\parallel}^{-m} - \alpha \frac{T_{\parallel}^{-m}}{ Z_{p}^{-m}},
  &
  T_{i\bot}^{*m} &= 0,
  &
  v_{i\bot}^{* m} &= v_{i\bot}^{-m} - \alpha \frac{T_{i\bot}^{-m}}{Z_{s}^{-m}}.
\end{align}

To see that~\eref{eqn:flux:full} results in a stable flux, first consider the
interior mortar rate of energy change integral~\eref{eqn:mortar:energy}.
Rewriting the integrand in terms of the parallel and perpendicular components
gives
\begin{align}
  \notag
  &\left(v_{i}^{-m}-v_{i}^{+m}\right)T_{i}^{*m} +
  v_{i}^{*m}\left(T_{i}^{-m}-T_{i}^{+m}\right)
  - v_{i}^{-m}T_{i}^{-m}
  + v_{i}^{+m}T_{i}^{+m}\\
  &\qquad=
  \left(v_{\parallel}^{-m}-v_{\parallel}^{+m}\right)T_{\parallel}^{*m} +
  v_{\parallel}^{*m}\left(T_{\parallel}^{-m}-T_{\parallel}^{+m}\right)
  - v_{\parallel}^{-m}T_{\parallel}^{-m}
  + v_{\parallel}^{+m}T_{\parallel}^{+m}\\
  \notag
  &\qquad\quad+
  \left(v_{i\bot}^{-m}-v_{i\bot}^{+m}\right)T_{i\bot}^{*m} +
  v_{i\bot}^{*m}\left(T_{i\bot}^{-m}-T_{i\bot}^{+m}\right)
  - v_{i\bot}^{-m}T_{i\bot}^{-m}
  + v_{i\bot}^{+m}T_{i\bot}^{+m},
\end{align}
where it has been used that, by definition, $v_{i\bot}^{\pm m}$, $T_{i\bot}^{\pm
m}$, $v_{i\bot}^{*m}$, and $T_{i\bot}^{*m}$ are orthogonal to $n_{i}^{m}$.
Substituting in the numerical flux expressions~\eref{eqn:flux:start}--\eref{eqn:flux:end} gives
\begin{align}
  \notag
  &\left(v_{i}^{-m}-v_{i}^{+m}\right)T_{i}^{*m} +
  v_{i}^{*m}\left(T_{i}^{-m}-T_{i}^{+m}\right)
  - v_{i}^{-m}T_{i}^{-m}
  + v_{i}^{+m}T_{i}^{+m}\\
  &\qquad=
  -\alpha \; k_{p}^{m}
  \left(
  Z_{p}^{-m} Z_{p}^{+m}
  {\left(v_{\parallel}^{-m}-v_{\parallel}^{+m}\right)}^{2}
  +
  {\left(T_{\parallel}^{-m}-T_{\parallel}^{+m}\right)}^{2}\right)\\
  \notag
  &\qquad\quad-
  \sum_{i=1}^{3}
  \alpha \; k_{s}^{m}
  \left(
  Z_{s}^{-m} Z_{s}^{+m}
  {\left(v_{i\bot}^{-m}-v_{i\bot}^{+m}\right)}^{2}
  +
  {\left(T_{i\bot}^{-m}-T_{i\bot}^{+m}\right)}^{2}\right).
\end{align}
A similar calculation for the boundary mortar elements, gives that the integrand
of~\eref{eqn:mortar:energy:bc} is
\begin{align}
  v_{i}^{-m}T_{i}^{*m} + v_{i}^{*m}T_{i}^{-m} - v_{i}^{-m}T_{i}^{-m}
  =
  -\alpha \; \frac{{\left(T^{-m}_{\parallel}\right)}^{2}}{Z_{p}^{-m}}
  -\sum_{i=1}^{3}\alpha \; \frac{{\left(T^{-m}_{i\bot}\right)}^{2}}{Z_{s}^{-m}}.
\end{align}
Thus for both~\eref{eqn:mortar:energy} and~\eref{eqn:mortar:energy:bc} the
integrand is non-positive if $\alpha \ge 0$, and the flux yields a stable
scheme by Theorem~\ref{thm:exact:stable}.

\begin{corollary}
  For a variable coefficient, isotropic material the skew-symmetric scheme~\eref{eqn:dg:vel}--\eref{eqn:dg:stress} using numerical flux~\eref{eqn:flux:full} with~\eref{eqn:flux:start}--\eref{eqn:flux:end} leads to
  a semi-discrete scheme that satisfies $\mathcal{E}(t) \le \mathcal{E}(0)$.
\end{corollary}

\subsection{Quadrature Integration}\label{sec:quad}
We now turn to the case when quadrature is used to evaluate the integrals in the~\eref{eqn:dg:vel}--\eref{eqn:dg:stress}, and show that this has semi-discrete
stability even with some variational crimes. In the analysis that follows, we
make minimal assumptions about the quadrature rules and interpolation
procedures.  The analysis is independent of element shape.

A quadrature based version of~\eref{eqn:dg:vel}--\eref{eqn:dg:stress} is:
For each $e \in \EE$, find $\mathrm{v}_{j}^{e} \in \hat{\VV}^{e}$ and
symmetric $\sigma_{ij}^{e} = \sigma_{ji}^{e} \in \hat{\VV}^{e}$ such that:
\begin{align}
  \label{eqn:dg:vel:quad}
  \mat{M}_{\rho}^{e} \fd{\vec{v}_{i}^{e}}{t}
  =&
  -
  {\left(\mat{S}_{j}^{e}\right)}^{T}\vec{\sigma}_{ij}^{e}
  +
  \sum_{m \in \MM^{e}}
  {\left(\mat{P}^{m,e}\right)}^{T}\mat{W}^{m}\vec{T}_{i}^{*m[e]},\\
  \label{eqn:dg:stress:quad}
  \mat{M}^{e}_{S_{ijkl}} \fd{\vec{\sigma}_{kl}^{e}}{t}
  =&
  \frac{1}{2} \left(\mat{S}_{j}^{e}\vec{v}_{i}^{e} + \mat{S}_{i}^{e}\vec{v}_{j}^{e}\right)
  +
  \sum_{m \in \MM^{e}}
  {\left(\mat{P}^{m,e}_{n_{j}}\right)}^{T} \mat{W}^{m}
  \left(\vec{v}_{i}^{*m} - \vec{v}_{i}^{m[e]}\right).
\end{align}
Here, the subscript $n_{j}$ in $\mat{P}^{m,e}_{n_{j}}$ denotes the fact that
this projection operator could be defined such that the unit normal $n_{j}$ is
multiplied before or after the projection, that is
$\mat{P}^{m,e}_{n_{j}}\vec{\sigma}_{ij}^{e}$ approximates
${\left(n_{j}\sigma_{ij}\right)}^{m,e}$ at the quadrature nodes; see discussion
following~\eref{eqn:dg:vel}--\eref{eqn:dg:stress}.
The trial velocity vector on the mortar depends on which side of mortar $m$
element $e$ is on. Namely,  $\vec{v}_{i}^{m[e]} = \vec{v}_{i}^{+m}$ if $e \in
\EE^{+m}$ and $\vec{v}_{i}^{m[e]} = \vec{v}_{i}^{-m}$ if $e \in \EE^{-m}$ with
\begin{align}
  \label{eqn:mortar:state:velocity:discrete}
  \vec{v}_{i}^{\pm m}
  &=
  \sum_{e \in \EE^{\pm m}}
  \mat{P}^{m,e} \vec{v}_{i}^{e}.
\end{align}
The remaining notation in~\eref{eqn:dg:vel:quad}--\eref{eqn:dg:stress:quad} is discussed in
\sref{sec:notation:quad}.
We call this semi-discrete scheme the \emph{symmetric flux integral method}
(SFIM) as both the trial and test functions are projected to the mortar for
integration, and note that SFIM can be applied to either the full-side or
split-side mortars as shown in \fref{fig:mortars}.

The energy in element $e \in \EE$ is defined as
\begin{align}
  \label{eqn:energy:quad}
  \mathcal{E}^{e} &=
  \frac{1}{2}
  {\left(\vec{v}_{i}^{e}\right)}^{T}
  \mat{M}^{e}_{\rho}
  \vec{v}_{i}^{e}
  +
  \frac{1}{2}
  {\left(\vec{\sigma}_{ij}^{e}\right)}^{T}
  \mat{M}_{S_{ijkl}}^{e}
  \vec{\sigma}_{kl}^{e},
\end{align}
with the total energy in the system defined as $\mathcal{E} = \sum_{e\in\EE}
\mathcal{E}^{e}$. For this to be a well-defined norm, it is required that
$\mat{M}^{e}_{\rho}$ and $\mat{\bar{M}}^{e}_{S}$, see~\eref{eqn:mass:stress},
be symmetric positive definite.

The time derivative of the energy gives
\begin{align}
  \fd{\mathcal{E}^{e}}{t} &=
  \sum_{m \in \MM^{e}}
  \fd{\mathcal{E}^{m,e}}{t},\\
  \fd{\mathcal{E}^{m,e}}{t} &=
  {\left(\vec{v}^{m,e}_{i}\right)}^{T} \mat{W}^{m}\vec{T}_{i}^{*m[e]}
  +
  {\left(\vec{v}_{i}^{*m}\right)}^{T}
  \mat{W}^{m}
  \vec{T}^{m,e}_{i}
  -
  {\left(\vec{v}_{i}^{m[e]}\right)}^{T}
  \mat{W}^{m}
  \vec{T}^{m,e}_{i},
\end{align}
where
$\vec{v}_{i}^{m,e} = \mat{P}^{m,e}\vec{v}_{i}^{e}$
and
$\vec{T}^{m,e}_{i} = \mat{P}^{m,e}_{n_{j}}\vec{\sigma}_{ij}^{e}$.
Considering only a single mortar element $m \in \MM$ and summing contributions
from all $e \in \EE^{m}$ gives
\begin{align}
  \fd{\mathcal{E}^{m}}{t} &=
  \sum_{e \in \EE^{m}}
  \fd{\mathcal{E}^{m,e}}{t},
\end{align}
which for an interior mortar element is
\begin{align}
  \label{eqn:mortar:energy:quad}
  \fd{\mathcal{E}^{m}}{t} =\;&
  {\left(\vec{v}^{-m}_{i}-\vec{v}^{+m}\right)}^{T} \mat{W}^{m}\vec{T}_{i}^{*m}
  +
  {\left(\vec{v}_{i}^{*m}\right)}^{T}
  \mat{W}^{m}
  {\left(\vec{T}^{-m}_{i}-\vec{T}^{+m}_{i}\right)}
  \\
  \notag
  &-
  {\left(\vec{v}_{i}^{-m}\right)}^{T}
  \mat{W}^{m}
  \vec{T}^{-m}_{i}
  +
  {\left(\vec{v}_{i}^{+m}\right)}^{T}
  \mat{W}^{m}
  \vec{T}^{+m}_{i}
\end{align}
and for a boundary mortar element is
\begin{align}
  \label{eqn:mortar:energy:bc:quad}
  \fd{\mathcal{E}^{m}}{t} &=
  {\left(\vec{v}^{-m}_{i}\right)}^{T} \mat{W}^{m}\vec{T}_{i}^{*m}
  +
  {\left(\vec{v}_{i}^{*m}\right)}^{T}
  \mat{W}^{m}
  \vec{T}^{-m}_{i}
  -
  {\left(\vec{v}_{i}^{-m}\right)}^{T}
  \mat{W}^{m}
  \vec{T}^{-m}_{i}.
\end{align}
Here the plus and minus side traction states on the mortar are defined as
\begin{align}
  \vec{T}_{i}^{\pm m}
  &=
  \mp \sum_{e \in \EE^{\pm m}}
  \vec{T}_{i}^{m,e};
\end{align}
the velocity state on the mortar is defined in~\eref{eqn:mortar:state:velocity:discrete}.
The traction component of the numerical flux is defined as
$\mat{T}_{i}^{*m} = \mat{T}_{i}^{*m[e]}$ if $e \in \EE^{-m}$ and
$\mat{T}_{i}^{*m} = -\mat{T}_{i}^{*m[e]}$ if $e \in \EE^{+m}$.

\begin{theorem}\label{thm:quad:stable}
  Given a numerical flux such that the integrands of the energy rate~\eref{eqn:mortar:energy} and~\eref{eqn:mortar:energy:bc} are non-positive,
  the energy rates~\eref{eqn:mortar:energy:quad} and~\eref{eqn:mortar:energy:bc:quad} are non-positive and the quadrature-based
  discontinuous Galerkin method~\eref{eqn:dg:vel:quad}--\eref{eqn:dg:stress:quad} satisfies the energy
  estimate $\mathcal{E}(t) \le \mathcal{E}(0)$ and is energy stable.
\end{theorem}
\begin{proof}
  Recall that $\mat{W}^{m}$ is a diagonal matrix of quadrature weights and
  surface Jacobian determinants, thus~\eref{eqn:mortar:energy:quad} can be written as
  \begin{align}
  \fd{\mathcal{E}^{m}}{t}
  =&\;
    \sum_{k=1}^{n_{q}^{m}}
    \omega^{m}_{k}\;
    \Big\{
    S_{J}^{m}
    \left(
    \left(v_{i}^{-m}-v_{i}^{+m}\right)T_{i}^{*m} +
    v_{i}^{*m}\left(T_{i}^{-m}-T_{i}^{+m}\right)
    \right.\\
    &\notag
    \qquad
    \qquad
    \quad
    \left.
    - v_{i}^{-m}T_{i}^{-m}
    + v_{i}^{+m}T_{i}^{+m}
    \right){\Big\}}_{k},
  \end{align}
  and~\eref{eqn:mortar:energy:bc:quad} as
  \begin{align}
  \fd{\mathcal{E}^{m}}{t}
  =&\;
    \sum_{k=1}^{n_{q}^{m}}
    \omega^{m}_{k}\;
    \Big\{
    S_{J}^{m}
    \left(
    v_{i}^{-m}T_{i}^{*m} +
    v_{i}^{*m}T_{i}^{-m}
    - v_{i}^{-m}T_{i}^{-m}
    \right){\Big\}}_{k},
  \end{align}
  where ${\left\{\cdot\right\}}_{k}$ denotes that the term inside the brackets
  is evaluated at mortar quadrature node $k$. Since the terms inside the
  brackets are the same as the integrands of~\eref{eqn:mortar:energy} and~\eref{eqn:mortar:energy:bc}, a numerical flux that results in non-positive
  integrands for~\eref{eqn:mortar:energy} and~\eref{eqn:mortar:energy:bc} will
  result in~\eref{eqn:mortar:energy:quad} and~\eref{eqn:mortar:energy:bc:quad}
  being non-positive. The remainder of the proof is identical to the proof of
  Theorem~\ref{thm:exact:stable}.
  \qed{}
\end{proof}

\subsection{Conservation and Constant Preservation}\label{sec:constant}
Since the diffeomorphic mappings between the physical and reference elements
(see \sref{sec:mesh}) were defined prior to the introduction of an
approximation space, gaps and overlaps in the mesh may occur when the mappings
are approximated. For an isoparametric geometry
approximation, one approach is to interpolate the geometry transform at
nodal degrees of freedom. Though this approach is straightforward to implement,
differences in the interpolations of the mappings across nonconforming
interfaces can result in differing approximations; across conforming faces this
problem does not arise because the aliasing errors are the same on both sides of
a face. By construction, the stability of SFIM is not impacted by a discretely
discontinuous mesh (since the volume and surface stability are decoupled), but a
discretely discontinuous mesh can lead to a lack of conservation and breakdown
of constant preserving.

Here, we show how SFIM can be made conservative and constant preservation by
imposing a set of accuracy and consistency constraints. Most of the constraints
are quite natural (such as the ability to exactly approximate and differentiate
constants). One of the constraints requires that the operators satisfy a
discrete divergence theorem and this implicitly implies some continuity of the
metric terms. For tensor product hexahedral elements, with
Legendre-Gauss-Lobatto quadrature we satisfy these constraints by making the
mesh \emph{discretely watertight} (i.e., removing holes and overlaps) as well as
ensuring that aliasing errors in certain metric relations are the same across
nonconforming faces; details are given in \aref{app:div}.

In the following, we assume that the domain is periodic so that the impact of the
boundary conditions can be ignored.

The following constraints are imposed in order to make SFIM conservative and
constant preserving:
\begin{constraint}[Approximation Consistency]\label{con:app:con}
  We assume that the volume and mortar approximation spaces can exactly
  represent constants. The notation $\vec{1}^{e}$ and $\vec{1}^{m}$ are used to
  denote the expansion of $1$ within the chosen basis for a volume element $e
  \in \EE$ and mortar element $m \in \MM$, respectively; for nodal basis
  functions these would be vectors of ones. We also define $\vec{0}^{e} =
  0\vec{1}^{e}$ and $\vec{0}^{m} = 0 \vec{1}^{m}$.
\end{constraint}
\begin{constraint}[Projection Consistency]\label{con:proj:con}
  It is assumed that the discrete projection operators $\mat{P}^{m,e}$ exactly
  project constants from the volume to the mortar. Namely, we assume that for
  each $m \in \MM$ the following holds
  \begin{align}
    \sum_{e\in\EE^{- m}} \mat{P}^{m,e} \vec{1}^{e} =
    \sum_{e\in\EE^{+ m}} \mat{P}^{m,e} \vec{1}^{e} &= \vec{1}^{m}.
  \end{align}
\end{constraint}
\begin{constraint}[Stiffness Consistency]\label{con:stiff:con}
  The stiffness matrices $\mat{S}_{j}^{e}$ are assumed to differentiate
  constants exactly. Namely it is assumed that
  \begin{align}
    \mat{S}_{j}^{e} \vec{1}^{e} = \vec{0}^{e}.
  \end{align}
\end{constraint}
\begin{constraint}[Consistent Numerical Flux]\label{con:flux:con}
  The numerical flux is assumed to be consistent in the sense that if
  $v_{i}^{+m} = v_{i}^{-m}$ and $T_{i}^{+m} = T_{i}^{-m}$ for $i = 1,2,3$ then
  $v_{i}^{*m} = v_{i}^{\pm m}$ and $T_{i}^{*m} = T_{i}^{\pm m}$.
\end{constraint}
\begin{constraint}[Discrete Divergence Theorem]\label{con:div}
  We assume that the stiffness matrices satisfy the following discrete
  divergence theorem:
  \begin{align}
    \label{eqn:disc:div}
    {\left(\mat{S}_{j}^{e}\right)}^{T} \vec{1}^{e} &=
    \sum_{m\in\MM^{e}}{\left(\mat{P}^{m,e}\right)}^{T}
    \mat{W}^{m}\vec{n}_{j}^{m[e]},
  \end{align}
  where $\vec{n}_{j}^{m[e]} = \vec{n}_{j}^{m}$ if $e \in \EE^{-m}$ and
  $\vec{n}_{j}^{m[e]} = -\vec{n}_{j}^{m}$ if $e \in \EE^{+m}$ with
  $\vec{n}_{j}^{m}$ being the components of the normal vector at the mortar
  degrees of freedom. We call this a discrete divergence theorem because, after
  multiplication by ${\left(\vec{q}^{e}\right)}^{T}$,~\eref{eqn:disc:div} approximates
  \begin{align}
    \int_{e}J^{e} \pd{q^{e}}{x_{j}} = \sum_{m\in\MM^{e}} \int_{m}S_{J}^{m} n_{j}^{e}
    \mathcal{P}^{m,e} q^{e}
    =
    \int_{\partial e} S_{J}^{e} n_{j}^{e} q^{e},
  \end{align}
  where the last equality assumes that $\mathcal{P}^{m,e}$ is an exact L$^{2}$
  projection.
\end{constraint}
\begin{constraint}[Consistent Constant Traction
  Calculation]\label{con:trac:con} If for a mortar element $m \in \MM$ the
  projected stresses on neighboring volume elements are constant, i.e., for some
  $\sigma_{ij} \in \mathbb{R}$ the element stresses are $\vec{\sigma}_{ij}^{e} =
  \sigma_{ij} \vec{1}^{e}$ for all $e \in \EE^{m}$, then we assume that
  $T^{-m}_{i} = T^{+m}_{i} = \sigma_{ij} \vec{n}_{j}^{m}$; namely
  that the traction on the mortar elements are the $\sigma_{ij}$ weighted sum of
  the normal vectors on the mortar. Recall that following~\eref{eqn:dg:stress}
  we gave two approaches to computing the mortar tractions, either projecting
  the stresses with traction calculation on the mortar or projecting the volume
  face computed tractions to the mortar elements. The first approach, projecting
  the stresses, automatically satisfies this constraint by
  Constraint~\ref{con:proj:con}. If the second approach, projecting the
  tractions, is used then additional assumptions on the normal vectors as
  calculated on the faces of the volume elements as well as the accuracy of the
  projection operators for high-order functions would be needed. In the
  results section we use the approach of projecting the stresses and this
  constraint is satisfied by construction.
\end{constraint}

Constraints~\ref{con:app:con}--\ref{con:stiff:con} are reasonable accuracy and
consistency assumptions for most approximations spaces, and
Constraint~\ref{con:flux:con} is fairly standard for discontinuous Galerkin
methods. As noted above, Constraint~\ref{con:trac:con} holds for our choice of
projecting the tractions. Thus, it only remains to show that
Constraint~\ref{con:div} holds. One of the key features in satisfying
Constraint~\ref{con:div} is the computation of the metric terms. For conforming
meshes, the discrete divergence theorem can be ensured using the curl invariant
form of Kopriva~\cite{Kopriva2006jsc}. In \aref{app:div} we show how the
curl invariant form can be adapted for nonconforming meshes with tensor product
hexahedral elements with Legendre-Gauss-Lobatto quadrature.

For periodic domains, one of the steady state solutions admitted by governing
equations~\eref{eqn:continuous:pde}, regardless of the variability of the
material properties, is constant velocities $v_{i}$ and stresses $\sigma_{ij}$.
By Constraint~\ref{con:app:con}, the constant solution is exactly admissible by
the approximation, and thus it may be desirable to require that a discretization
of~\eref{eqn:continuous:pde} represent this steady state solution with no error,
i.e., that the scheme be constant preserving. It is known that even on
conforming meshes, when the elements are non-affine constant preserving requires
careful handling of the metric terms; see for instance the work of
Kopriva~\cite{Kopriva2006jsc}.

To see that when the above assumptions are satisfied that SFIM is constant
preserving, we substitute $v_{i}^{e} = \beta_{i}$ and $\sigma_{ij} = \alpha_{ij}$ into
the right-hand side~\eref{eqn:dg:vel:quad}--\eref{eqn:dg:stress:quad} with
$\beta_{i}$ and $\alpha_{ij}$ being constants:
\begin{align}
  \label{eqn:dg:vel:constant}
  \mat{M}_{\rho}^{e} \fd{\vec{v}_{i}^{e}}{t}
  =&
  -
  \alpha_{ij} {\left(\mat{S}_{j}^{e}\right)}^{T}\vec{1}^{e}
  +
  \sum_{m \in \MM^{e}}
  {\left(\mat{P}^{m,e}\right)}^{T}\mat{W}^{m}\vec{T}_{i}^{*m[e]},\\
  \label{eqn:dg:stress:constant}
  \mat{M}^{e}_{S_{ijkl}} \fd{\vec{\sigma}_{kl}^{e}}{t}
  =&
  \frac{1}{2} \left(\beta_{i} \mat{S}_{j}^{e}\vec{1}^{e} +
  \beta_{j} \mat{S}_{i}^{e}\vec{1}^{e}\right)
  +
  \sum_{m \in \MM^{e}}
  {\left(\mat{P}^{m,e}_{n_{j}}\right)}^{T} \mat{W}^{m}
  \left(\vec{\beta}_{i}^{*m} - \vec{\beta}_{i}^{m[e]}\right).
\end{align}
Using Constraints~\ref{con:stiff:con} (stiffness consistency)
and~\ref{con:div} (discrete divergence theorem) relations~\eref{eqn:dg:vel:constant}--\eref{eqn:dg:stress:constant} can be rewritten as
\begin{align}
  \label{eqn:dg:vel:constant2}
  \mat{M}_{\rho}^{e} \fd{\vec{v}_{i}^{e}}{t}
  =&
  \sum_{m \in \MM^{e}}
  {\left(\mat{P}^{m,e}\right)}^{T}\mat{W}^{m}
  \left(\vec{T}_{i}^{*m[e]}
  -
  \alpha_{ij}
  \vec{n}_{j}^{m[e]}\right),\\
  \label{eqn:dg:stress:constant2}
  \mat{M}^{e}_{S_{ijkl}} \fd{\vec{\sigma}_{kl}^{e}}{t}
  =&
  \sum_{m \in \MM^{e}}
  {\left(\mat{P}^{m,e}_{n_{j}}\right)}^{T} \mat{W}^{m}
  \left(\vec{\beta}_{i}^{*m} - \vec{\beta}_{i}^{m[e]}\right).
\end{align}
Projection consistency (Constraint~\ref{con:proj:con}) implies that
\begin{align}
  \vec{\beta}_{i}^{m[e]} = \vec{\beta}_{i}^{m[e']}\quad\forall e,e'\in\EE^{m},
\end{align}
and this along with flux consistency (Constraint~\ref{con:flux:con}) gives
\begin{align}
  \vec{\beta}_{i}^{*m} &= \vec{\beta}_{i}^{m[e]},&
  \vec{T}_{i}^{*m} &= \alpha_{ij}\vec{n}_{j}^{m[e]}.
\end{align}
It then follows that~\eref{eqn:dg:stress:constant2} becomes
\begin{align}
  \mat{M}_{\rho}^{e} \fd{\vec{v}_{i}^{e}}{t}
  &=
  \vec{0},&
  \mat{M}^{e}_{S_{ijkl}} \fd{\vec{\sigma}_{kl}^{e}}{t}
  &= \vec{0},
\end{align}
and the solution remains constant in time.

For linear elasticity, the conserved quantities are the components of momentum
$\rho v_{i}$ and strain $\epsilon_{ij} = S_{ijkl} \sigma_{kl}$. To show discrete
conservation we need to show that
\begin{align}
  \sum_{e\in\EE} {\left(\vec{1}^{e}\right)}^{T}\mat{M}_{\rho}^{e} \fd{\vec{v}_{i}^{e}}{t}
  &=
  0,&
  \sum_{e\in\EE} {\left(\vec{1}^{e}\right)}^{T} \mat{M}^{e}_{S_{ijkl}} \fd{\vec{\sigma}_{kl}^{e}}{t}
  &=
  0.
\end{align}
Multiplying~\eref{eqn:dg:vel:quad}--\eref{eqn:dg:stress:quad} on the left by
${\left(\vec{1}^{e}\right)}^{T}$ we have
\begin{align}
  {\left(\vec{1}^{e}\right)}^{T}\mat{M}_{\rho}^{e} \fd{\vec{v}_{i}^{e}}{t}
  =&
  -
  {\left(\vec{1}^{e}\right)}^{T}{\left(\mat{S}_{j}^{e}\right)}^{T}\vec{\sigma}_{ij}^{e}
  +
  \sum_{m \in \MM^{e}}
  {\left(\vec{1}^{m}\right)}^{T}\mat{W}^{m}\vec{T}_{i}^{*m[e]},\\
  {\left(\vec{1}^{e}\right)}^{T}\mat{M}^{e}_{S_{ijkl}} \fd{\vec{\sigma}_{kl}^{e}}{t}
  =&
  \frac{{\left(\vec{1}^{e}\right)}^{T}}{2} \left(\mat{S}_{j}^{e}\vec{v}_{i}^{e} + \mat{S}_{i}^{e}\vec{v}_{j}^{e}\right)
  +
  \sum_{m \in \MM^{e}}
  {\left(\vec{n}_{j}^{e}\right)}^{T} \mat{W}^{m}
  \left(\vec{v}_{i}^{*m} - \vec{v}_{i}^{m[e]}\right).
\end{align}
Using the stiffness consistency and discrete divergence assumptions
(Constraints~\ref{con:stiff:con} and~\ref{con:div}) these become
\begin{align}
  \label{eqn:con:vel}
  {\left(\vec{1}^{e}\right)}^{T}\mat{M}_{\rho}^{e} \fd{\vec{v}_{i}^{e}}{t}
  =&
  \sum_{m \in \MM^{e}}
  {\left(\vec{1}^{e}\right)}^{T}{\left(\mat{P}^{m,e}\right)}^{T}\mat{W}^{m}\vec{T}_{i}^{*m[e]},\\
  \label{eqn:con:strain}
  {\left(\vec{1}^{e}\right)}^{T}\mat{M}^{e}_{S_{ijkl}} \fd{\vec{\sigma}_{kl}^{e}}{t}
  =&
  \sum_{m \in \MM^{e}}
  {\left(\vec{1}^{e}\right)}^{T}{\left(\mat{P}^{m,e}_{n_{j}}\right)}^{T} \mat{W}^{m}
  \vec{v}_{i}^{*m}.
\end{align}
Considering~\eref{eqn:con:vel} and summing over all the elements gives
\begin{align}
  \notag
  \sum_{e\in\EE}
  {\left(\vec{1}^{e}\right)}^{T}\mat{M}_{\rho}^{e} \fd{\vec{v}_{i}^{e}}{t}
  &=
  \sum_{e\in\EE}
  \sum_{m \in \MM^{e}}
  {\left(\vec{1}^{e}\right)}^{T}{\left(\mat{P}^{m,e}\right)}^{T}\mat{W}^{m}\vec{T}_{i}^{*m[e]}\\
  \notag
  &=
  \sum_{m \in \MM}
  \sum_{e\in\MM^{e}}
  {\left(\vec{1}^{e}\right)}^{T}{\left(\mat{P}^{m,e}\right)}^{T}\mat{W}^{m}\vec{T}_{i}^{*m[e]}\\
  \label{eqn:con:vel2}
  &=
  \sum_{m \in \MM}
  \sum_{e\in\MM^{+e}}
  {\left(\vec{1}^{e}\right)}^{T}{\left(\mat{P}^{m,e}\right)}^{T}\mat{W}^{m}\vec{T}_{i}^{*m[e]}\\
  \notag
  &\phantom{=}+
  \sum_{m \in \MM}
  \sum_{e\in\MM^{-e}}
  {\left(\vec{1}^{e}\right)}^{T}{\left(\mat{P}^{m,e}\right)}^{T}\mat{W}^{m}\vec{T}_{i}^{*m[e]}.
\end{align}
Applying projection consistency (Constraint~\ref{con:proj:con}) along with the
fact that $T_{i}^{*m} = T_{i}^{*m[e]}$ for $e\in\EE^{-m}$ and $T_{i}^{*m} =
-T_{i}^{*m[e]}$ for $e\in\EE^{+m}$ gives
\begin{align}
  \sum_{e\in\MM^{\pm e}}
  {\left(\vec{1}^{e}\right)}^{T}{\left(\mat{P}^{m,e}\right)}^{T}\mat{W}^{m}\vec{T}_{i}^{*m[e]}
  =
  \mp
  {\left(\vec{1}^{m}\right)}^{T}\mat{W}^{m}\vec{T}_{i}^{*m},
\end{align}
and thus~\eref{eqn:con:vel2} becomes
\begin{align}
  \notag
  \sum_{e\in\EE}
  {\left(\vec{1}^{e}\right)}^{T}\mat{M}_{\rho}^{e} \fd{\vec{v}_{i}^{e}}{t}
  &=
  \sum_{m \in \MM}
  \left(
  -{\left(\vec{1}^{m}\right)}^{T}\mat{W}^{m}\vec{T}_{i}^{*m}
  +{\left(\vec{1}^{m}\right)}^{T}\mat{W}^{m}\vec{T}_{i}^{*m}\right)
  =
  0.
\end{align}
A similar calculation for~\eref{eqn:con:strain} results in
\begin{align}
  \label{eqn:con:strain2}
  \sum_{e \in \EE}
  {\left(\vec{1}^{e}\right)}^{T}\mat{M}^{e}_{S_{ijkl}} \fd{\vec{\sigma}_{kl}^{e}}{t}
  =&
  \sum_{m \in \MM}
  \sum_{e \in \EE^{+m}}
  {\left(\vec{1}^{e}\right)}^{T}{\left(\mat{P}^{m,e}_{n_{j}}\right)}^{T} \mat{W}^{m}
  \vec{v}_{i}^{*m}\\\notag
  &\phantom{=}+
  \sum_{m \in \MM}
  \sum_{e \in \EE^{+m}}
  {\left(\vec{1}^{e}\right)}^{T}{\left(\mat{P}^{m,e}_{n_{j}}\right)}^{T} \mat{W}^{m}
  \vec{v}_{i}^{*m}.
\end{align}
Which, after applying the traction consistency constraint
(Constraint~\ref{con:trac:con}) leads to~\eref{eqn:con:strain2} becoming
\begin{align}
  \sum_{e \in \EE}
  {\left(\vec{1}^{e}\right)}^{T}\mat{M}^{e}_{S_{ijkl}} \fd{\vec{\sigma}_{kl}^{e}}{t}
  =&
  \sum_{m \in \MM}
  \left(
  -{\left(\vec{n}_{j}^{e}\right)}^{T}\mat{W}^{m} \vec{v}_{i}^{*m}
  +{\left(\vec{n}_{j}^{e}\right)}^{T}\mat{W}^{m} \vec{v}_{i}^{*m}
  \right)
  =
  0.
\end{align}
Thus the scheme satisfies discrete conservation.

The discrete divergence theorem (Constraint~\ref{con:div}) and traction
consistency (Constraint~\ref{con:trac:con}) are only required for
conservation of strain, and if only conservation of momentum is required a
scheme without these properties can be used. If the converse was desired
(automatic conservation of strain) then the weak and strong derivatives in
formulation~\eref{eqn:dg:vel}--\eref{eqn:dg:stress} should be swapped.
\subsection{Comment on the Implementation of~\eref{eqn:dg:stress:quad}}
Due to the implied summation on the right-hand side~\eref{eqn:dg:stress:quad}
the mass matrix to be inverted is not $\mat{M}^{e}_{S_{ijkl}}$ but the combined
matrix $\mat{\bar{M}}^{e}_{S}$ defined in~\eref{eqn:mass:stress}.
In the results section, we will be using tensor product hexahedral elements with
Legendre-Gauss-Lobatto quadrature. With this, mass matrix
$\mat{M}^{e}_{S_{ijkl}} = \mat{M}^{e} \mat{S}_{ijkl}^{e}$ is diagonal with
$\mat{M}^{e}$ being a diagonal matrix of quadrature weights and Jacobian
determinants and $\mat{S}_{ijkl}^{e}$ being a diagonal matrix of the compliance
tensor evaluated at the quadrature nodes (not to be confused with the stiffness
matrix $\mat{S}_{i}^{e}$). In this case, the combined mass matrix is
\begin{align}
  \mat{\bar{M}}^{e}_{S}
  &=
  \mat{\bar{S}}^{e} \left(\mat{I}_{6\times 6} \otimes \mat{M}^{e}\right) =
  \left(\mat{I}_{6\times 6} \otimes \mat{M}^{e}\right) \mat{\bar{S}}^{e},\\
  \mat{\bar{S}}^{e}
  &=
  \begin{bmatrix}
      \mat{S}^{e}_{1111} &   \mat{S}^{e}_{1122} &   \mat{S}^{e}_{1133} & 2 \mat{S}^{e}_{1123} & 2 \mat{S}^{e}_{1113} & 2 \mat{S}^{e}_{1112} \\
      \mat{S}^{e}_{1122} &   \mat{S}^{e}_{2222} &   \mat{S}^{e}_{2233} & 2 \mat{S}^{e}_{2223} & 2 \mat{S}^{e}_{2213} & 2 \mat{S}^{e}_{2212} \\
      \mat{S}^{e}_{1133} &   \mat{S}^{e}_{2233} &   \mat{S}^{e}_{3333} & 2 \mat{S}^{e}_{3323} & 2 \mat{S}^{e}_{3313} & 2 \mat{S}^{e}_{3312} \\
    2 \mat{S}^{e}_{1123} & 2 \mat{S}^{e}_{2223} & 2 \mat{S}^{e}_{3323} & 4 \mat{S}^{e}_{2323} & 4 \mat{S}^{e}_{2313} & 4 \mat{S}^{e}_{2312} \\
    2 \mat{S}^{e}_{1113} & 2 \mat{S}^{e}_{2213} & 2 \mat{S}^{e}_{3313} & 4 \mat{S}^{e}_{2313} & 4 \mat{S}^{e}_{1313} & 4 \mat{S}^{e}_{1312} \\
    2 \mat{S}^{e}_{1112} & 2 \mat{S}^{e}_{2212} & 2 \mat{S}^{e}_{3312} & 4 \mat{S}^{e}_{2312} & 4 \mat{S}^{e}_{1312} & 4 \mat{S}^{e}_{1212}
  \end{bmatrix},
\end{align}
Additionally, the inverse of the combined mass matrix is
\begin{align}
  {\left(\mat{\bar{M}}^{e}_{S}\right)}^{-1}
  &=
  \mat{\bar{M}}^{e}_{C}
  =
  \mat{\bar{C}}^{e} \left(\mat{I}_{6\times 6} \otimes {\left(\mat{M}^{e}\right)}^{-1}\right)
  =
  \left(\mat{I}_{6\times 6} \otimes {\left(\mat{M}^{e}\right)}^{-1}\right) \mat{\bar{C}}^{e},\\
  \mat{\bar{C}}^{e}
  &=
  \begin{bmatrix}
    \mat{C}^{e}_{1111} & \mat{C}^{e}_{1122} & \mat{C}^{e}_{1133} & \mat{C}^{e}_{1123} & \mat{C}^{e}_{1113} & \mat{C}^{e}_{1112} \\
    \mat{C}^{e}_{1122} & \mat{C}^{e}_{2222} & \mat{C}^{e}_{2233} & \mat{C}^{e}_{2223} & \mat{C}^{e}_{2213} & \mat{C}^{e}_{2212} \\
    \mat{C}^{e}_{1133} & \mat{C}^{e}_{2233} & \mat{C}^{e}_{3333} & \mat{C}^{e}_{3323} & \mat{C}^{e}_{3313} & \mat{C}^{e}_{3312} \\
    \mat{C}^{e}_{1123} & \mat{C}^{e}_{2223} & \mat{C}^{e}_{3323} & \mat{C}^{e}_{2323} & \mat{C}^{e}_{2313} & \mat{C}^{e}_{2312} \\
    \mat{C}^{e}_{1113} & \mat{C}^{e}_{2213} & \mat{C}^{e}_{3313} & \mat{C}^{e}_{2313} & \mat{C}^{e}_{1313} & \mat{C}^{e}_{1312} \\
    \mat{C}^{e}_{1112} & \mat{C}^{e}_{2212} & \mat{C}^{e}_{3312} & \mat{C}^{e}_{2312} & \mat{C}^{e}_{1312} & \mat{C}^{e}_{1212}
  \end{bmatrix},
\end{align}
where $\mat{C}^{e}_{ijkl}$ is the diagonal matrix of stiffness tensor elements
$C_{ijkl}$ evaluated at the quadrature nodes. With this, for tensor product
elements,~\eref{eqn:dg:stress:quad} is equivalently written as
\begin{align}
  \mat{M}^{e} \fd{\vec{\sigma}_{kl}^{e}}{t}
  =&
  \frac{1}{2} \mat{C}_{ijkl}^{e} \left(\mat{S}_{k}^{e}\vec{v}_{l}^{e} + \mat{S}_{l}^{e}\vec{v}_{k}^{e}\right)\\
  \notag
  &+
  \sum_{m \in \MM^{e}}
  \mat{C}_{ijkl}^{e}
  {\left(\mat{P}^{m,e}_{n_{k}}\right)}^{T} \mat{W}^{m}
  \left(\vec{v}_{l}^{*m} - \vec{v}_{l}^{m[e]}\right),
\end{align}
where we highlight that $\mat{C}_{ijkl}^{e}$ is defined on the volume
element and not the mortar.

For many element types such a decomposition is not possible because the
interpolation and quadrature points are different. One option in this case is to
invert $\mat{\bar{M}}_{S}^{e}$ on each element. Alternatively, one could
use the weight-adjusted approach of Chan, Hewett, and
Warburton~\cite{ChanHewettWarburton2017_curviWADG} and let
\begin{align}
  \mat{\bar{M}}^{e}_{S}
  &=
  \left(\mat{I}_{6\times 6} \otimes \mat{M}^{e}\right)
  {\left(\mat{\bar{M}}^{e}_{C}\right)}^{-1}
  \left(\mat{I}_{6\times 6} \otimes \mat{M}^{e}\right),
\end{align}
which then allows~\eref{eqn:dg:stress:quad} to be written as
\begin{align}
  \mat{M}^{e} \fd{\vec{\sigma}_{kl}^{e}}{t}
  =&
  \frac{1}{2} \mat{M}_{C_{ijkl}}^{e} {\left(\mat{M}^{e}\right)}^{-1} \left(\mat{S}_{k}^{e}\vec{v}_{l}^{e} + \mat{S}_{l}^{e}\vec{v}_{k}^{e}\right)\\
  \notag
  &+
  \sum_{m \in \MM^{e}}
  \mat{M}_{C_{ijkl}}^{e} {\left(\mat{M}^{e}\right)}^{-1}
  {\left(\mat{P}^{m,e}_{n_{k}}\right)}^{T} \mat{W}^{m}
  \left(\vec{v}_{l}^{*m} - \vec{v}_{l}^{m[e]}\right).
\end{align}
The weight adjusted approach can also be used for the case of nonconstant
Jacobian determinants when the mass matrix $\mat{M}^{e}$ is not
diagonal~\cite{ChanHewettWarburton2017_curviWADG}.
\subsection{Comparison with a Previous DG Mortar Method}\label{sec:kopriva}
Here we compare the proposed method with the approach outlined by
Kopriva~\cite{Kopriva1996jcp} and Kopriva, Woodruff, and
Hussaini~\cite{KoprivaWoodruffHussaini2002}, which was analyzed by Bui-Thanh and
Ghattas~\cite{BuiThanhGhattas2012sinum}. In this previous approach, the fluxes are
computed on the mortar, but they are projected back to the local element space
for integration. This is as opposed to our approach where we project the test
and trial functions to the mortar space, and all surface integrals are computed
on the mortar.

In this earlier approach, method~\eref{eqn:dg:vel}--\eref{eqn:dg:stress} would
be
\begin{align}
  \label{eqn:dg:vel:kop}
  \int_{\hat{e}}
  J^{e} \rho\; \phi_{i} \pd{v_{i}^{e}}{t}
  =&
  -
  \int_{\hat{e}}
  J^{e}
  \pd{\phi_{i}}{x_{j}}\sigma_{ij}^{e}
  +
  \int_{\partial\hat{e}} S_{J}^{e}\phi_{i}^{e}T_{i}^{*e},\\
  \label{eqn:dg:stress:kop}
  \int_{\hat{e}}
  J^{e}
  \psi_{ij} S_{ijkl} \pd{\sigma_{kl}^{e}}{t}
  =&
  \int_{\hat{e}}
  \frac{J^{e}}{2} \psi_{ij} \left(\pd{v_{i}^{e}}{x_{j}} + \pd{v_{j}^{e}}{x_{i}}\right)
  \\\notag
  &+
  \int_{\partial\hat{e}} S_{J}^{e} n_{j}^{e} \psi_{ij}^{e}
  \left(v_{i}^{*e} - v_{i}^{e}\right).
\end{align}
The numerical flux terms $T^{*e}_{i}$ and $v^{*e}_{i}$ are the L$^{2}$-projected
fluxes from the mortar onto the surface of the volume element, e.g.,
L$^{2}$-projection from the set of mortar elements $\MM^{e}$.
Since the L$^{2}$-projection is only applied to the trial function (and not the
test function), we call this approach the \emph{asymmetric flux integral method}
(AFIM).  As with the SFIM proposed above, the AFIM can be applied to either the
full-side or split-side mortars shown in \fref{fig:mortars}.

With exact integration and L$^{2}$-projection, this scheme has the same
stability properties as~\eref{eqn:dg:vel}--\eref{eqn:dg:stress}, but when
inexact quadrature is used this scheme may admit energy growth (even on
affine elements).  In the case of affine meshes with constant coefficients,
Bui-Thanh and Ghattas~\cite{BuiThanhGhattas2012sinum} showed that
Legendre-Gauss-Lobatto integration leads to an energy estimate of the form
$\mathcal{E}(t) \le e^{ct} \mathcal{E}(0)$ where $c > 0$ is a small but positive
constant that converges to zero under mesh refinement.  Thus, even though
stable, the scheme admits exponential energy growth.

The main difficulty in achieving a strictly non-increasing energy estimate for
this scheme is that when quadrature is used, the inexact face mass matrix and
L$^{2}$-projection operator are no longer consistent. This means that integrals
over the element face space and the mortar space are not equivalent.  One
approach to overcoming these difficulties would be to use inexact
L$^{2}$-projections defined such that the quadrature version of~\eref{eqn:dg:vel:kop}--\eref{eqn:dg:stress:kop} is purely dissipative. With some
reformation, the proposed scheme~\eref{eqn:dg:vel:quad}--\eref{eqn:dg:stress:quad} could be interpreted as giving
a consistent approach for defining these inexact L$^{2}$-projections.

\section{Computational Results}\label{sec:results}
In this section, the energy stability, conservation, and constant
preserving properties of SFIM are verified for $d = 3$ with isotropic elasticity
and the accuracy of the scheme is investigated.  The adapted meshes in the
examples are generated with \texttt{p4est}~\cite{BursteddeWilcoxGhattas11} using
its topology iterator~\cite{IsaacBursteddeWilcoxEtAl15}.  Hexahedral elements
with tensor product Legendre-Gauss-Lobatto (LGL) nodes are used for
interpolation and integration (e.g., the so-called discontinuous Galerkin
spectral element method~\cite{Black1999Kybernetika,Kopriva2009book}); LGL
quadrature is also used on the mortar elements. A computational advantage of LGL
quadrature is that some of the degrees of freedom exist on the element faces
which means that for conforming faces, no interpolation operation is required to
compute the flux.  That said, an LGL quadrature rule with $N+1$ points can only
integrate polynomials of degree $2N-1$ exactly, and thus the diagonal mass
matrix defined using LGL quadrature is inexact for two polynomials of degree
$N$. When curved elements and variable coefficients are used, the metric terms
and material properties are multiplied with the quadrature weights, thus
maintaining the diagonal structure of the mass matrix (at the cost of further
under-integration).  Despite this, as shown above, the newly proposed method
does not require exact integration as stability is achieved through the use of
skew-symmetry. The geometry terms are approximated in an isoparametric fashion
(i.e., the coordinate mapping is evaluated at the interpolation nodes of the
reference element) with the metric terms evaluated using the curl invariant form
of Kopriva~\cite{Kopriva2006jsc}; see also \aref{app:div}.

All of the schemes introduced above are in semi-discrete form with time left
continuous.  The scheme can be written as a system of linear ordinary differential
equations:
\begin{align}
  \label{eqn:ode}
  \fd{\vec{q}}{t} = \mat{A}\vec{q},
\end{align}
where $\vec{q}$ is the vector of stresses and particle velocities at all the
degrees of freedom. As is standard when discretizing hyperbolic equations, in our
implementation $\mat{A}$ is not explicitly formed only its action on $\vec{q}$
is computed. We integrate~\eref{eqn:ode} in time using the fourth-order,
low-storage, Runge-Kutta scheme of Carpenter and Kennedy~\cite{CarpenterKennedy1994}
((5,4) $2N$-Storage RK scheme, solution $3$). To compute the time step for the
Runge-Kutta method, at every node of the mesh we compute
\begin{align}
  \beta &= \min(\beta_{1}, \beta_{2}, \beta_{3}),\\
  \beta_{k} &= {\left(N
  \sqrt{C_{p} \pd{r_{k}}{x_{i}}\pd{r_{k}}{x_{i} }}\right)}^{-1}&
  \mbox{(no summation over $k$)},
\end{align}
where $N$ is the polynomial order and $C_{p} = \sqrt{(\lambda + 2\mu)/\rho}$ is
the p-wave speed of the material at the node. The maximum time step is then
chosen to be proportional to the minimum $\beta$ over the whole mesh.

\subsection{Planewave Solution in Periodic Box}
\begin{table}
  \centering
  \begin{tabular}{rllll}
    \toprule
    & \multicolumn{1}{c}{SFIM, full-side}
    & \multicolumn{1}{c}{SFIM, split-side}
    & \multicolumn{1}{c}{AFIM, full-side}
    & \multicolumn{1}{c}{AFIM, split-side}\\
    {$E$}
    & \multicolumn{1}{c}{error (rate)}
    & \multicolumn{1}{c}{error (rate)}
    & \multicolumn{1}{c}{error (rate)}
    & \multicolumn{1}{c}{error (rate)}\\
    \midrule
    &
    \multicolumn{4}{c}{Polynomial order $N = 3$}\\
    $     36$ & $1.9\times10^{  0\phantom{0}}$         & $1.8\times10^{  0\phantom{0}}$         & $1.7\times10^{  0\phantom{0}}$         & $1.7\times10^{  0\phantom{0}}$        \\
    $    288$ & $5.6\times10^{ -2\phantom{0}}$ $(5.1)$ & $5.5\times10^{ -2\phantom{0}}$ $(5.0)$ & $4.9\times10^{ -2\phantom{0}}$ $(5.1)$ & $4.9\times10^{ -2\phantom{0}}$ $(5.1)$\\
    $   2304$ & $1.4\times10^{ -3\phantom{0}}$ $(5.3)$ & $1.4\times10^{ -3\phantom{0}}$ $(5.3)$ & $1.2\times10^{ -3\phantom{0}}$ $(5.3)$ & $1.2\times10^{ -3\phantom{0}}$ $(5.3)$\\
    $  18432$ & $8.8\times10^{ -5\phantom{0}}$ $(4.0)$ & $8.9\times10^{ -5\phantom{0}}$ $(4.0)$ & $6.6\times10^{ -5\phantom{0}}$ $(4.2)$ & $6.6\times10^{ -5\phantom{0}}$ $(4.2)$\\
    $ 147456$ & $7.5\times10^{ -6\phantom{0}}$ $(3.5)$ & $8.6\times10^{ -6\phantom{0}}$ $(3.4)$ & $4.7\times10^{ -6\phantom{0}}$ $(3.8)$ & $4.7\times10^{ -6\phantom{0}}$ $(3.8)$\\
    $1179648$ & $6.9\times10^{ -7\phantom{0}}$ $(3.4)$ & $9.6\times10^{ -7\phantom{0}}$ $(3.2)$ & $3.6\times10^{ -7\phantom{0}}$ $(3.7)$ & $3.6\times10^{ -7\phantom{0}}$ $(3.7)$\\
    \midrule
    &
    \multicolumn{4}{c}{Polynomial order $N = 4$}\\
    $     36$ & $1.1\times10^{ -1\phantom{0}}$         & $1.1\times10^{ -1\phantom{0}}$         & $9.5\times10^{ -2\phantom{0}}$         & $9.6\times10^{ -2\phantom{0}}$        \\
    $    288$ & $1.3\times10^{ -3\phantom{0}}$ $(6.5)$ & $1.2\times10^{ -3\phantom{0}}$ $(6.5)$ & $1.1\times10^{ -3\phantom{0}}$ $(6.5)$ & $1.1\times10^{ -3\phantom{0}}$ $(6.5)$\\
    $   2304$ & $4.3\times10^{ -5\phantom{0}}$ $(4.9)$ & $4.3\times10^{ -5\phantom{0}}$ $(4.8)$ & $3.5\times10^{ -5\phantom{0}}$ $(5.0)$ & $3.4\times10^{ -5\phantom{0}}$ $(5.0)$\\
    $  18432$ & $1.8\times10^{ -6\phantom{0}}$ $(4.6)$ & $1.9\times10^{ -6\phantom{0}}$ $(4.5)$ & $1.2\times10^{ -6\phantom{0}}$ $(4.9)$ & $1.2\times10^{ -6\phantom{0}}$ $(4.8)$\\
    $ 147456$ & $7.7\times10^{ -8\phantom{0}}$ $(4.5)$ & $9.8\times10^{ -8\phantom{0}}$ $(4.3)$ & $4.3\times10^{ -8\phantom{0}}$ $(4.8)$ & $4.3\times10^{ -8\phantom{0}}$ $(4.8)$\\
    \midrule
    &
    \multicolumn{4}{c}{Polynomial order $N = 5$}\\
    $     36$ & $4.2\times10^{ -3\phantom{0}}$         & $4.0\times10^{ -3\phantom{0}}$         & $3.5\times10^{ -3\phantom{0}}$         & $3.4\times10^{ -3\phantom{0}}$        \\
    $    288$ & $7.6\times10^{ -5\phantom{0}}$ $(5.8)$ & $7.4\times10^{ -5\phantom{0}}$ $(5.8)$ & $6.5\times10^{ -5\phantom{0}}$ $(5.8)$ & $6.4\times10^{ -5\phantom{0}}$ $(5.7)$\\
    $   2304$ & $1.5\times10^{ -6\phantom{0}}$ $(5.7)$ & $1.5\times10^{ -6\phantom{0}}$ $(5.6)$ & $1.2\times10^{ -6\phantom{0}}$ $(5.8)$ & $1.2\times10^{ -6\phantom{0}}$ $(5.8)$\\
    $  18432$ & $3.1\times10^{ -8\phantom{0}}$ $(5.6)$ & $3.5\times10^{ -8\phantom{0}}$ $(5.4)$ & $2.2\times10^{ -8\phantom{0}}$ $(5.8)$ & $2.1\times10^{ -8\phantom{0}}$ $(5.8)$\\
    $ 147456$ & $6.8\times10^{-10           }$ $(5.5)$ & $9.3\times10^{-10           }$ $(5.2)$ & $4.1\times10^{-10           }$ $(5.7)$ & $4.1\times10^{-10           }$ $(5.7)$\\
    \midrule
    &
    \multicolumn{4}{c}{Polynomial order $N = 6$}\\
    $     36$ & $5.2\times10^{ -4\phantom{0}}$         & $5.2\times10^{ -4\phantom{0}}$         & $4.7\times10^{ -4\phantom{0}}$         & $4.7\times10^{ -4\phantom{0}}$        \\
    $    288$ & $4.4\times10^{ -6\phantom{0}}$ $(6.9)$ & $4.3\times10^{ -6\phantom{0}}$ $(6.9)$ & $3.7\times10^{ -6\phantom{0}}$ $(7.0)$ & $3.6\times10^{ -6\phantom{0}}$ $(7.0)$\\
    $   2304$ & $4.3\times10^{ -8\phantom{0}}$ $(6.7)$ & $4.5\times10^{ -8\phantom{0}}$ $(6.6)$ & $3.2\times10^{ -8\phantom{0}}$ $(6.8)$ & $3.2\times10^{ -8\phantom{0}}$ $(6.8)$\\
    $  18432$ & $4.6\times10^{-10           }$ $(6.5)$ & $5.4\times10^{-10           }$ $(6.4)$ & $3.0\times10^{-10           }$ $(6.7)$ & $3.0\times10^{-10           }$ $(6.7)$\\
    \midrule
    &
    \multicolumn{4}{c}{Polynomial order $N = 7$}\\
    $     36$ & $3.9\times10^{ -5\phantom{0}}$         & $3.7\times10^{ -5\phantom{0}}$         & $3.1\times10^{ -5\phantom{0}}$         & $3.0\times10^{ -5\phantom{0}}$        \\
    $    288$ & $2.2\times10^{ -7\phantom{0}}$ $(7.5)$ & $2.2\times10^{ -7\phantom{0}}$ $(7.4)$ & $1.8\times10^{ -7\phantom{0}}$ $(7.4)$ & $1.8\times10^{ -7\phantom{0}}$ $(7.4)$\\
    $   2304$ & $1.2\times10^{ -9\phantom{0}}$ $(7.5)$ & $1.2\times10^{ -9\phantom{0}}$ $(7.5)$ & $9.4\times10^{-10           }$ $(7.6)$ & $9.4\times10^{-10           }$ $(7.6)$\\
    \bottomrule
  \end{tabular}
  \caption{Error and estimated convergence rates for a planewave propagating
  through an adapted, affine mesh. The base mesh and a log-log plot of the error
  for the SFIM with full-side mortar elements are shown in
  \fref{fig:planewave}.\label{tab:planewave}}
\end{table}
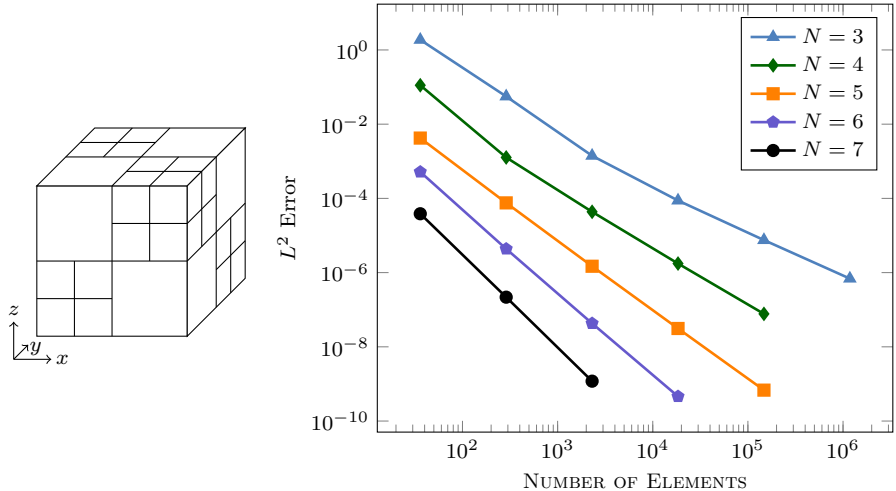
\begin{figure}[tb]
  \centering
  \begin{minipage}{0.28\textwidth}
  \begin{tikzpicture}
    \draw[black,fill=white] ( 0, 0, 0) -- ( 2, 0, 0) -- ( 2, 2, 0) -- ( 0, 2, 0) -- cycle;
    \draw[black]            ( 0, 1, 0) -- ( 2, 1, 0);
    \draw[black]            ( 1, 0, 0) -- ( 1, 2, 0);
    \draw[black,fill=white] ( 0, 2, 0) -- ( 2, 2, 0) -- ( 2, 2,-2) -- ( 0, 2,-2) -- cycle;
    \draw[black]            ( 0, 2,-1) -- ( 2, 2,-1);
    \draw[black]            ( 1, 2, 0) -- ( 1, 2,-2);
    \draw[black,fill=white] ( 2, 0, 0) -- ( 2, 0,-2) -- ( 2, 2,-2) -- ( 2, 2, 0) -- cycle;
    \draw[black]            ( 2, 2,-1) -- ( 2, 0,-1);
    \draw[black]            ( 2, 1, 0) -- ( 2, 1,-2);
    \draw[black]            ( 2, 1,-1.5) -- ( 2, 0,-1.5);
    \draw[black]            ( 2, 0.5,-1) -- ( 2, 0.5,-2);
    \draw[black]            ( 2, 2,-0.5) -- ( 2, 1,-0.5);
    \draw[black]            ( 2, 1.5,0) -- ( 2, 1.5,-1);
    \draw[black]            ( 1, 2,-0.5) -- ( 2, 2,-0.5);
    \draw[black]            ( 1.5, 2,0) -- ( 1.5, 2,-1);
    \draw[black]            ( 1, 1.5,0) -- ( 2, 1.5,0);
    \draw[black]            ( 1.5, 1,0) -- ( 1.5, 2,0);
    \draw[black]            ( 0, 0.5,0) -- ( 1, 0.5,0);
    \draw[black]            ( 0.5, 0,0) -- ( 0.5, 1,0);
    \draw[black]            ( 0, 2,-1.5) -- ( 1, 2,-1.5);
    \draw[black]            ( 0.5, 2,-1) -- ( 0.5, 2,-2);
    \draw[black,->]           (-0.5,-0.5,-0.5) -- ( 0,-0.5,-0.5);
    \node at (0.15,-0.5,-0.5)   {$x$};
    \draw[black,->]           (-0.5,-0.5,-0.5) -- (-0.5,-0.5,-1);
    \node at (-0.4,-0.55,-1)   {$y$};
    \draw[black,->]           (-0.5,-0.5,-0.5) -- (-0.5, 0,-0.5);
    \node at (-0.5, 0.15,-0.5)   {$z$};
  \end{tikzpicture}
  \end{minipage}
  \begin{minipage}{0.7\textwidth}
  \begin{tikzpicture}
    \begin{loglogaxis}[
        xlabel=\textsc{Number of Elements},
        ylabel=L$^2$ Error,
        legend pos= north east
      ]

      \addplot[color=blu,mark=triangle*,line width=1] plot coordinates {%
      (     36, 1.8741722800475045e+00)
      (    288, 5.5878350954557579e-02)
      (   2304, 1.4062963708636018e-03)
      (  18432, 8.7566614639383086e-05)
      ( 147456, 7.5225245583356782e-06)
      (1179648, 6.8946858545280761e-07)
      };

      \addplot[color=grn,mark=diamond*,line width=1] plot coordinates {%
      (     36, 1.1179884147242004e-01)
      (    288, 1.2673687279605818e-03)
      (   2304, 4.3488768309695393e-05)
      (  18432, 1.7551778919635786e-06)
      ( 147456, 7.7219780694960657e-08)
      };

      \addplot[color=org,mark=square*,line width=1] plot coordinates {%
      (     36, 4.2189754823362393e-03)
      (    288, 7.5684322429350501e-05)
      (   2304, 1.4840635580510757e-06)
      (  18432, 3.1196878622945801e-08)
      ( 147456, 6.8069542506441082e-10)
      };

      \addplot[color=prp,mark=pentagon*,line width=1] plot coordinates {%
      (     36, 5.1653277688456058e-04)
      (    288, 4.3547317556368376e-06)
      (   2304, 4.2816312193076380e-08)
      (  18432, 4.6042395782750303e-10)
      };

      \addplot[color=black,mark=*,line width=1] plot coordinates {%
        (     36, 3.8565217760936989e-05)
        (    288, 2.1774845535348811e-07)
        (   2304, 1.1868702110569641e-09)
      };

      \legend{$N=3$\\$N=4$\\$N=5$\\$N=6$\\$N=7$\\};
    \end{loglogaxis}
  \end{tikzpicture}
  \end{minipage}
  \caption{(left) Base mesh for planewave test. (right) Log-log plot of $E$
  (number of elements) versus the L$^{2}$ error (measured with the energy norm)
  for a mesh of affine (square) elements. Only the error for the full-side
  mortar is shown as the other three cases look similar. Numerical values and
  rates for all four methods are given in
  Table~\ref{tab:planewave}.\label{fig:planewave}}
\end{figure}
In this test, SFIM and AFIM are compared with both full-side
and split-side mortar elements. To do this, an affine mesh with constant
material properties is considered. Subsequent tests involve variable material
properties and non-affine meshes and will only use SFIM\@. Recall that SFIM,
described in \sref{sec:quad}, refers to the symmetric application of
L$^2$-projection operators to both the test and trial function, whereas AFIM,
described in \sref{sec:kopriva}, refers to the asymmetric application of the
L$^2$-projection operators to the trial and flux functions (but not the
test function). Recall also that full-side mortar elements conform to the larger
volume element faces and the split-side mortar elements conform the to smaller
volume element faces (see \fref{fig:mortars}).

The domain is taken to be the unit cube: $\Omega = {[0,1]}^3$. The domain is
initially partitioned into a $2 \times 2 \times 2$ mesh of hexahedral elements,
and then four elements are further subdivided into $8$ elements; see
\fref{fig:planewave}. This base mesh has $|\EE| = 36$ elements, and in the base mesh
only faces internal to the refined blocks are conforming.

The material is taken to be homogeneous, isotropic with $\rho = 2$, $\mu = 3$,
and $\lambda = 4$. The solution is a planewave propagating
in the $x_{1}$-direction:
\begin{align}
  u_{1} &= \cos\left(2 \pi \left(c_{p} t + x_{1}\right) \right),&
  u_{2} &= \cos\left(2 \pi \left(c_{s} t + x_{1}\right) \right),&
  u_{3} &= \cos\left(2 \pi \left(c_{s} t + x_{1}\right) \right),
\end{align}
with $c_{p} = \sqrt{(\lambda + 2 \mu) / \rho}$ and $c_{s} = \sqrt{\mu / \rho}$
being the P- and S-wave speeds of the material. Here the solution is written in
terms of the displacements, and the velocity and stresses are
\begin{align}
  \label{eqn:u2vS}
  v_{i} &= \pd{u_{i}}{t},&
  \sigma_{ij} &= \lambda \delta_{ij} \pd{u_{k}}{x_{k}}
  + \mu \left( \pd{u_{i}}{x_{j}} + \pd{u_{j}}{x_{i}} \right).
\end{align}

Table~\ref{tab:planewave} and \fref{fig:planewave} show convergence results for
this planewave test for varying polynomial order $N$; refinement of the mesh is
done with bisection so that each hexahedral element is subdivided into $8$
affine elements of equal size. The final time of the simulation is $t = 20 /
c_{s}$, e.g., the planewave propagates around the unit cube $20$ times. As
Table~\ref{tab:planewave} shows, four methods are considered for each $N$: SFIM
with split-side mortar elements, SFIM with full-side mortar elements, AFIM with
split-side mortar elements, and AFIM with full-side mortar elements.  In all
cases, the error in the solution is measured using the quadrature based energy
norm~\eref{eqn:energy:quad}. The numerical flux used for these tests is the
upwind flux described by~\eref{eqn:flux:full}--\eref{eqn:flux:material} with
$\alpha = 1$.

As can be seen Table~\ref{tab:planewave}, all four methods converge at
high-rates. That said, the AFIMs do have lower errors and
improved rates. The two newly proposed methods seem
to be tending towards convergence rates at order $N$ (as opposed to $N+1/2$),
suggesting that some accuracy is lost with the improved stability properties.
\subsection{Eigenvalue Spectrum for Periodic Box}
\begin{table}
  \begin{center}
    \begin{tabular}{llll}
      \toprule
      & $\max_{k}\;\Re(\lambda_{k})$ & $\min_{k}\;\Re(\lambda_{k})$ & $\max_{k}\;|\Im(\lambda_{k})|$\\
      \midrule
      & \multicolumn{3}{c}{Upwind Flux}\\
      SFIM, full-side            & $6.91\times10^{-13}$ & $-5.66\times10^{2}$ & $4.25\times10^{2}$\\
      SFIM, split-side         & $4.81\times10^{-13}$ & $-3.07\times10^{2}$ & $1.47\times10^{2}$\\
      AFIM, full-side    & $7.19\times10^{-05}$ & $-3.06\times10^{2}$ & $1.73\times10^{2}$\\
      AFIM, split-side & $6.19\times10^{-05}$ & $-3.07\times10^{2}$ & $1.47\times10^{2}$\\
      \midrule
      & \multicolumn{3}{c}{Central Flux}\\
      SFIM, full-side            & $1.58\times10^{-12}$ & $-9.32\times10^{-13}$ & $4.89\times10^{2}$\\
      SFIM, split-side         & $1.10\times10^{-12}$ & $-8.08\times10^{-13}$ & $2.07\times10^{2}$\\
      AFIM, full-side    & $8.57\times10^{-1 }$ & $-8.57\times10^{-1 }$ & $2.08\times10^{2}$\\
      AFIM, split-side & $1.15\times10      $ & $-1.15\times10      $ & $2.07\times10^{2}$\\
      \bottomrule
    \end{tabular}
  \end{center}
  \caption{Table comparing the extrema of the real and imaginary parts of the
  eigenvalue spectrum for all four methods with both an upwind and central flux
  with $N=4$ on the base mesh ($|\EE| = 36$). Eigenvalues are computed by forming
  the matrix and using the MATLAB~\cite{MATLAB2015b} \texttt{eig}
  command. (See also \fref{fig:eig}.)\label{tab:eig}}
\end{table}
\begin{figure}
  \begin{minipage}{0.48\textwidth}
    \includegraphics{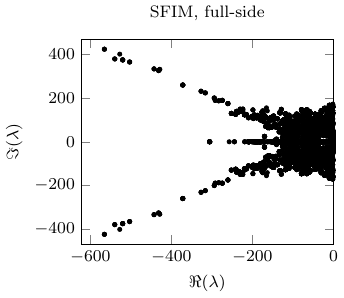}
  \end{minipage}
  \hfill
  \begin{minipage}{0.48\textwidth}
    \includegraphics{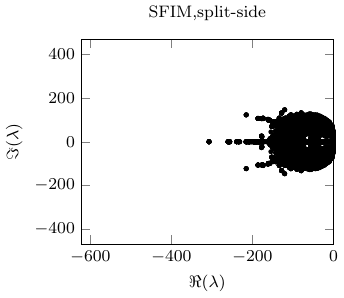}
  \end{minipage}
  \medskip

  \begin{minipage}{0.48\textwidth}
    \includegraphics{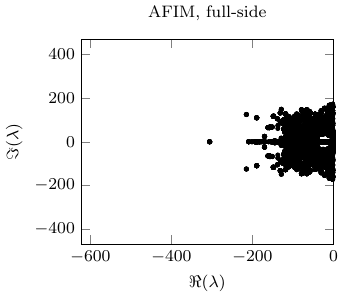}
  \end{minipage}
  \hfill
  \begin{minipage}{0.48\textwidth}
    \includegraphics{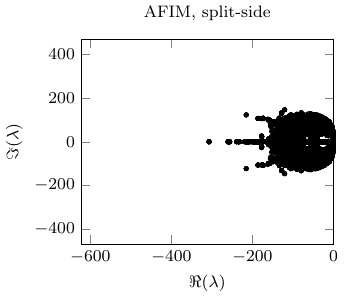}
  \end{minipage}
  \caption{Eigenvalue spectrum for the $|\EE| = 36$ mesh with polynomial order $N =
  4$ upwind flux. (See Table~\ref{tab:eig} for the limits of the spectrum for
  these spectra.)\label{fig:eig}}
\end{figure}
To highlight the stability properties of the methods, we now consider the
eigenvalue spectrum of each of the methods using the mesh and material
properties of the previous periodic box test problem. \fref{fig:eig} shows the
eigenvalue spectrum for all the methods with an upwind flux on the base mesh
($\EE|
= 36$); the spectrum is computed by forming the matrix $\mat{A}$
and then finding the eigenvalues with the MATLAB~\cite{MATLAB2015b} \texttt{eig}
command. Table~\ref{tab:eig} gives the maximum and minimum real part of the
eigenvalue spectrum for all four methods with both the upwind ($\alpha = 1$) and
central ($\alpha = 0$) flux.  As can be seen, the SFIM with both split-side and
full-side mortar elements are stable (in the sense that the maximum real part
of the eigenvalue spectrum is close to zero) consistent with the energy analysis
earlier in the paper. The AFIMs have positive real parts, even
with the upwind flux, consistent
with the energy analysis of Bui-Thanh and Ghattas which allows for
energy growth (e.g., a positive, real part of the spectrum). For this mesh,
which has a total of $40500$ degrees of freedom, the AFIMs
with the upwind flux have, respectively, $18$ and $12$
computed eigenvalues that have a real part larger than $10^{-12}$ with a maximum
of, respectively, $7.19\times10^{-05}$ and $6.40\times10^{-04}$.

As can be seen from both the \fref{fig:eig} and Table~\ref{tab:eig}, the
spectral radius of the SFIM with full-side mortar elements is larger than the
SFIM with split-side mortar elements ($2.3$ times larger with the upwind flux
and $2.4$ times larger with the central flux).  The implication of this is that
the largest time step that can be used for the SFIM with full-side mortar
elements (on this mesh) is almost half the
size of that of the SFIM with split-side mortar elements. Given that this particular mesh has a
high nonconforming to conforming ratio ($2/3$ of the mortar is nonconforming)
it is unclear whether this stiffness would be seen in practical simulations. For
the upwind flux, both SFIM and AFIM with full-side mortar elements have twice
as many eigenvalues with approximately zero real part as compared with the
SFIM and AFIM with split-side mortar elements ($\sim 8300$ versus $\sim 4200$
eigenvalues with real component less than $10^{-12}$ in magnitude).  This
suggests that the SFIM and AFIM with split-side mortar elements are slightly
more dissipative across the nonconforming interfaces (assuming that there is
energy in the associated eigenfunctions of the solution).

Since even on affine meshes with constant coefficients, AFIM does not guarantee
an energy non-increasing property in the remainder of the paper we only consider
SFIM\@. That said, AFIM is not energy increasing on all meshes when upwinding
($\alpha = 1$) is used; every test we have run using AFIM with a central flux
($\alpha = 0$) has significant energy growth.

\subsection{Conservation, Constant Preserving, and Long-Time Energy Stability}
\begin{figure}[tb]
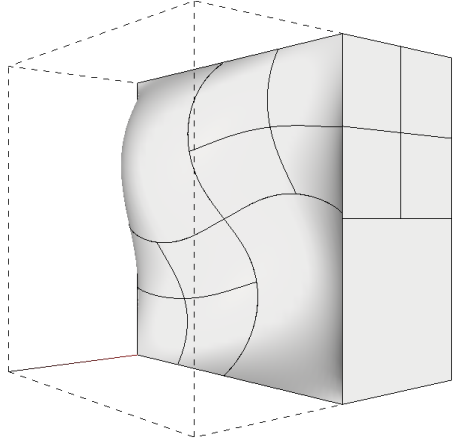

  \centering
%
  \caption{Cross-section of the skewed mesh defined by transform~\eref{eqn:curve}. The total number of elements in the mesh is $|\EE| =
  36$\label{fig:skew:mesh}}
\end{figure}
To further explore the long-time stability as well as the conservation and
constant preserving properties of SFIM a curvilinear mesh with heterogeneous
material properties is now considered. Let $\Omega = {[-1,1]}^3$ be periodic in
all directions. The mesh is made curvilinear by introducing the following global
coordinate transform
\begin{align}
  \label{eqn:curve}
  x_{i} &= r_{j} Q_{ij},
\end{align}
where the coefficients $Q_{ij}$ are the elements of the matrix
\begin{align}
  \mat{Q}(\beta)
  =
  \begin{bmatrix}
    \cos^{2}(\beta) & -\cos(\beta) \sin(\beta) & \sin(\beta)\\
    \sin(\beta) & \cos(\beta) & 0\\
    -\cos(\beta) \sin(\beta) & \sin^{2}(\beta) & \cos(\beta)
  \end{bmatrix}
\end{align}
where $\beta = (\pi/4) \prod_{i=1}^{3}(1-r_{i}^2)$. When $r_{1} = r_{2} = r_{3} =
0$ this corresponds to two rotations of $\pi/4$ and there is no rotation when
$|r_{i}| = 1$ for any $i\in\{1,2,3\}$. A cross-section of the mesh is shown in
\fref{fig:skew:mesh} and the same refinement patten as in the planewave test has
been used (every other corner has been refined once for a total of 36 elements).
At every degree of freedom of the mesh the material properties are assigned
using a pseudorandom number generator which results in a discontinuous,
non-smooth description; the S-wave and P-wave speeds vary from $4.4$ to $6.5$
and $7.2$ to $10.7$, respectively. All of the tests in this section are run with
polynomial order $N = 4$.

Three different approaches to handling the geometry are considered which we call
the interpolated, watertight, and continuous metric approach. In the
interpolated geometry handling, the coordinate transform~\eref{eqn:curve} is
sampled at the LGL nodal degrees of freedom. Since the transform is
non-polynomial the mesh will be discretely discontinuous across nonconforming
faces. In the watertight approach, the physical mesh is made continuous by
interpolating the $x_{i}$ face values from larger faces and edges to smaller faces
and edges\footnote{Nonconforming edge values need to be modified in order to
ensure that new gaps and overlaps are not introduced along the edges of
conforming elements.}. Finally, in the continuous metric approach the product
$S_{J}n_{i}$ is made proportional across nonconforming faces and edges by
ensuring that aliasing errors in the computation of these terms are incurred
similarly; see \aref{app:div}.

In order to test the long-time stability and the conservation properties of SFIM,
the initial condition is generated with a pseudorandom number generator. This is
done to widely distribute the energy in the solution across various eigenmodes
of the operator.  The simulation is then run to time $t = 10 L / c_{s\min}$
where $L = 2 \sqrt{3}$ is the longest corner-to-corner distance in the domain
and $c_{s\min} = \sqrt{20}$ is the slowest wave speed in the system; in the
figures in this section time is normalized by $t_{0} = L / c_{s\min}$.

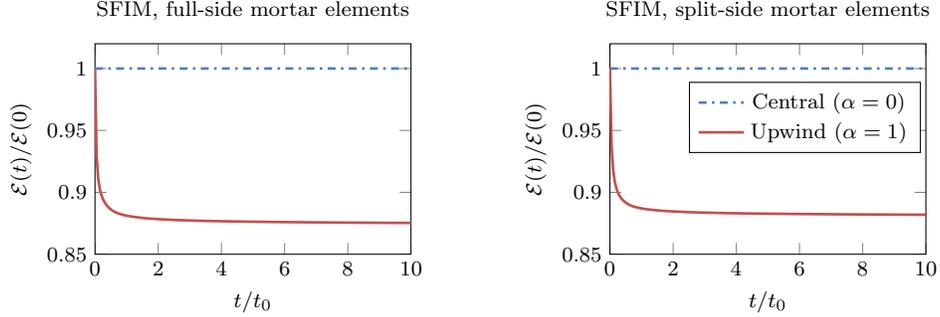
\begin{figure}
  \begin{minipage}{0.45\textwidth}
    \begin{tikzpicture}
      \begin{axis}[%
        width=0.75\textwidth,
        height=0.5\textwidth,
        scale only axis,
        xmin=0,xmax=10,
        ymin=0.85,ymax=1.02,
        legend style={at={(0.25,0.65)},anchor=west},
        ylabel={$\mathcal{E}(t) / \mathcal{E}(0)$}, xlabel={$t / t_{0}$},
        title={SFIM, full-side mortar elements}
        ]
          \addplot [color=blu,line width=1,dashdotted]
            table[col sep=comma, row sep=crcr]{%
              0.000000e+00, 1.000000e+00\\
              4.983863e-02, 1.000000e+00\\
              9.967726e-02, 1.000000e+00\\
              1.495159e-01, 1.000000e+00\\
              1.993545e-01, 9.999999e-01\\
              2.491931e-01, 9.999999e-01\\
              2.990318e-01, 9.999999e-01\\
              3.488704e-01, 9.999999e-01\\
              3.987090e-01, 9.999999e-01\\
              4.485477e-01, 9.999999e-01\\
              4.983863e-01, 9.999999e-01\\
              5.482249e-01, 9.999999e-01\\
              5.980635e-01, 9.999998e-01\\
              6.479022e-01, 9.999998e-01\\
              6.977408e-01, 9.999998e-01\\
              7.475794e-01, 9.999998e-01\\
              7.974181e-01, 9.999998e-01\\
              8.472567e-01, 9.999998e-01\\
              8.970953e-01, 9.999998e-01\\
              9.469339e-01, 9.999998e-01\\
              9.967726e-01, 9.999998e-01\\
              1.046611e+00, 9.999998e-01\\
              1.096450e+00, 9.999997e-01\\
              1.146288e+00, 9.999997e-01\\
              1.196127e+00, 9.999997e-01\\
              1.245966e+00, 9.999997e-01\\
              1.295804e+00, 9.999997e-01\\
              1.345643e+00, 9.999997e-01\\
              1.395482e+00, 9.999997e-01\\
              1.445320e+00, 9.999997e-01\\
              1.495159e+00, 9.999997e-01\\
              1.544997e+00, 9.999997e-01\\
              1.594836e+00, 9.999997e-01\\
              1.644675e+00, 9.999997e-01\\
              1.694513e+00, 9.999996e-01\\
              1.744352e+00, 9.999996e-01\\
              1.794191e+00, 9.999996e-01\\
              1.844029e+00, 9.999996e-01\\
              1.893868e+00, 9.999996e-01\\
              1.943707e+00, 9.999996e-01\\
              1.993545e+00, 9.999996e-01\\
              2.043384e+00, 9.999996e-01\\
              2.093222e+00, 9.999996e-01\\
              2.143061e+00, 9.999996e-01\\
              2.192900e+00, 9.999996e-01\\
              2.242738e+00, 9.999996e-01\\
              2.292577e+00, 9.999996e-01\\
              2.342416e+00, 9.999996e-01\\
              2.392254e+00, 9.999995e-01\\
              2.442093e+00, 9.999995e-01\\
              2.491931e+00, 9.999995e-01\\
              2.541770e+00, 9.999995e-01\\
              2.591609e+00, 9.999995e-01\\
              2.641447e+00, 9.999995e-01\\
              2.691286e+00, 9.999995e-01\\
              2.741125e+00, 9.999995e-01\\
              2.790963e+00, 9.999995e-01\\
              2.840802e+00, 9.999995e-01\\
              2.890640e+00, 9.999995e-01\\
              2.940479e+00, 9.999995e-01\\
              2.990318e+00, 9.999995e-01\\
              3.040156e+00, 9.999995e-01\\
              3.089995e+00, 9.999995e-01\\
              3.139834e+00, 9.999995e-01\\
              3.189672e+00, 9.999994e-01\\
              3.239511e+00, 9.999994e-01\\
              3.289349e+00, 9.999994e-01\\
              3.339188e+00, 9.999994e-01\\
              3.389027e+00, 9.999994e-01\\
              3.438865e+00, 9.999994e-01\\
              3.488704e+00, 9.999994e-01\\
              3.538543e+00, 9.999994e-01\\
              3.588381e+00, 9.999994e-01\\
              3.638220e+00, 9.999994e-01\\
              3.688058e+00, 9.999994e-01\\
              3.737897e+00, 9.999994e-01\\
              3.787736e+00, 9.999994e-01\\
              3.837574e+00, 9.999994e-01\\
              3.887413e+00, 9.999994e-01\\
              3.937252e+00, 9.999994e-01\\
              3.987090e+00, 9.999994e-01\\
              4.036929e+00, 9.999993e-01\\
              4.086768e+00, 9.999993e-01\\
              4.136606e+00, 9.999993e-01\\
              4.186445e+00, 9.999993e-01\\
              4.236283e+00, 9.999993e-01\\
              4.286122e+00, 9.999993e-01\\
              4.335961e+00, 9.999993e-01\\
              4.385799e+00, 9.999993e-01\\
              4.435638e+00, 9.999993e-01\\
              4.485477e+00, 9.999993e-01\\
              4.535315e+00, 9.999993e-01\\
              4.585154e+00, 9.999993e-01\\
              4.634992e+00, 9.999993e-01\\
              4.684831e+00, 9.999993e-01\\
              4.734670e+00, 9.999993e-01\\
              4.784508e+00, 9.999993e-01\\
              4.834347e+00, 9.999993e-01\\
              4.884186e+00, 9.999993e-01\\
              4.934024e+00, 9.999993e-01\\
              4.983863e+00, 9.999992e-01\\
              5.033701e+00, 9.999992e-01\\
              5.083540e+00, 9.999992e-01\\
              5.133379e+00, 9.999992e-01\\
              5.183217e+00, 9.999992e-01\\
              5.233056e+00, 9.999992e-01\\
              5.282895e+00, 9.999992e-01\\
              5.332733e+00, 9.999992e-01\\
              5.382572e+00, 9.999992e-01\\
              5.432410e+00, 9.999992e-01\\
              5.482249e+00, 9.999992e-01\\
              5.532088e+00, 9.999992e-01\\
              5.581926e+00, 9.999992e-01\\
              5.631765e+00, 9.999992e-01\\
              5.681604e+00, 9.999992e-01\\
              5.731442e+00, 9.999992e-01\\
              5.781281e+00, 9.999992e-01\\
              5.831120e+00, 9.999992e-01\\
              5.880958e+00, 9.999992e-01\\
              5.930797e+00, 9.999992e-01\\
              5.980635e+00, 9.999992e-01\\
              6.030474e+00, 9.999991e-01\\
              6.080313e+00, 9.999991e-01\\
              6.130151e+00, 9.999991e-01\\
              6.179990e+00, 9.999991e-01\\
              6.229829e+00, 9.999991e-01\\
              6.279667e+00, 9.999991e-01\\
              6.329506e+00, 9.999991e-01\\
              6.379344e+00, 9.999991e-01\\
              6.429183e+00, 9.999991e-01\\
              6.479022e+00, 9.999991e-01\\
              6.528860e+00, 9.999991e-01\\
              6.578699e+00, 9.999991e-01\\
              6.628538e+00, 9.999991e-01\\
              6.678376e+00, 9.999991e-01\\
              6.728215e+00, 9.999991e-01\\
              6.778053e+00, 9.999991e-01\\
              6.827892e+00, 9.999991e-01\\
              6.877731e+00, 9.999991e-01\\
              6.927569e+00, 9.999991e-01\\
              6.977408e+00, 9.999991e-01\\
              7.027247e+00, 9.999991e-01\\
              7.077085e+00, 9.999990e-01\\
              7.126924e+00, 9.999990e-01\\
              7.176762e+00, 9.999990e-01\\
              7.226601e+00, 9.999990e-01\\
              7.276440e+00, 9.999990e-01\\
              7.326278e+00, 9.999990e-01\\
              7.376117e+00, 9.999990e-01\\
              7.425956e+00, 9.999990e-01\\
              7.475794e+00, 9.999990e-01\\
              7.525633e+00, 9.999990e-01\\
              7.575471e+00, 9.999990e-01\\
              7.625310e+00, 9.999990e-01\\
              7.675149e+00, 9.999990e-01\\
              7.724987e+00, 9.999990e-01\\
              7.774826e+00, 9.999990e-01\\
              7.824665e+00, 9.999990e-01\\
              7.874503e+00, 9.999990e-01\\
              7.924342e+00, 9.999990e-01\\
              7.974181e+00, 9.999990e-01\\
              8.024019e+00, 9.999990e-01\\
              8.073858e+00, 9.999990e-01\\
              8.123696e+00, 9.999990e-01\\
              8.173535e+00, 9.999990e-01\\
              8.223374e+00, 9.999989e-01\\
              8.273212e+00, 9.999989e-01\\
              8.323051e+00, 9.999989e-01\\
              8.372890e+00, 9.999989e-01\\
              8.422728e+00, 9.999989e-01\\
              8.472567e+00, 9.999989e-01\\
              8.522405e+00, 9.999989e-01\\
              8.572244e+00, 9.999989e-01\\
              8.622083e+00, 9.999989e-01\\
              8.671921e+00, 9.999989e-01\\
              8.721760e+00, 9.999989e-01\\
              8.771599e+00, 9.999989e-01\\
              8.821437e+00, 9.999989e-01\\
              8.871276e+00, 9.999989e-01\\
              8.921114e+00, 9.999989e-01\\
              8.970953e+00, 9.999989e-01\\
              9.020792e+00, 9.999989e-01\\
              9.070630e+00, 9.999989e-01\\
              9.120469e+00, 9.999989e-01\\
              9.170308e+00, 9.999989e-01\\
              9.220146e+00, 9.999989e-01\\
              9.269985e+00, 9.999989e-01\\
              9.319823e+00, 9.999989e-01\\
              9.369662e+00, 9.999988e-01\\
              9.419501e+00, 9.999988e-01\\
              9.469339e+00, 9.999988e-01\\
              9.519178e+00, 9.999988e-01\\
              9.569017e+00, 9.999988e-01\\
              9.618855e+00, 9.999988e-01\\
              9.668694e+00, 9.999988e-01\\
              9.718533e+00, 9.999988e-01\\
              9.768371e+00, 9.999988e-01\\
              9.818210e+00, 9.999988e-01\\
              9.868048e+00, 9.999988e-01\\
              9.917887e+00, 9.999988e-01\\
              9.967726e+00, 9.999988e-01\\
              1.000000e+01, 9.999988e-01\\
          };
          \addplot [color=red,line width=1,solid]
            table[col sep=comma, row sep=crcr]{%
              0.000000e+00, 1.000000e+00\\
              4.983863e-02, 9.296563e-01\\
              9.967726e-02, 9.114511e-01\\
              1.495159e-01, 9.029091e-01\\
              1.993545e-01, 8.979553e-01\\
              2.491931e-01, 8.945987e-01\\
              2.990318e-01, 8.921083e-01\\
              3.488704e-01, 8.901899e-01\\
              3.987090e-01, 8.886189e-01\\
              4.485477e-01, 8.873383e-01\\
              4.983863e-01, 8.862627e-01\\
              5.482249e-01, 8.853725e-01\\
              5.980635e-01, 8.845747e-01\\
              6.479022e-01, 8.838809e-01\\
              6.977408e-01, 8.833429e-01\\
              7.475794e-01, 8.828679e-01\\
              7.974181e-01, 8.824547e-01\\
              8.472567e-01, 8.820663e-01\\
              8.970953e-01, 8.817338e-01\\
              9.469339e-01, 8.814317e-01\\
              9.967726e-01, 8.811839e-01\\
              1.046611e+00, 8.809306e-01\\
              1.096450e+00, 8.807073e-01\\
              1.146288e+00, 8.804986e-01\\
              1.196127e+00, 8.803065e-01\\
              1.245966e+00, 8.801318e-01\\
              1.295804e+00, 8.799607e-01\\
              1.345643e+00, 8.798044e-01\\
              1.395482e+00, 8.796532e-01\\
              1.445320e+00, 8.795046e-01\\
              1.495159e+00, 8.793543e-01\\
              1.544997e+00, 8.792305e-01\\
              1.594836e+00, 8.790968e-01\\
              1.644675e+00, 8.789872e-01\\
              1.694513e+00, 8.788886e-01\\
              1.744352e+00, 8.787908e-01\\
              1.794191e+00, 8.787030e-01\\
              1.844029e+00, 8.786144e-01\\
              1.893868e+00, 8.785239e-01\\
              1.943707e+00, 8.784465e-01\\
              1.993545e+00, 8.783657e-01\\
              2.043384e+00, 8.782954e-01\\
              2.093222e+00, 8.782292e-01\\
              2.143061e+00, 8.781555e-01\\
              2.192900e+00, 8.780917e-01\\
              2.242738e+00, 8.780258e-01\\
              2.292577e+00, 8.779551e-01\\
              2.342416e+00, 8.778911e-01\\
              2.392254e+00, 8.778313e-01\\
              2.442093e+00, 8.777728e-01\\
              2.491931e+00, 8.777164e-01\\
              2.541770e+00, 8.776662e-01\\
              2.591609e+00, 8.776097e-01\\
              2.641447e+00, 8.775612e-01\\
              2.691286e+00, 8.775122e-01\\
              2.741125e+00, 8.774668e-01\\
              2.790963e+00, 8.774238e-01\\
              2.840802e+00, 8.773810e-01\\
              2.890640e+00, 8.773403e-01\\
              2.940479e+00, 8.773061e-01\\
              2.990318e+00, 8.772673e-01\\
              3.040156e+00, 8.772295e-01\\
              3.089995e+00, 8.771971e-01\\
              3.139834e+00, 8.771621e-01\\
              3.189672e+00, 8.771277e-01\\
              3.239511e+00, 8.770952e-01\\
              3.289349e+00, 8.770611e-01\\
              3.339188e+00, 8.770252e-01\\
              3.389027e+00, 8.769926e-01\\
              3.438865e+00, 8.769611e-01\\
              3.488704e+00, 8.769307e-01\\
              3.538543e+00, 8.768984e-01\\
              3.588381e+00, 8.768699e-01\\
              3.638220e+00, 8.768423e-01\\
              3.688058e+00, 8.768128e-01\\
              3.737897e+00, 8.767874e-01\\
              3.787736e+00, 8.767611e-01\\
              3.837574e+00, 8.767336e-01\\
              3.887413e+00, 8.767087e-01\\
              3.937252e+00, 8.766850e-01\\
              3.987090e+00, 8.766625e-01\\
              4.036929e+00, 8.766395e-01\\
              4.086768e+00, 8.766167e-01\\
              4.136606e+00, 8.765972e-01\\
              4.186445e+00, 8.765755e-01\\
              4.236283e+00, 8.765533e-01\\
              4.286122e+00, 8.765323e-01\\
              4.335961e+00, 8.765105e-01\\
              4.385799e+00, 8.764896e-01\\
              4.435638e+00, 8.764697e-01\\
              4.485477e+00, 8.764505e-01\\
              4.535315e+00, 8.764310e-01\\
              4.585154e+00, 8.764116e-01\\
              4.634992e+00, 8.763948e-01\\
              4.684831e+00, 8.763768e-01\\
              4.734670e+00, 8.763574e-01\\
              4.784508e+00, 8.763389e-01\\
              4.834347e+00, 8.763204e-01\\
              4.884186e+00, 8.763015e-01\\
              4.934024e+00, 8.762810e-01\\
              4.983863e+00, 8.762633e-01\\
              5.033701e+00, 8.762460e-01\\
              5.083540e+00, 8.762276e-01\\
              5.133379e+00, 8.762134e-01\\
              5.183217e+00, 8.761986e-01\\
              5.233056e+00, 8.761822e-01\\
              5.282895e+00, 8.761676e-01\\
              5.332733e+00, 8.761542e-01\\
              5.382572e+00, 8.761394e-01\\
              5.432410e+00, 8.761256e-01\\
              5.482249e+00, 8.761113e-01\\
              5.532088e+00, 8.760968e-01\\
              5.581926e+00, 8.760826e-01\\
              5.631765e+00, 8.760682e-01\\
              5.681604e+00, 8.760545e-01\\
              5.731442e+00, 8.760408e-01\\
              5.781281e+00, 8.760267e-01\\
              5.831120e+00, 8.760145e-01\\
              5.880958e+00, 8.760020e-01\\
              5.930797e+00, 8.759881e-01\\
              5.980635e+00, 8.759750e-01\\
              6.030474e+00, 8.759628e-01\\
              6.080313e+00, 8.759503e-01\\
              6.130151e+00, 8.759376e-01\\
              6.179990e+00, 8.759261e-01\\
              6.229829e+00, 8.759136e-01\\
              6.279667e+00, 8.759008e-01\\
              6.329506e+00, 8.758894e-01\\
              6.379344e+00, 8.758780e-01\\
              6.429183e+00, 8.758656e-01\\
              6.479022e+00, 8.758538e-01\\
              6.528860e+00, 8.758428e-01\\
              6.578699e+00, 8.758321e-01\\
              6.628538e+00, 8.758210e-01\\
              6.678376e+00, 8.758107e-01\\
              6.728215e+00, 8.758007e-01\\
              6.778053e+00, 8.757907e-01\\
              6.827892e+00, 8.757813e-01\\
              6.877731e+00, 8.757716e-01\\
              6.927569e+00, 8.757605e-01\\
              6.977408e+00, 8.757488e-01\\
              7.027247e+00, 8.757386e-01\\
              7.077085e+00, 8.757280e-01\\
              7.126924e+00, 8.757167e-01\\
              7.176762e+00, 8.757061e-01\\
              7.226601e+00, 8.756964e-01\\
              7.276440e+00, 8.756861e-01\\
              7.326278e+00, 8.756758e-01\\
              7.376117e+00, 8.756668e-01\\
              7.425956e+00, 8.756568e-01\\
              7.475794e+00, 8.756465e-01\\
              7.525633e+00, 8.756380e-01\\
              7.575471e+00, 8.756291e-01\\
              7.625310e+00, 8.756198e-01\\
              7.675149e+00, 8.756110e-01\\
              7.724987e+00, 8.756025e-01\\
              7.774826e+00, 8.755946e-01\\
              7.824665e+00, 8.755866e-01\\
              7.874503e+00, 8.755784e-01\\
              7.924342e+00, 8.755704e-01\\
              7.974181e+00, 8.755620e-01\\
              8.024019e+00, 8.755541e-01\\
              8.073858e+00, 8.755467e-01\\
              8.123696e+00, 8.755387e-01\\
              8.173535e+00, 8.755301e-01\\
              8.223374e+00, 8.755225e-01\\
              8.273212e+00, 8.755154e-01\\
              8.323051e+00, 8.755071e-01\\
              8.372890e+00, 8.754991e-01\\
              8.422728e+00, 8.754919e-01\\
              8.472567e+00, 8.754839e-01\\
              8.522405e+00, 8.754757e-01\\
              8.572244e+00, 8.754679e-01\\
              8.622083e+00, 8.754597e-01\\
              8.671921e+00, 8.754514e-01\\
              8.721760e+00, 8.754437e-01\\
              8.771599e+00, 8.754365e-01\\
              8.821437e+00, 8.754293e-01\\
              8.871276e+00, 8.754215e-01\\
              8.921114e+00, 8.754140e-01\\
              8.970953e+00, 8.754073e-01\\
              9.020792e+00, 8.754001e-01\\
              9.070630e+00, 8.753932e-01\\
              9.120469e+00, 8.753864e-01\\
              9.170308e+00, 8.753794e-01\\
              9.220146e+00, 8.753728e-01\\
              9.269985e+00, 8.753665e-01\\
              9.319823e+00, 8.753594e-01\\
              9.369662e+00, 8.753523e-01\\
              9.419501e+00, 8.753461e-01\\
              9.469339e+00, 8.753400e-01\\
              9.519178e+00, 8.753337e-01\\
              9.569017e+00, 8.753274e-01\\
              9.618855e+00, 8.753211e-01\\
              9.668694e+00, 8.753149e-01\\
              9.718533e+00, 8.753091e-01\\
              9.768371e+00, 8.753032e-01\\
              9.818210e+00, 8.752976e-01\\
              9.868048e+00, 8.752921e-01\\
              9.917887e+00, 8.752861e-01\\
              9.967726e+00, 8.752806e-01\\
              1.000000e+01, 8.752773e-01\\
          };
      \end{axis}
    \end{tikzpicture}%
  \end{minipage}
  \hfill
  \begin{minipage}{0.45\textwidth}
    \begin{tikzpicture}
      \begin{axis}[%
        width=0.75\textwidth,
        height=0.5\textwidth,
        scale only axis,
        xmin=0,xmax=10,
        ymin=0.85,ymax=1.02,
        legend style={at={(0.25,0.65)},anchor=west},
        ylabel={$\mathcal{E}(t) / \mathcal{E}(0)$}, xlabel={$t / t_{0}$},
        title={SFIM, split-side mortar elements}
        ]
          \addplot [color=blu,line width=1,dashdotted]
            table[col sep=comma, row sep=crcr]{%
              0.000000e+00, 1.000000e+00\\
              4.983863e-02, 1.000000e+00\\
              9.967726e-02, 9.999999e-01\\
              1.495159e-01, 9.999999e-01\\
              1.993545e-01, 9.999999e-01\\
              2.491931e-01, 9.999999e-01\\
              2.990318e-01, 9.999998e-01\\
              3.488704e-01, 9.999998e-01\\
              3.987090e-01, 9.999998e-01\\
              4.485477e-01, 9.999998e-01\\
              4.983863e-01, 9.999997e-01\\
              5.482249e-01, 9.999997e-01\\
              5.980635e-01, 9.999997e-01\\
              6.479022e-01, 9.999997e-01\\
              6.977408e-01, 9.999997e-01\\
              7.475794e-01, 9.999996e-01\\
              7.974181e-01, 9.999996e-01\\
              8.472567e-01, 9.999996e-01\\
              8.970953e-01, 9.999996e-01\\
              9.469339e-01, 9.999996e-01\\
              9.967726e-01, 9.999996e-01\\
              1.046611e+00, 9.999995e-01\\
              1.096450e+00, 9.999995e-01\\
              1.146288e+00, 9.999995e-01\\
              1.196127e+00, 9.999995e-01\\
              1.245966e+00, 9.999995e-01\\
              1.295804e+00, 9.999995e-01\\
              1.345643e+00, 9.999994e-01\\
              1.395482e+00, 9.999994e-01\\
              1.445320e+00, 9.999994e-01\\
              1.495159e+00, 9.999994e-01\\
              1.544997e+00, 9.999994e-01\\
              1.594836e+00, 9.999994e-01\\
              1.644675e+00, 9.999994e-01\\
              1.694513e+00, 9.999993e-01\\
              1.744352e+00, 9.999993e-01\\
              1.794191e+00, 9.999993e-01\\
              1.844029e+00, 9.999993e-01\\
              1.893868e+00, 9.999993e-01\\
              1.943707e+00, 9.999993e-01\\
              1.993545e+00, 9.999993e-01\\
              2.043384e+00, 9.999992e-01\\
              2.093222e+00, 9.999992e-01\\
              2.143061e+00, 9.999992e-01\\
              2.192900e+00, 9.999992e-01\\
              2.242738e+00, 9.999992e-01\\
              2.292577e+00, 9.999992e-01\\
              2.342416e+00, 9.999992e-01\\
              2.392254e+00, 9.999992e-01\\
              2.442093e+00, 9.999991e-01\\
              2.491931e+00, 9.999991e-01\\
              2.541770e+00, 9.999991e-01\\
              2.591609e+00, 9.999991e-01\\
              2.641447e+00, 9.999991e-01\\
              2.691286e+00, 9.999991e-01\\
              2.741125e+00, 9.999991e-01\\
              2.790963e+00, 9.999991e-01\\
              2.840802e+00, 9.999990e-01\\
              2.890640e+00, 9.999990e-01\\
              2.940479e+00, 9.999990e-01\\
              2.990318e+00, 9.999990e-01\\
              3.040156e+00, 9.999990e-01\\
              3.089995e+00, 9.999990e-01\\
              3.139834e+00, 9.999990e-01\\
              3.189672e+00, 9.999990e-01\\
              3.239511e+00, 9.999990e-01\\
              3.289349e+00, 9.999989e-01\\
              3.339188e+00, 9.999989e-01\\
              3.389027e+00, 9.999989e-01\\
              3.438865e+00, 9.999989e-01\\
              3.488704e+00, 9.999989e-01\\
              3.538543e+00, 9.999989e-01\\
              3.588381e+00, 9.999989e-01\\
              3.638220e+00, 9.999989e-01\\
              3.688058e+00, 9.999989e-01\\
              3.737897e+00, 9.999988e-01\\
              3.787736e+00, 9.999988e-01\\
              3.837574e+00, 9.999988e-01\\
              3.887413e+00, 9.999988e-01\\
              3.937252e+00, 9.999988e-01\\
              3.987090e+00, 9.999988e-01\\
              4.036929e+00, 9.999988e-01\\
              4.086768e+00, 9.999988e-01\\
              4.136606e+00, 9.999988e-01\\
              4.186445e+00, 9.999987e-01\\
              4.236283e+00, 9.999987e-01\\
              4.286122e+00, 9.999987e-01\\
              4.335961e+00, 9.999987e-01\\
              4.385799e+00, 9.999987e-01\\
              4.435638e+00, 9.999987e-01\\
              4.485477e+00, 9.999987e-01\\
              4.535315e+00, 9.999987e-01\\
              4.585154e+00, 9.999987e-01\\
              4.634992e+00, 9.999986e-01\\
              4.684831e+00, 9.999986e-01\\
              4.734670e+00, 9.999986e-01\\
              4.784508e+00, 9.999986e-01\\
              4.834347e+00, 9.999986e-01\\
              4.884186e+00, 9.999986e-01\\
              4.934024e+00, 9.999986e-01\\
              4.983863e+00, 9.999986e-01\\
              5.033701e+00, 9.999986e-01\\
              5.083540e+00, 9.999986e-01\\
              5.133379e+00, 9.999985e-01\\
              5.183217e+00, 9.999985e-01\\
              5.233056e+00, 9.999985e-01\\
              5.282895e+00, 9.999985e-01\\
              5.332733e+00, 9.999985e-01\\
              5.382572e+00, 9.999985e-01\\
              5.432410e+00, 9.999985e-01\\
              5.482249e+00, 9.999985e-01\\
              5.532088e+00, 9.999985e-01\\
              5.581926e+00, 9.999985e-01\\
              5.631765e+00, 9.999984e-01\\
              5.681604e+00, 9.999984e-01\\
              5.731442e+00, 9.999984e-01\\
              5.781281e+00, 9.999984e-01\\
              5.831120e+00, 9.999984e-01\\
              5.880958e+00, 9.999984e-01\\
              5.930797e+00, 9.999984e-01\\
              5.980635e+00, 9.999984e-01\\
              6.030474e+00, 9.999984e-01\\
              6.080313e+00, 9.999984e-01\\
              6.130151e+00, 9.999984e-01\\
              6.179990e+00, 9.999983e-01\\
              6.229829e+00, 9.999983e-01\\
              6.279667e+00, 9.999983e-01\\
              6.329506e+00, 9.999983e-01\\
              6.379344e+00, 9.999983e-01\\
              6.429183e+00, 9.999983e-01\\
              6.479022e+00, 9.999983e-01\\
              6.528860e+00, 9.999983e-01\\
              6.578699e+00, 9.999983e-01\\
              6.628538e+00, 9.999983e-01\\
              6.678376e+00, 9.999982e-01\\
              6.728215e+00, 9.999982e-01\\
              6.778053e+00, 9.999982e-01\\
              6.827892e+00, 9.999982e-01\\
              6.877731e+00, 9.999982e-01\\
              6.927569e+00, 9.999982e-01\\
              6.977408e+00, 9.999982e-01\\
              7.027247e+00, 9.999982e-01\\
              7.077085e+00, 9.999982e-01\\
              7.126924e+00, 9.999982e-01\\
              7.176762e+00, 9.999982e-01\\
              7.226601e+00, 9.999981e-01\\
              7.276440e+00, 9.999981e-01\\
              7.326278e+00, 9.999981e-01\\
              7.376117e+00, 9.999981e-01\\
              7.425956e+00, 9.999981e-01\\
              7.475794e+00, 9.999981e-01\\
              7.525633e+00, 9.999981e-01\\
              7.575471e+00, 9.999981e-01\\
              7.625310e+00, 9.999981e-01\\
              7.675149e+00, 9.999981e-01\\
              7.724987e+00, 9.999981e-01\\
              7.774826e+00, 9.999981e-01\\
              7.824665e+00, 9.999980e-01\\
              7.874503e+00, 9.999980e-01\\
              7.924342e+00, 9.999980e-01\\
              7.974181e+00, 9.999980e-01\\
              8.024019e+00, 9.999980e-01\\
              8.073858e+00, 9.999980e-01\\
              8.123696e+00, 9.999980e-01\\
              8.173535e+00, 9.999980e-01\\
              8.223374e+00, 9.999980e-01\\
              8.273212e+00, 9.999980e-01\\
              8.323051e+00, 9.999980e-01\\
              8.372890e+00, 9.999980e-01\\
              8.422728e+00, 9.999979e-01\\
              8.472567e+00, 9.999979e-01\\
              8.522405e+00, 9.999979e-01\\
              8.572244e+00, 9.999979e-01\\
              8.622083e+00, 9.999979e-01\\
              8.671921e+00, 9.999979e-01\\
              8.721760e+00, 9.999979e-01\\
              8.771599e+00, 9.999979e-01\\
              8.821437e+00, 9.999979e-01\\
              8.871276e+00, 9.999979e-01\\
              8.921114e+00, 9.999979e-01\\
              8.970953e+00, 9.999979e-01\\
              9.020792e+00, 9.999978e-01\\
              9.070630e+00, 9.999978e-01\\
              9.120469e+00, 9.999978e-01\\
              9.170308e+00, 9.999978e-01\\
              9.220146e+00, 9.999978e-01\\
              9.269985e+00, 9.999978e-01\\
              9.319823e+00, 9.999978e-01\\
              9.369662e+00, 9.999978e-01\\
              9.419501e+00, 9.999978e-01\\
              9.469339e+00, 9.999978e-01\\
              9.519178e+00, 9.999978e-01\\
              9.569017e+00, 9.999978e-01\\
              9.618855e+00, 9.999977e-01\\
              9.668694e+00, 9.999977e-01\\
              9.718533e+00, 9.999977e-01\\
              9.768371e+00, 9.999977e-01\\
              9.818210e+00, 9.999977e-01\\
              9.868048e+00, 9.999977e-01\\
              9.917887e+00, 9.999977e-01\\
              9.967726e+00, 9.999977e-01\\
              1.000000e+01, 9.999977e-01\\
          };
          \addplot [color=red,line width=1,solid]
            table[col sep=comma, row sep=crcr]{%
              0.000000e+00, 1.000000e+00\\
              4.983863e-02, 9.390835e-01\\
              9.967726e-02, 9.187927e-01\\
              1.495159e-01, 9.090437e-01\\
              1.993545e-01, 9.036967e-01\\
              2.491931e-01, 9.000416e-01\\
              2.990318e-01, 8.974496e-01\\
              3.488704e-01, 8.955152e-01\\
              3.987090e-01, 8.939468e-01\\
              4.485477e-01, 8.927517e-01\\
              4.983863e-01, 8.917192e-01\\
              5.482249e-01, 8.908857e-01\\
              5.980635e-01, 8.901131e-01\\
              6.479022e-01, 8.894729e-01\\
              6.977408e-01, 8.889956e-01\\
              7.475794e-01, 8.885590e-01\\
              7.974181e-01, 8.881888e-01\\
              8.472567e-01, 8.878342e-01\\
              8.970953e-01, 8.875359e-01\\
              9.469339e-01, 8.872640e-01\\
              9.967726e-01, 8.870392e-01\\
              1.046611e+00, 8.868091e-01\\
              1.096450e+00, 8.866037e-01\\
              1.146288e+00, 8.864220e-01\\
              1.196127e+00, 8.862530e-01\\
              1.245966e+00, 8.861014e-01\\
              1.295804e+00, 8.859529e-01\\
              1.345643e+00, 8.858156e-01\\
              1.395482e+00, 8.856844e-01\\
              1.445320e+00, 8.855548e-01\\
              1.495159e+00, 8.854238e-01\\
              1.544997e+00, 8.853160e-01\\
              1.594836e+00, 8.852008e-01\\
              1.644675e+00, 8.851050e-01\\
              1.694513e+00, 8.850247e-01\\
              1.744352e+00, 8.849396e-01\\
              1.794191e+00, 8.848653e-01\\
              1.844029e+00, 8.847900e-01\\
              1.893868e+00, 8.847126e-01\\
              1.943707e+00, 8.846449e-01\\
              1.993545e+00, 8.845734e-01\\
              2.043384e+00, 8.845085e-01\\
              2.093222e+00, 8.844517e-01\\
              2.143061e+00, 8.843872e-01\\
              2.192900e+00, 8.843332e-01\\
              2.242738e+00, 8.842751e-01\\
              2.292577e+00, 8.842137e-01\\
              2.342416e+00, 8.841579e-01\\
              2.392254e+00, 8.841051e-01\\
              2.442093e+00, 8.840549e-01\\
              2.491931e+00, 8.840052e-01\\
              2.541770e+00, 8.839633e-01\\
              2.591609e+00, 8.839152e-01\\
              2.641447e+00, 8.838738e-01\\
              2.691286e+00, 8.838317e-01\\
              2.741125e+00, 8.837931e-01\\
              2.790963e+00, 8.837563e-01\\
              2.840802e+00, 8.837182e-01\\
              2.890640e+00, 8.836828e-01\\
              2.940479e+00, 8.836522e-01\\
              2.990318e+00, 8.836187e-01\\
              3.040156e+00, 8.835857e-01\\
              3.089995e+00, 8.835582e-01\\
              3.139834e+00, 8.835274e-01\\
              3.189672e+00, 8.834976e-01\\
              3.239511e+00, 8.834695e-01\\
              3.289349e+00, 8.834403e-01\\
              3.339188e+00, 8.834088e-01\\
              3.389027e+00, 8.833807e-01\\
              3.438865e+00, 8.833530e-01\\
              3.488704e+00, 8.833260e-01\\
              3.538543e+00, 8.832969e-01\\
              3.588381e+00, 8.832727e-01\\
              3.638220e+00, 8.832495e-01\\
              3.688058e+00, 8.832238e-01\\
              3.737897e+00, 8.832019e-01\\
              3.787736e+00, 8.831795e-01\\
              3.837574e+00, 8.831559e-01\\
              3.887413e+00, 8.831348e-01\\
              3.937252e+00, 8.831139e-01\\
              3.987090e+00, 8.830941e-01\\
              4.036929e+00, 8.830744e-01\\
              4.086768e+00, 8.830548e-01\\
              4.136606e+00, 8.830381e-01\\
              4.186445e+00, 8.830197e-01\\
              4.236283e+00, 8.830005e-01\\
              4.286122e+00, 8.829826e-01\\
              4.335961e+00, 8.829643e-01\\
              4.385799e+00, 8.829457e-01\\
              4.435638e+00, 8.829283e-01\\
              4.485477e+00, 8.829120e-01\\
              4.535315e+00, 8.828951e-01\\
              4.585154e+00, 8.828783e-01\\
              4.634992e+00, 8.828636e-01\\
              4.684831e+00, 8.828481e-01\\
              4.734670e+00, 8.828308e-01\\
              4.784508e+00, 8.828143e-01\\
              4.834347e+00, 8.827976e-01\\
              4.884186e+00, 8.827806e-01\\
              4.934024e+00, 8.827624e-01\\
              4.983863e+00, 8.827466e-01\\
              5.033701e+00, 8.827315e-01\\
              5.083540e+00, 8.827154e-01\\
              5.133379e+00, 8.827033e-01\\
              5.183217e+00, 8.826911e-01\\
              5.233056e+00, 8.826771e-01\\
              5.282895e+00, 8.826649e-01\\
              5.332733e+00, 8.826539e-01\\
              5.382572e+00, 8.826413e-01\\
              5.432410e+00, 8.826297e-01\\
              5.482249e+00, 8.826177e-01\\
              5.532088e+00, 8.826056e-01\\
              5.581926e+00, 8.825936e-01\\
              5.631765e+00, 8.825813e-01\\
              5.681604e+00, 8.825697e-01\\
              5.731442e+00, 8.825579e-01\\
              5.781281e+00, 8.825454e-01\\
              5.831120e+00, 8.825350e-01\\
              5.880958e+00, 8.825246e-01\\
              5.930797e+00, 8.825124e-01\\
              5.980635e+00, 8.825013e-01\\
              6.030474e+00, 8.824911e-01\\
              6.080313e+00, 8.824807e-01\\
              6.130151e+00, 8.824697e-01\\
              6.179990e+00, 8.824600e-01\\
              6.229829e+00, 8.824494e-01\\
              6.279667e+00, 8.824382e-01\\
              6.329506e+00, 8.824283e-01\\
              6.379344e+00, 8.824186e-01\\
              6.429183e+00, 8.824078e-01\\
              6.479022e+00, 8.823975e-01\\
              6.528860e+00, 8.823882e-01\\
              6.578699e+00, 8.823794e-01\\
              6.628538e+00, 8.823700e-01\\
              6.678376e+00, 8.823612e-01\\
              6.728215e+00, 8.823531e-01\\
              6.778053e+00, 8.823447e-01\\
              6.827892e+00, 8.823368e-01\\
              6.877731e+00, 8.823289e-01\\
              6.927569e+00, 8.823199e-01\\
              6.977408e+00, 8.823100e-01\\
              7.027247e+00, 8.823013e-01\\
              7.077085e+00, 8.822922e-01\\
              7.126924e+00, 8.822825e-01\\
              7.176762e+00, 8.822733e-01\\
              7.226601e+00, 8.822651e-01\\
              7.276440e+00, 8.822564e-01\\
              7.326278e+00, 8.822476e-01\\
              7.376117e+00, 8.822399e-01\\
              7.425956e+00, 8.822315e-01\\
              7.475794e+00, 8.822225e-01\\
              7.525633e+00, 8.822154e-01\\
              7.575471e+00, 8.822080e-01\\
              7.625310e+00, 8.822002e-01\\
              7.675149e+00, 8.821929e-01\\
              7.724987e+00, 8.821857e-01\\
              7.774826e+00, 8.821791e-01\\
              7.824665e+00, 8.821724e-01\\
              7.874503e+00, 8.821655e-01\\
              7.924342e+00, 8.821589e-01\\
              7.974181e+00, 8.821522e-01\\
              8.024019e+00, 8.821457e-01\\
              8.073858e+00, 8.821395e-01\\
              8.123696e+00, 8.821326e-01\\
              8.173535e+00, 8.821253e-01\\
              8.223374e+00, 8.821189e-01\\
              8.273212e+00, 8.821133e-01\\
              8.323051e+00, 8.821065e-01\\
              8.372890e+00, 8.820997e-01\\
              8.422728e+00, 8.820939e-01\\
              8.472567e+00, 8.820874e-01\\
              8.522405e+00, 8.820802e-01\\
              8.572244e+00, 8.820738e-01\\
              8.622083e+00, 8.820671e-01\\
              8.671921e+00, 8.820599e-01\\
              8.721760e+00, 8.820532e-01\\
              8.771599e+00, 8.820472e-01\\
              8.821437e+00, 8.820411e-01\\
              8.871276e+00, 8.820344e-01\\
              8.921114e+00, 8.820282e-01\\
              8.970953e+00, 8.820229e-01\\
              9.020792e+00, 8.820170e-01\\
              9.070630e+00, 8.820111e-01\\
              9.120469e+00, 8.820056e-01\\
              9.170308e+00, 8.819995e-01\\
              9.220146e+00, 8.819939e-01\\
              9.269985e+00, 8.819887e-01\\
              9.319823e+00, 8.819828e-01\\
              9.369662e+00, 8.819769e-01\\
              9.419501e+00, 8.819718e-01\\
              9.469339e+00, 8.819667e-01\\
              9.519178e+00, 8.819615e-01\\
              9.569017e+00, 8.819562e-01\\
              9.618855e+00, 8.819510e-01\\
              9.668694e+00, 8.819461e-01\\
              9.718533e+00, 8.819415e-01\\
              9.768371e+00, 8.819366e-01\\
              9.818210e+00, 8.819320e-01\\
              9.868048e+00, 8.819277e-01\\
              9.917887e+00, 8.819228e-01\\
              9.967726e+00, 8.819185e-01\\
              1.000000e+01, 8.819160e-01\\
          };
        \legend{Central ($\alpha = 0$)\\Upwind ($\alpha = 1$)\\};
      \end{axis}
    \end{tikzpicture}%
  \end{minipage}
  \caption{Energy dissipation in a pseudorandomly-heterogeneous curvilinear box
  comparing SFIM with interpolated geometry using the full-side and split-side
  mortar elements with the upwind ($\alpha = 1$) and central ($\alpha = 0$)
  flux.\label{fig:energy:skew}}
\end{figure}
\fref{fig:energy:skew} shows the energy in the solution versus time for SFIM
with both full-side and split-side mortar elements using both the upwind and
central fluxes when the coordinate transform~\eref{eqn:curve} is interpolated.
As can be seen, the upwind schemes quickly
dissipate energy of the unresolved modes in the solution and then remain stable.
The central method preserves the
initial energy; there is an $\sim 10^{-4}\%$ energy lose for both central
schemes but this is likely from dissipation of the Runge-Kutta scheme. This
long time energy plot also suggests that with an upwind flux SFIM with
split-side mortar elements is slightly more dissipative than with full-side
mortar elements, with the split-side mortar elements dissipating $13\%$ of the
initial energy as compared with $18\%$ with full-side mortar elements.  That
said the impacts of this on a real problem would depend on the
distribution of the energy across the eigenmodes of the solution. In
\fref{fig:energy:skew} only the interpolate curvilinear transform results are
shown, and similar results are seen with the watertight and continuous metric
handling of the geometry.

\begin{figure}
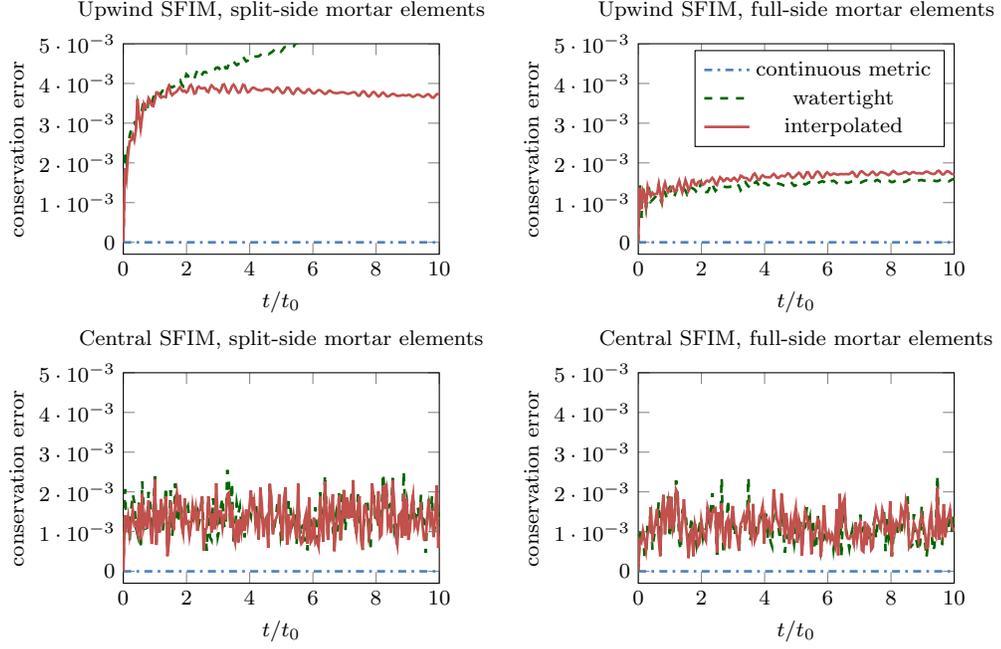

  \begin{minipage}{0.45\textwidth}
%
  \end{minipage}
  \caption{Conservation error for both the full-side and split-side mortar with
  both upwinding ($\alpha = 1$) and central flux ($\alpha = 0$) for the
  continuous metric term, watertight, and interpolated handling of geometry.
  The conservation error is defined by~\eref{eqn:con:err}.\label{fig:conservative:box} }
\end{figure}
In order to test the conservation properties of SFIM, we run the same test
problem using the the interpolated, watertight, and continuous metric handling
of geometry with both types of mortar elements using the central
and upwind fluxes. The conservation error for all these runs is shown in
\fref{fig:conservative:box}. We define the conservation error as
\begin{align}
  \label{eqn:con:err}
  \mbox{conservation error} =
  \sum_{i}e_{\rho v_{i}} +
  \sum_{i,j}e_{\epsilon_{kl}},
\end{align}
with the error in each component defined as
\begin{align}
  e_{\rho v_{i}} &= \left|\frac{\int_{\Omega} \rho (v_{i} - v_{i}^{0})}
                          {\int_{\Omega} \rho v_{i}^{0}}\right|,&
  e_{\epsilon_{ij}} &=
    \left|\frac{\int_{\Omega} S_{ijkl} (\sigma_{kl} - \sigma_{kl}^{0})}
               {\int_{\Omega} S_{ijkl} \sigma_{kl}^{0}}\right|.
\end{align}
Here $v_{i}^{0}$ and $\sigma_{kl}^{0}$ are the initial values and the integrals
are evaluated using the above outlined quadrature approximations. Thus, the
conservation error is the sum of the normalized conservation error for each
conserved variable. As can be seen when the metric terms are made continuous the
scheme is conservative.

\begin{figure}
  \begin{minipage}{0.45\textwidth}
    \begin{tikzpicture}
      \begin{axis}[%
        width=0.75\textwidth,
        height=0.5\textwidth,
        scale only axis,
        xmin=0,xmax=10,
        ymin=-1e-4,ymax=4.5e-3,
        ytick={0,1e-3,2e-3,3e-3,4e-3,5e-3},
        legend style={at={(0.3,0.6)},anchor=west},
        title={Upwind SFIM, split-side mortar elements},
        ylabel={conservation error}, xlabel={$t / t_{0}$},
        scaled ticks=false,
        ]
          \addplot [color=blu,line width=1,dashdotted]
            table[col sep=comma, row sep=crcr]{%
              0.000000e+00, 0.000000e+00\\
              4.983863e-02, 1.862767e-03\\
              9.967726e-02, 1.605702e-03\\
              1.495159e-01, 2.192159e-03\\
              1.993545e-01, 2.538455e-03\\
              2.491931e-01, 2.568725e-03\\
              2.990318e-01, 2.702544e-03\\
              3.488704e-01, 2.566996e-03\\
              3.987090e-01, 2.829807e-03\\
              4.485477e-01, 3.607648e-03\\
              4.983863e-01, 3.176872e-03\\
              5.482249e-01, 2.862528e-03\\
              5.980635e-01, 3.178069e-03\\
              6.479022e-01, 3.482622e-03\\
              6.977408e-01, 3.383720e-03\\
              7.475794e-01, 3.411476e-03\\
              7.974181e-01, 3.630342e-03\\
              8.472567e-01, 3.568139e-03\\
              8.970953e-01, 3.448280e-03\\
              9.469339e-01, 3.428787e-03\\
              9.967726e-01, 3.582393e-03\\
              1.046611e+00, 3.775173e-03\\
              1.096450e+00, 3.769263e-03\\
              1.146288e+00, 3.682701e-03\\
              1.196127e+00, 3.643076e-03\\
              1.245966e+00, 3.670850e-03\\
              1.295804e+00, 3.668415e-03\\
              1.345643e+00, 3.654633e-03\\
              1.395482e+00, 3.859416e-03\\
              1.445320e+00, 3.933532e-03\\
              1.495159e+00, 3.759817e-03\\
              1.544997e+00, 3.704454e-03\\
              1.594836e+00, 3.749310e-03\\
              1.644675e+00, 3.771167e-03\\
              1.694513e+00, 3.801864e-03\\
              1.744352e+00, 3.862501e-03\\
              1.794191e+00, 3.908067e-03\\
              1.844029e+00, 3.831252e-03\\
              1.893868e+00, 3.798459e-03\\
              1.943707e+00, 3.815682e-03\\
              1.993545e+00, 3.773730e-03\\
              2.043384e+00, 3.810477e-03\\
              2.093222e+00, 3.925631e-03\\
              2.143061e+00, 3.964591e-03\\
              2.192900e+00, 3.879895e-03\\
              2.242738e+00, 3.744995e-03\\
              2.292577e+00, 3.760898e-03\\
              2.342416e+00, 3.891765e-03\\
              2.392254e+00, 3.921929e-03\\
              2.442093e+00, 3.882274e-03\\
              2.491931e+00, 3.873475e-03\\
              2.541770e+00, 3.889215e-03\\
              2.591609e+00, 3.881596e-03\\
              2.641447e+00, 3.804981e-03\\
              2.691286e+00, 3.782716e-03\\
              2.741125e+00, 3.886670e-03\\
              2.790963e+00, 3.953185e-03\\
              2.840802e+00, 3.914928e-03\\
              2.890640e+00, 3.865534e-03\\
              2.940479e+00, 3.863199e-03\\
              2.990318e+00, 3.838194e-03\\
              3.040156e+00, 3.780309e-03\\
              3.089995e+00, 3.836991e-03\\
              3.139834e+00, 3.964281e-03\\
              3.189672e+00, 3.978938e-03\\
              3.239511e+00, 3.885122e-03\\
              3.289349e+00, 3.784099e-03\\
              3.339188e+00, 3.760464e-03\\
              3.389027e+00, 3.831378e-03\\
              3.438865e+00, 3.929350e-03\\
              3.488704e+00, 3.971899e-03\\
              3.538543e+00, 3.918938e-03\\
              3.588381e+00, 3.825909e-03\\
              3.638220e+00, 3.767764e-03\\
              3.688058e+00, 3.774504e-03\\
              3.737897e+00, 3.867198e-03\\
              3.787736e+00, 3.969946e-03\\
              3.837574e+00, 3.951841e-03\\
              3.887413e+00, 3.843381e-03\\
              3.937252e+00, 3.775914e-03\\
              3.987090e+00, 3.792325e-03\\
              4.036929e+00, 3.848731e-03\\
              4.086768e+00, 3.890179e-03\\
              4.136606e+00, 3.893883e-03\\
              4.186445e+00, 3.867788e-03\\
              4.236283e+00, 3.842597e-03\\
              4.286122e+00, 3.837597e-03\\
              4.335961e+00, 3.828440e-03\\
              4.385799e+00, 3.820323e-03\\
              4.435638e+00, 3.840959e-03\\
              4.485477e+00, 3.861709e-03\\
              4.535315e+00, 3.861175e-03\\
              4.585154e+00, 3.859989e-03\\
              4.634992e+00, 3.852179e-03\\
              4.684831e+00, 3.812925e-03\\
              4.734670e+00, 3.776934e-03\\
              4.784508e+00, 3.808613e-03\\
              4.834347e+00, 3.878531e-03\\
              4.884186e+00, 3.900302e-03\\
              4.934024e+00, 3.860394e-03\\
              4.983863e+00, 3.798486e-03\\
              5.033701e+00, 3.754741e-03\\
              5.083540e+00, 3.768594e-03\\
              5.133379e+00, 3.836614e-03\\
              5.183217e+00, 3.896596e-03\\
              5.233056e+00, 3.894008e-03\\
              5.282895e+00, 3.831501e-03\\
              5.332733e+00, 3.757581e-03\\
              5.382572e+00, 3.733994e-03\\
              5.432410e+00, 3.786978e-03\\
              5.482249e+00, 3.861557e-03\\
              5.532088e+00, 3.877773e-03\\
              5.581926e+00, 3.838447e-03\\
              5.631765e+00, 3.795853e-03\\
              5.681604e+00, 3.765227e-03\\
              5.731442e+00, 3.756371e-03\\
              5.781281e+00, 3.792396e-03\\
              5.831120e+00, 3.847067e-03\\
              5.880958e+00, 3.856909e-03\\
              5.930797e+00, 3.812388e-03\\
              5.980635e+00, 3.763414e-03\\
              6.030474e+00, 3.746209e-03\\
              6.080313e+00, 3.765635e-03\\
              6.130151e+00, 3.807259e-03\\
              6.179990e+00, 3.835361e-03\\
              6.229829e+00, 3.827824e-03\\
              6.279667e+00, 3.797460e-03\\
              6.329506e+00, 3.760520e-03\\
              6.379344e+00, 3.734841e-03\\
              6.429183e+00, 3.749691e-03\\
              6.479022e+00, 3.797897e-03\\
              6.528860e+00, 3.822739e-03\\
              6.578699e+00, 3.800496e-03\\
              6.628538e+00, 3.771589e-03\\
              6.678376e+00, 3.762553e-03\\
              6.728215e+00, 3.758475e-03\\
              6.778053e+00, 3.756037e-03\\
              6.827892e+00, 3.764788e-03\\
              6.877731e+00, 3.776185e-03\\
              6.927569e+00, 3.779864e-03\\
              6.977408e+00, 3.775420e-03\\
              7.027247e+00, 3.760841e-03\\
              7.077085e+00, 3.743340e-03\\
              7.126924e+00, 3.740702e-03\\
              7.176762e+00, 3.753056e-03\\
              7.226601e+00, 3.767202e-03\\
              7.276440e+00, 3.780393e-03\\
              7.326278e+00, 3.779895e-03\\
              7.376117e+00, 3.742789e-03\\
              7.425956e+00, 3.694018e-03\\
              7.475794e+00, 3.695028e-03\\
              7.525633e+00, 3.751279e-03\\
              7.575471e+00, 3.801714e-03\\
              7.625310e+00, 3.799686e-03\\
              7.675149e+00, 3.750366e-03\\
              7.724987e+00, 3.691476e-03\\
              7.774826e+00, 3.670751e-03\\
              7.824665e+00, 3.707381e-03\\
              7.874503e+00, 3.763948e-03\\
              7.924342e+00, 3.786611e-03\\
              7.974181e+00, 3.759259e-03\\
              8.024019e+00, 3.705301e-03\\
              8.073858e+00, 3.670216e-03\\
              8.123696e+00, 3.691793e-03\\
              8.173535e+00, 3.748747e-03\\
              8.223374e+00, 3.772482e-03\\
              8.273212e+00, 3.737521e-03\\
              8.323051e+00, 3.689301e-03\\
              8.372890e+00, 3.672254e-03\\
              8.422728e+00, 3.688094e-03\\
              8.472567e+00, 3.718748e-03\\
              8.522405e+00, 3.740019e-03\\
              8.572244e+00, 3.732050e-03\\
              8.622083e+00, 3.705187e-03\\
              8.671921e+00, 3.687978e-03\\
              8.721760e+00, 3.688298e-03\\
              8.771599e+00, 3.694496e-03\\
              8.821437e+00, 3.706608e-03\\
              8.871276e+00, 3.708283e-03\\
              8.921114e+00, 3.696463e-03\\
              8.970953e+00, 3.696839e-03\\
              9.020792e+00, 3.714342e-03\\
              9.070630e+00, 3.711411e-03\\
              9.120469e+00, 3.680762e-03\\
              9.170308e+00, 3.663441e-03\\
              9.220146e+00, 3.680306e-03\\
              9.269985e+00, 3.707905e-03\\
              9.319823e+00, 3.720479e-03\\
              9.369662e+00, 3.712186e-03\\
              9.419501e+00, 3.687196e-03\\
              9.469339e+00, 3.663653e-03\\
              9.519178e+00, 3.666266e-03\\
              9.569017e+00, 3.696040e-03\\
              9.618855e+00, 3.725999e-03\\
              9.668694e+00, 3.726696e-03\\
              9.718533e+00, 3.690067e-03\\
              9.768371e+00, 3.645393e-03\\
              9.818210e+00, 3.641100e-03\\
              9.868048e+00, 3.685146e-03\\
              9.917887e+00, 3.729430e-03\\
              9.967726e+00, 3.733721e-03\\
              1.000000e+01, 3.717472e-03\\
          };
          \addplot [color=red,line width=1,solid]
            table[col sep=comma, row sep=crcr]{%
              0.000000e+00, 0.000000e+00\\
              4.983863e-02, 4.440892e-16\\
              9.967726e-02, 2.220446e-16\\
              1.495159e-01, 3.330669e-16\\
              1.993545e-01, 0.000000e+00\\
              2.491931e-01, 5.551115e-16\\
              2.990318e-01, 5.551115e-16\\
              3.488704e-01, 5.551115e-16\\
              3.987090e-01, 3.330669e-16\\
              4.485477e-01, 4.440892e-16\\
              4.983863e-01, 3.330669e-16\\
              5.482249e-01, 3.330669e-16\\
              5.980635e-01, 2.220446e-16\\
              6.479022e-01, 2.220446e-16\\
              6.977408e-01, 5.551115e-16\\
              7.475794e-01, 7.771561e-16\\
              7.974181e-01, 4.440892e-16\\
              8.472567e-01, 4.440892e-16\\
              8.970953e-01, 6.661338e-16\\
              9.469339e-01, 6.661338e-16\\
              9.967726e-01, 8.881784e-16\\
              1.046611e+00, 8.881784e-16\\
              1.096450e+00, 8.881784e-16\\
              1.146288e+00, 6.661338e-16\\
              1.196127e+00, 4.440892e-16\\
              1.245966e+00, 5.551115e-16\\
              1.295804e+00, 6.661338e-16\\
              1.345643e+00, 4.440892e-16\\
              1.395482e+00, 6.661338e-16\\
              1.445320e+00, 5.551115e-16\\
              1.495159e+00, 7.771561e-16\\
              1.544997e+00, 5.551115e-16\\
              1.594836e+00, 6.661338e-16\\
              1.644675e+00, 5.551115e-16\\
              1.694513e+00, 5.551115e-16\\
              1.744352e+00, 6.661338e-16\\
              1.794191e+00, 4.440892e-16\\
              1.844029e+00, 6.661338e-16\\
              1.893868e+00, 6.661338e-16\\
              1.943707e+00, 6.661338e-16\\
              1.993545e+00, 5.551115e-16\\
              2.043384e+00, 6.661338e-16\\
              2.093222e+00, 4.440892e-16\\
              2.143061e+00, 6.661338e-16\\
              2.192900e+00, 7.771561e-16\\
              2.242738e+00, 1.110223e-15\\
              2.292577e+00, 8.881784e-16\\
              2.342416e+00, 1.110223e-15\\
              2.392254e+00, 1.110223e-15\\
              2.442093e+00, 1.332268e-15\\
              2.491931e+00, 1.554312e-15\\
              2.541770e+00, 1.443290e-15\\
              2.591609e+00, 1.554312e-15\\
              2.641447e+00, 1.443290e-15\\
              2.691286e+00, 1.443290e-15\\
              2.741125e+00, 1.221245e-15\\
              2.790963e+00, 1.110223e-15\\
              2.840802e+00, 9.992007e-16\\
              2.890640e+00, 1.221245e-15\\
              2.940479e+00, 1.221245e-15\\
              2.990318e+00, 1.332268e-15\\
              3.040156e+00, 1.332268e-15\\
              3.089995e+00, 1.221245e-15\\
              3.139834e+00, 1.332268e-15\\
              3.189672e+00, 1.221245e-15\\
              3.239511e+00, 1.443290e-15\\
              3.289349e+00, 1.332268e-15\\
              3.339188e+00, 1.110223e-15\\
              3.389027e+00, 7.771561e-16\\
              3.438865e+00, 1.221245e-15\\
              3.488704e+00, 1.221245e-15\\
              3.538543e+00, 1.332268e-15\\
              3.588381e+00, 1.221245e-15\\
              3.638220e+00, 1.332268e-15\\
              3.688058e+00, 1.665335e-15\\
              3.737897e+00, 1.554312e-15\\
              3.787736e+00, 1.665335e-15\\
              3.837574e+00, 1.887379e-15\\
              3.887413e+00, 1.887379e-15\\
              3.937252e+00, 2.109424e-15\\
              3.987090e+00, 2.442491e-15\\
              4.036929e+00, 2.109424e-15\\
              4.086768e+00, 2.331468e-15\\
              4.136606e+00, 2.331468e-15\\
              4.186445e+00, 2.553513e-15\\
              4.236283e+00, 2.553513e-15\\
              4.286122e+00, 2.775558e-15\\
              4.335961e+00, 2.886580e-15\\
              4.385799e+00, 2.553513e-15\\
              4.435638e+00, 2.997602e-15\\
              4.485477e+00, 2.997602e-15\\
              4.535315e+00, 2.886580e-15\\
              4.585154e+00, 3.108624e-15\\
              4.634992e+00, 2.553513e-15\\
              4.684831e+00, 3.219647e-15\\
              4.734670e+00, 2.886580e-15\\
              4.784508e+00, 2.331468e-15\\
              4.834347e+00, 2.553513e-15\\
              4.884186e+00, 2.886580e-15\\
              4.934024e+00, 2.664535e-15\\
              4.983863e+00, 2.331468e-15\\
              5.033701e+00, 2.331468e-15\\
              5.083540e+00, 2.442491e-15\\
              5.133379e+00, 2.442491e-15\\
              5.183217e+00, 2.664535e-15\\
              5.233056e+00, 2.775558e-15\\
              5.282895e+00, 2.109424e-15\\
              5.332733e+00, 2.331468e-15\\
              5.382572e+00, 1.887379e-15\\
              5.432410e+00, 2.109424e-15\\
              5.482249e+00, 2.109424e-15\\
              5.532088e+00, 1.776357e-15\\
              5.581926e+00, 1.887379e-15\\
              5.631765e+00, 1.665335e-15\\
              5.681604e+00, 1.998401e-15\\
              5.731442e+00, 1.776357e-15\\
              5.781281e+00, 1.665335e-15\\
              5.831120e+00, 1.887379e-15\\
              5.880958e+00, 1.776357e-15\\
              5.930797e+00, 1.998401e-15\\
              5.980635e+00, 1.998401e-15\\
              6.030474e+00, 1.665335e-15\\
              6.080313e+00, 1.776357e-15\\
              6.130151e+00, 1.776357e-15\\
              6.179990e+00, 1.887379e-15\\
              6.229829e+00, 1.554312e-15\\
              6.279667e+00, 1.665335e-15\\
              6.329506e+00, 1.776357e-15\\
              6.379344e+00, 1.887379e-15\\
              6.429183e+00, 1.443290e-15\\
              6.479022e+00, 1.554312e-15\\
              6.528860e+00, 1.887379e-15\\
              6.578699e+00, 1.776357e-15\\
              6.628538e+00, 1.776357e-15\\
              6.678376e+00, 1.776357e-15\\
              6.728215e+00, 1.665335e-15\\
              6.778053e+00, 1.887379e-15\\
              6.827892e+00, 1.887379e-15\\
              6.877731e+00, 2.109424e-15\\
              6.927569e+00, 1.998401e-15\\
              6.977408e+00, 2.331468e-15\\
              7.027247e+00, 2.220446e-15\\
              7.077085e+00, 2.220446e-15\\
              7.126924e+00, 2.442491e-15\\
              7.176762e+00, 2.442491e-15\\
              7.226601e+00, 2.553513e-15\\
              7.276440e+00, 2.220446e-15\\
              7.326278e+00, 1.998401e-15\\
              7.376117e+00, 2.331468e-15\\
              7.425956e+00, 1.998401e-15\\
              7.475794e+00, 1.998401e-15\\
              7.525633e+00, 2.220446e-15\\
              7.575471e+00, 2.331468e-15\\
              7.625310e+00, 1.998401e-15\\
              7.675149e+00, 2.220446e-15\\
              7.724987e+00, 1.887379e-15\\
              7.774826e+00, 1.665335e-15\\
              7.824665e+00, 1.887379e-15\\
              7.874503e+00, 1.998401e-15\\
              7.924342e+00, 1.887379e-15\\
              7.974181e+00, 2.220446e-15\\
              8.024019e+00, 1.665335e-15\\
              8.073858e+00, 1.998401e-15\\
              8.123696e+00, 2.109424e-15\\
              8.173535e+00, 2.331468e-15\\
              8.223374e+00, 2.109424e-15\\
              8.273212e+00, 2.109424e-15\\
              8.323051e+00, 2.442491e-15\\
              8.372890e+00, 2.442491e-15\\
              8.422728e+00, 2.220446e-15\\
              8.472567e+00, 2.220446e-15\\
              8.522405e+00, 2.553513e-15\\
              8.572244e+00, 1.998401e-15\\
              8.622083e+00, 1.887379e-15\\
              8.671921e+00, 1.776357e-15\\
              8.721760e+00, 2.109424e-15\\
              8.771599e+00, 2.331468e-15\\
              8.821437e+00, 2.220446e-15\\
              8.871276e+00, 2.220446e-15\\
              8.921114e+00, 2.220446e-15\\
              8.970953e+00, 2.442491e-15\\
              9.020792e+00, 1.998401e-15\\
              9.070630e+00, 2.109424e-15\\
              9.120469e+00, 1.887379e-15\\
              9.170308e+00, 2.220446e-15\\
              9.220146e+00, 1.998401e-15\\
              9.269985e+00, 1.776357e-15\\
              9.319823e+00, 1.776357e-15\\
              9.369662e+00, 1.998401e-15\\
              9.419501e+00, 2.220446e-15\\
              9.469339e+00, 2.109424e-15\\
              9.519178e+00, 2.331468e-15\\
              9.569017e+00, 1.887379e-15\\
              9.618855e+00, 2.220446e-15\\
              9.668694e+00, 2.109424e-15\\
              9.718533e+00, 2.220446e-15\\
              9.768371e+00, 2.331468e-15\\
              9.818210e+00, 2.220446e-15\\
              9.868048e+00, 2.331468e-15\\
              9.917887e+00, 1.998401e-15\\
              9.967726e+00, 1.776357e-15\\
              1.000000e+01, 2.109424e-15\\
          };
        \legend{strain error\\momentum error\\};
      \end{axis}
    \end{tikzpicture}%
  \end{minipage}
  \hfill
  \begin{minipage}{0.45\textwidth}
    \begin{tikzpicture}
      \begin{axis}[%
        width=0.75\textwidth,
        height=0.5\textwidth,
        scale only axis,
        xmin=0,xmax=10,
        ymin=-3e-4,ymax=5e-3,
        legend pos= north east,
        title={Central SFIM, split-side mortar elements},
        ylabel={conservation error}, xlabel={$t / t_{0}$},
        ytick={0,1e-3,2e-3,3e-3,4e-3,5e-3},
        scaled ticks=false,
        ]
          \addplot [color=blu,line width=1,dashdotted]
            table[col sep=comma, row sep=crcr]{%
              0.000000e+00, 0.000000e+00\\
              4.983863e-02, 1.409499e-03\\
              9.967726e-02, 1.139821e-03\\
              1.495159e-01, 1.381731e-03\\
              1.993545e-01, 9.332857e-04\\
              2.491931e-01, 1.239485e-03\\
              2.990318e-01, 1.538130e-03\\
              3.488704e-01, 1.170396e-03\\
              3.987090e-01, 1.222340e-03\\
              4.485477e-01, 9.880876e-04\\
              4.983863e-01, 1.696482e-03\\
              5.482249e-01, 8.054774e-04\\
              5.980635e-01, 1.951551e-03\\
              6.479022e-01, 1.139117e-03\\
              6.977408e-01, 1.381568e-03\\
              7.475794e-01, 1.095751e-03\\
              7.974181e-01, 1.781310e-03\\
              8.472567e-01, 1.216418e-03\\
              8.970953e-01, 6.887492e-04\\
              9.469339e-01, 1.147997e-03\\
              9.967726e-01, 2.320944e-03\\
              1.046611e+00, 6.630822e-04\\
              1.096450e+00, 1.539449e-03\\
              1.146288e+00, 9.748926e-04\\
              1.196127e+00, 1.820844e-03\\
              1.245966e+00, 1.233433e-03\\
              1.295804e+00, 1.394963e-03\\
              1.345643e+00, 1.893584e-03\\
              1.395482e+00, 3.802052e-04\\
              1.445320e+00, 7.461630e-04\\
              1.495159e+00, 1.759034e-03\\
              1.544997e+00, 1.368409e-03\\
              1.594836e+00, 1.035627e-03\\
              1.644675e+00, 1.700389e-03\\
              1.694513e+00, 1.969293e-03\\
              1.744352e+00, 1.490097e-03\\
              1.794191e+00, 1.924140e-03\\
              1.844029e+00, 1.152758e-03\\
              1.893868e+00, 9.430170e-04\\
              1.943707e+00, 1.832281e-03\\
              1.993545e+00, 1.215140e-03\\
              2.043384e+00, 1.250023e-03\\
              2.093222e+00, 8.793414e-04\\
              2.143061e+00, 1.144889e-03\\
              2.192900e+00, 1.648985e-03\\
              2.242738e+00, 1.360918e-03\\
              2.292577e+00, 1.596333e-03\\
              2.342416e+00, 1.021330e-03\\
              2.392254e+00, 1.502903e-03\\
              2.442093e+00, 1.589237e-03\\
              2.491931e+00, 1.878275e-03\\
              2.541770e+00, 5.195279e-04\\
              2.591609e+00, 9.943645e-04\\
              2.641447e+00, 1.108192e-03\\
              2.691286e+00, 1.889614e-03\\
              2.741125e+00, 1.471780e-03\\
              2.790963e+00, 9.768058e-04\\
              2.840802e+00, 1.452815e-03\\
              2.890640e+00, 1.058064e-03\\
              2.940479e+00, 2.219792e-03\\
              2.990318e+00, 1.288309e-03\\
              3.040156e+00, 1.509162e-03\\
              3.089995e+00, 1.516158e-03\\
              3.139834e+00, 1.097512e-03\\
              3.189672e+00, 1.392605e-03\\
              3.239511e+00, 1.345119e-03\\
              3.289349e+00, 1.932237e-03\\
              3.339188e+00, 1.979047e-03\\
              3.389027e+00, 1.103246e-03\\
              3.438865e+00, 1.690034e-03\\
              3.488704e+00, 9.265802e-04\\
              3.538543e+00, 7.328737e-04\\
              3.588381e+00, 1.301871e-03\\
              3.638220e+00, 7.273134e-04\\
              3.688058e+00, 1.306642e-03\\
              3.737897e+00, 9.938150e-04\\
              3.787736e+00, 1.398080e-03\\
              3.837574e+00, 1.025820e-03\\
              3.887413e+00, 1.553450e-03\\
              3.937252e+00, 1.192690e-03\\
              3.987090e+00, 1.584171e-03\\
              4.036929e+00, 2.036936e-03\\
              4.086768e+00, 1.610922e-03\\
              4.136606e+00, 1.782758e-03\\
              4.186445e+00, 1.159138e-03\\
              4.236283e+00, 1.386438e-03\\
              4.286122e+00, 7.970769e-04\\
              4.335961e+00, 2.110248e-03\\
              4.385799e+00, 1.676301e-03\\
              4.435638e+00, 8.626334e-04\\
              4.485477e+00, 1.297620e-03\\
              4.535315e+00, 1.372394e-03\\
              4.585154e+00, 5.926828e-04\\
              4.634992e+00, 2.203085e-03\\
              4.684831e+00, 1.870048e-03\\
              4.734670e+00, 9.317201e-04\\
              4.784508e+00, 1.258451e-03\\
              4.834347e+00, 1.324349e-03\\
              4.884186e+00, 1.439316e-03\\
              4.934024e+00, 1.935563e-03\\
              4.983863e+00, 1.229408e-03\\
              5.033701e+00, 7.773924e-04\\
              5.083540e+00, 1.816711e-03\\
              5.133379e+00, 2.118606e-03\\
              5.183217e+00, 6.703875e-04\\
              5.233056e+00, 1.436901e-03\\
              5.282895e+00, 9.915514e-04\\
              5.332733e+00, 1.230801e-03\\
              5.382572e+00, 1.439874e-03\\
              5.432410e+00, 1.124846e-03\\
              5.482249e+00, 8.151337e-04\\
              5.532088e+00, 2.209687e-03\\
              5.581926e+00, 8.748244e-04\\
              5.631765e+00, 6.935638e-04\\
              5.681604e+00, 7.488760e-04\\
              5.731442e+00, 1.272070e-03\\
              5.781281e+00, 1.449137e-03\\
              5.831120e+00, 6.920859e-04\\
              5.880958e+00, 1.855821e-03\\
              5.930797e+00, 7.330103e-04\\
              5.980635e+00, 8.048585e-04\\
              6.030474e+00, 1.350499e-03\\
              6.080313e+00, 7.097800e-04\\
              6.130151e+00, 7.790951e-04\\
              6.179990e+00, 6.229948e-04\\
              6.229829e+00, 1.884496e-03\\
              6.279667e+00, 1.075522e-03\\
              6.329506e+00, 7.440455e-04\\
              6.379344e+00, 2.266229e-03\\
              6.429183e+00, 1.476640e-03\\
              6.479022e+00, 7.880522e-04\\
              6.528860e+00, 1.621974e-03\\
              6.578699e+00, 1.412443e-03\\
              6.628538e+00, 1.941090e-03\\
              6.678376e+00, 5.052495e-04\\
              6.728215e+00, 1.087520e-03\\
              6.778053e+00, 1.366389e-03\\
              6.827892e+00, 1.909548e-03\\
              6.877731e+00, 7.903637e-04\\
              6.927569e+00, 1.268381e-03\\
              6.977408e+00, 2.048136e-03\\
              7.027247e+00, 1.330951e-03\\
              7.077085e+00, 1.854369e-03\\
              7.126924e+00, 1.329300e-03\\
              7.176762e+00, 1.678620e-03\\
              7.226601e+00, 1.341949e-03\\
              7.276440e+00, 2.296315e-03\\
              7.326278e+00, 1.231417e-03\\
              7.376117e+00, 9.594144e-04\\
              7.425956e+00, 1.609712e-03\\
              7.475794e+00, 1.891646e-03\\
              7.525633e+00, 7.792976e-04\\
              7.575471e+00, 5.363200e-04\\
              7.625310e+00, 1.651346e-03\\
              7.675149e+00, 1.657231e-03\\
              7.724987e+00, 1.331042e-03\\
              7.774826e+00, 8.494324e-04\\
              7.824665e+00, 1.062170e-03\\
              7.874503e+00, 2.216837e-03\\
              7.924342e+00, 1.840214e-03\\
              7.974181e+00, 1.445647e-03\\
              8.024019e+00, 1.310285e-03\\
              8.073858e+00, 2.012387e-03\\
              8.123696e+00, 8.608922e-04\\
              8.173535e+00, 9.249868e-04\\
              8.223374e+00, 2.140525e-03\\
              8.273212e+00, 1.582928e-03\\
              8.323051e+00, 2.199719e-03\\
              8.372890e+00, 1.135227e-03\\
              8.422728e+00, 1.318083e-03\\
              8.472567e+00, 9.625366e-04\\
              8.522405e+00, 2.288602e-03\\
              8.572244e+00, 1.750778e-03\\
              8.622083e+00, 8.785222e-04\\
              8.671921e+00, 6.635298e-04\\
              8.721760e+00, 1.384773e-03\\
              8.771599e+00, 1.455082e-03\\
              8.821437e+00, 8.433618e-04\\
              8.871276e+00, 1.951344e-03\\
              8.921114e+00, 8.961111e-04\\
              8.970953e+00, 1.833480e-03\\
              9.020792e+00, 7.899430e-04\\
              9.070630e+00, 1.036090e-03\\
              9.120469e+00, 1.666305e-03\\
              9.170308e+00, 1.917937e-03\\
              9.220146e+00, 1.129226e-03\\
              9.269985e+00, 1.260733e-03\\
              9.319823e+00, 1.054256e-03\\
              9.369662e+00, 1.334531e-03\\
              9.419501e+00, 1.727010e-03\\
              9.469339e+00, 1.152767e-03\\
              9.519178e+00, 8.017031e-04\\
              9.569017e+00, 8.084845e-04\\
              9.618855e+00, 1.355257e-03\\
              9.668694e+00, 1.822917e-03\\
              9.718533e+00, 1.210919e-03\\
              9.768371e+00, 1.407223e-03\\
              9.818210e+00, 8.604977e-04\\
              9.868048e+00, 1.381625e-03\\
              9.917887e+00, 2.163631e-03\\
              9.967726e+00, 5.867016e-04\\
              1.000000e+01, 9.016917e-04\\
          };
          \addplot [color=red,line width=1,solid]
            table[col sep=comma, row sep=crcr]{%
              0.000000e+00, 0.000000e+00\\
              4.983863e-02, 1.110223e-16\\
              9.967726e-02, 5.551115e-16\\
              1.495159e-01, 5.551115e-16\\
              1.993545e-01, 4.440892e-16\\
              2.491931e-01, 2.220446e-16\\
              2.990318e-01, 0.000000e+00\\
              3.488704e-01, 3.330669e-16\\
              3.987090e-01, 3.330669e-16\\
              4.485477e-01, 5.551115e-16\\
              4.983863e-01, 4.440892e-16\\
              5.482249e-01, 3.330669e-16\\
              5.980635e-01, 4.440892e-16\\
              6.479022e-01, 7.771561e-16\\
              6.977408e-01, 8.881784e-16\\
              7.475794e-01, 7.771561e-16\\
              7.974181e-01, 8.881784e-16\\
              8.472567e-01, 8.881784e-16\\
              8.970953e-01, 7.771561e-16\\
              9.469339e-01, 5.551115e-16\\
              9.967726e-01, 7.771561e-16\\
              1.046611e+00, 7.771561e-16\\
              1.096450e+00, 5.551115e-16\\
              1.146288e+00, 7.771561e-16\\
              1.196127e+00, 7.771561e-16\\
              1.245966e+00, 5.551115e-16\\
              1.295804e+00, 6.661338e-16\\
              1.345643e+00, 7.771561e-16\\
              1.395482e+00, 5.551115e-16\\
              1.445320e+00, 5.551115e-16\\
              1.495159e+00, 6.661338e-16\\
              1.544997e+00, 1.110223e-15\\
              1.594836e+00, 6.661338e-16\\
              1.644675e+00, 6.661338e-16\\
              1.694513e+00, 8.881784e-16\\
              1.744352e+00, 1.110223e-15\\
              1.794191e+00, 8.881784e-16\\
              1.844029e+00, 1.332268e-15\\
              1.893868e+00, 1.332268e-15\\
              1.943707e+00, 1.110223e-15\\
              1.993545e+00, 4.440892e-16\\
              2.043384e+00, 8.881784e-16\\
              2.093222e+00, 6.661338e-16\\
              2.143061e+00, 1.110223e-15\\
              2.192900e+00, 6.661338e-16\\
              2.242738e+00, 1.110223e-15\\
              2.292577e+00, 9.992007e-16\\
              2.342416e+00, 1.110223e-15\\
              2.392254e+00, 1.110223e-15\\
              2.442093e+00, 8.881784e-16\\
              2.491931e+00, 1.110223e-15\\
              2.541770e+00, 9.992007e-16\\
              2.591609e+00, 1.110223e-15\\
              2.641447e+00, 8.881784e-16\\
              2.691286e+00, 8.881784e-16\\
              2.741125e+00, 6.661338e-16\\
              2.790963e+00, 6.661338e-16\\
              2.840802e+00, 8.881784e-16\\
              2.890640e+00, 6.661338e-16\\
              2.940479e+00, 1.110223e-15\\
              2.990318e+00, 6.661338e-16\\
              3.040156e+00, 1.110223e-15\\
              3.089995e+00, 1.110223e-15\\
              3.139834e+00, 6.661338e-16\\
              3.189672e+00, 8.881784e-16\\
              3.239511e+00, 6.661338e-16\\
              3.289349e+00, 6.661338e-16\\
              3.339188e+00, 8.881784e-16\\
              3.389027e+00, 4.440892e-16\\
              3.438865e+00, 4.440892e-16\\
              3.488704e+00, 8.881784e-16\\
              3.538543e+00, 2.220446e-16\\
              3.588381e+00, 8.881784e-16\\
              3.638220e+00, 8.881784e-16\\
              3.688058e+00, 1.110223e-15\\
              3.737897e+00, 1.110223e-15\\
              3.787736e+00, 1.110223e-15\\
              3.837574e+00, 1.110223e-15\\
              3.887413e+00, 1.110223e-15\\
              3.937252e+00, 1.110223e-15\\
              3.987090e+00, 1.110223e-15\\
              4.036929e+00, 8.881784e-16\\
              4.086768e+00, 8.881784e-16\\
              4.136606e+00, 8.881784e-16\\
              4.186445e+00, 8.881784e-16\\
              4.236283e+00, 6.661338e-16\\
              4.286122e+00, 8.881784e-16\\
              4.335961e+00, 6.661338e-16\\
              4.385799e+00, 6.661338e-16\\
              4.435638e+00, 8.881784e-16\\
              4.485477e+00, 6.661338e-16\\
              4.535315e+00, 6.661338e-16\\
              4.585154e+00, 8.881784e-16\\
              4.634992e+00, 8.881784e-16\\
              4.684831e+00, 8.881784e-16\\
              4.734670e+00, 1.110223e-15\\
              4.784508e+00, 8.881784e-16\\
              4.834347e+00, 6.661338e-16\\
              4.884186e+00, 8.881784e-16\\
              4.934024e+00, 4.440892e-16\\
              4.983863e+00, 3.330669e-16\\
              5.033701e+00, 5.551115e-16\\
              5.083540e+00, 5.551115e-16\\
              5.133379e+00, 8.881784e-16\\
              5.183217e+00, 3.330669e-16\\
              5.233056e+00, 8.881784e-16\\
              5.282895e+00, 5.551115e-16\\
              5.332733e+00, 3.330669e-16\\
              5.382572e+00, 2.220446e-16\\
              5.432410e+00, 4.440892e-16\\
              5.482249e+00, 2.220446e-16\\
              5.532088e+00, 5.551115e-16\\
              5.581926e+00, 8.881784e-16\\
              5.631765e+00, 4.440892e-16\\
              5.681604e+00, 3.330669e-16\\
              5.731442e+00, 4.440892e-16\\
              5.781281e+00, 5.551115e-16\\
              5.831120e+00, 4.440892e-16\\
              5.880958e+00, 4.440892e-16\\
              5.930797e+00, 3.330669e-16\\
              5.980635e+00, 6.661338e-16\\
              6.030474e+00, 7.771561e-16\\
              6.080313e+00, 6.661338e-16\\
              6.130151e+00, 4.440892e-16\\
              6.179990e+00, 8.881784e-16\\
              6.229829e+00, 6.661338e-16\\
              6.279667e+00, 7.771561e-16\\
              6.329506e+00, 1.110223e-15\\
              6.379344e+00, 1.332268e-15\\
              6.429183e+00, 8.881784e-16\\
              6.479022e+00, 9.992007e-16\\
              6.528860e+00, 6.661338e-16\\
              6.578699e+00, 6.661338e-16\\
              6.628538e+00, 5.551115e-16\\
              6.678376e+00, 7.771561e-16\\
              6.728215e+00, 8.881784e-16\\
              6.778053e+00, 7.771561e-16\\
              6.827892e+00, 9.992007e-16\\
              6.877731e+00, 9.992007e-16\\
              6.927569e+00, 1.554312e-15\\
              6.977408e+00, 1.221245e-15\\
              7.027247e+00, 1.110223e-15\\
              7.077085e+00, 1.110223e-15\\
              7.126924e+00, 7.771561e-16\\
              7.176762e+00, 8.881784e-16\\
              7.226601e+00, 7.771561e-16\\
              7.276440e+00, 1.110223e-15\\
              7.326278e+00, 1.332268e-15\\
              7.376117e+00, 1.554312e-15\\
              7.425956e+00, 1.665335e-15\\
              7.475794e+00, 1.776357e-15\\
              7.525633e+00, 1.776357e-15\\
              7.575471e+00, 1.443290e-15\\
              7.625310e+00, 1.332268e-15\\
              7.675149e+00, 1.332268e-15\\
              7.724987e+00, 1.776357e-15\\
              7.774826e+00, 1.554312e-15\\
              7.824665e+00, 1.443290e-15\\
              7.874503e+00, 1.443290e-15\\
              7.924342e+00, 1.665335e-15\\
              7.974181e+00, 1.443290e-15\\
              8.024019e+00, 1.776357e-15\\
              8.073858e+00, 1.665335e-15\\
              8.123696e+00, 1.776357e-15\\
              8.173535e+00, 1.665335e-15\\
              8.223374e+00, 1.998401e-15\\
              8.273212e+00, 1.998401e-15\\
              8.323051e+00, 1.887379e-15\\
              8.372890e+00, 1.776357e-15\\
              8.422728e+00, 1.443290e-15\\
              8.472567e+00, 1.554312e-15\\
              8.522405e+00, 1.110223e-15\\
              8.572244e+00, 1.110223e-15\\
              8.622083e+00, 1.332268e-15\\
              8.671921e+00, 1.110223e-15\\
              8.721760e+00, 1.110223e-15\\
              8.771599e+00, 1.110223e-15\\
              8.821437e+00, 1.110223e-15\\
              8.871276e+00, 1.332268e-15\\
              8.921114e+00, 9.992007e-16\\
              8.970953e+00, 1.332268e-15\\
              9.020792e+00, 1.110223e-15\\
              9.070630e+00, 1.110223e-15\\
              9.120469e+00, 9.992007e-16\\
              9.170308e+00, 9.992007e-16\\
              9.220146e+00, 8.881784e-16\\
              9.269985e+00, 8.881784e-16\\
              9.319823e+00, 9.992007e-16\\
              9.369662e+00, 6.661338e-16\\
              9.419501e+00, 8.881784e-16\\
              9.469339e+00, 8.881784e-16\\
              9.519178e+00, 9.992007e-16\\
              9.569017e+00, 1.110223e-15\\
              9.618855e+00, 9.992007e-16\\
              9.668694e+00, 8.881784e-16\\
              9.718533e+00, 1.110223e-15\\
              9.768371e+00, 1.110223e-15\\
              9.818210e+00, 1.110223e-15\\
              9.868048e+00, 1.110223e-15\\
              9.917887e+00, 1.221245e-15\\
              9.967726e+00, 9.992007e-16\\
              1.000000e+01, 8.881784e-16\\
          };
      \end{axis}
    \end{tikzpicture}%
  \end{minipage}
  \caption{Conservation  errors for velocities and strains for SFIM
  using a split-side mortar with an upwind and central flux with interpolated
  mesh.\label{fig:conservative:constant:components}}
\end{figure}
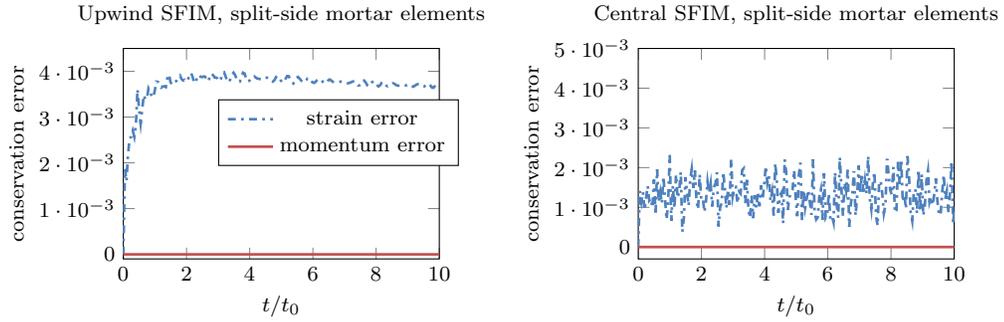
To demonstrate the origin of the conervation error, in
\fref{fig:conservative:constant:components} the conservation error is shown
separately for the strain ($\sum_{ij}e_{\epsilon_{ij}}$) and momentum
($\sum_{i}e_{\rho v_{i}}$) components using the interpolated geometry approach.
As can be seen the conservation error is coming purely from lack of conservation
of the strain. As noted in \sref{sec:constant}, the momentum components are conservative by
construction since the velocity update~\eref{eqn:dg:vel:quad} uses a weak
derivative.

\begin{figure}
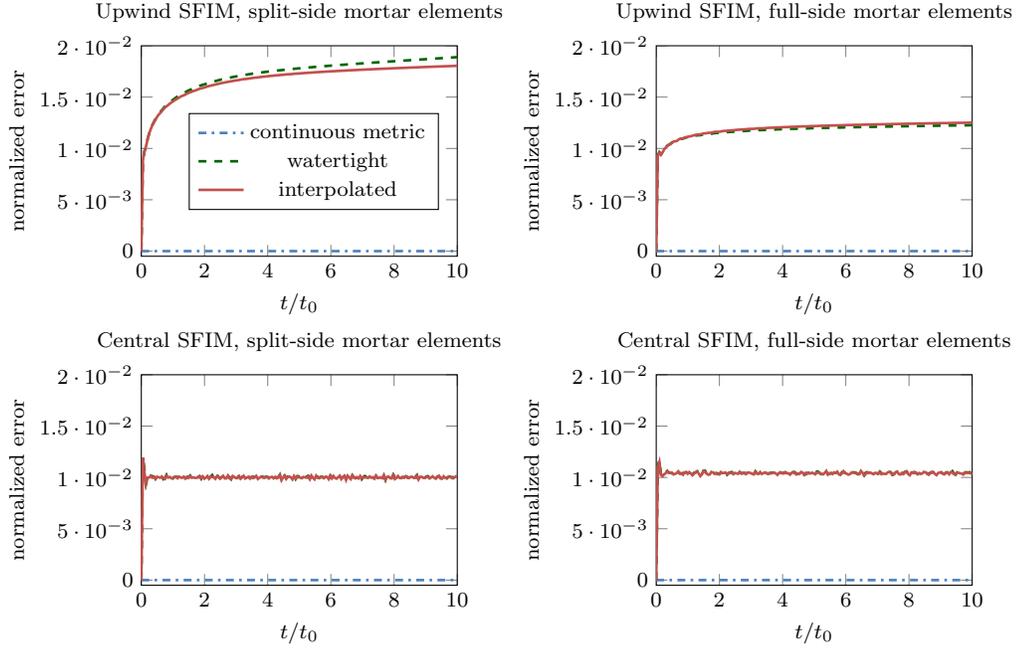

  \begin{minipage}{0.45\textwidth}
%
  \end{minipage}
  \caption{Constant preservation error for both the full-side and split-side
  mortar with both upwinding ($\alpha = 1$) and central flux ($\alpha = 0$) for
  the continuous metric term, watertight, and interpolated handling of
  geometry. The normalized error is the L$^{2}$ error in the solution normalized
  by the initial energy.\label{fig:constant:box} }
\end{figure}
In order to test the constant preserving properties of SFIM, the initial
condition is made constant:
\begin{align}
  v_{1}       &= 1, &
  v_{2}       &= 2, &
  v_{3}       &= 3, &
  \sigma_{11} &= 4, &
  \sigma_{12} &= 5, &
  \sigma_{13} &= 6, \\
  \sigma_{22} &= 7, &
  \sigma_{23} &= 8, &
  \sigma_{33} &= 9,
\end{align}
with the material properties still assigned pseudorandom values.
\fref{fig:constant:box} shows the L$^{2}$ error normalized by the initial energy
in the solution for different choices of schemes. As can be seen both the
interpolated and watertight geometry treatment do not exactly preserve
constants, but the continuous metric treatment does for both the full-side and
split-side mortar regardless of numerical flux. Unlike in the conservation
error, the constant preservation error cannot be separated into velocity and
stress errors since once one component moves away from the constant all the
components are affected.

\subsection{Mode of a Heterogeneous Spherical Shell}\label{sec:shell:mode}
\begin{figure}[tb]
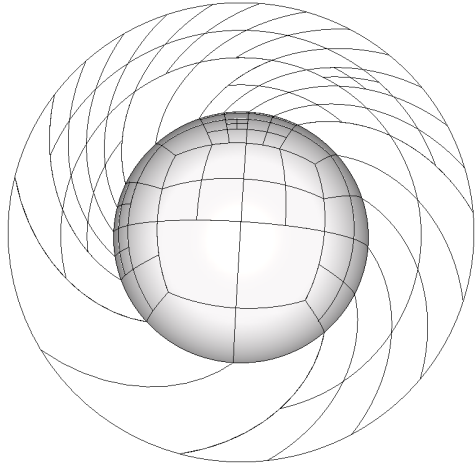

  \centering
%
  \caption{Cross-section of the base mesh for the spherical shell test problem
  showing both the psuedorandom refinement and the rotation of outside of the shell
  with respect to the inside. Total number of elements in the base mesh is
  $|\EE| =
  745$.\label{fig:shell:mesh}}
\end{figure}
\begin{figure}
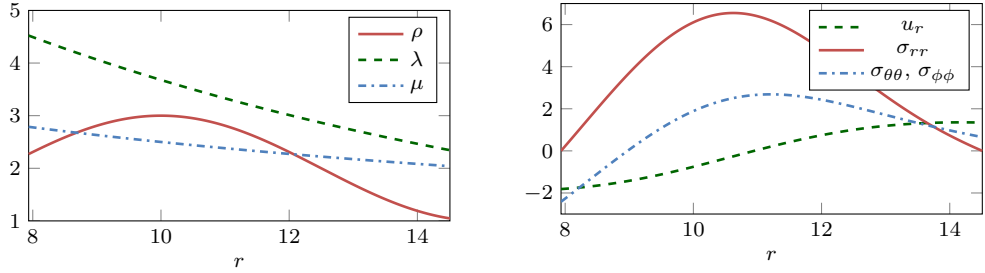

  \centering
  \begin{minipage}{0.45\textwidth}
    %
%
  \end{minipage}
  \caption{(left) Material properties as a function of $r$ for the spherical
  shell test problem. (right) Non-zero, modal solution fields (in spherical
  coordinates) at time $t = 0$ for the spherical shell test
  problem.\label{fig:shell:flds}}
\end{figure}
To further explore the accuracy of SFIM on curved meshes, an adapted spherical
shell is considered.  The spherical shell is initially meshed using $6$
hexahedral elements, and these elements are then refined using a pseudorandom
number generator to produce a non-conforming base mesh with $|\EE| = 745$ elements.
Since the common (hexahedral) decomposition of a spherical shell would result in
one of the grid lines being aligned with the radial direction, the complexity of
the problem is increased by rotating the outside of the shell by one radian with
respect to the inside of the shell (along the polar angle). A cross-section
through the center of the sphere of the mesh is shown in \fref{fig:shell:mesh}.
Only the continuous metric approach for geomtry is considered here.
The material properties are taken to be
\begin{align}
  \rho &= 2 + \cos\left(\frac{2 \pi r}{10}\right), &
  \mu  &= \frac{50}{10+r},&
  \lambda  = 10\exp\left(\frac{-r}{10}\right);
\end{align}
see also the left panel of \fref{fig:shell:flds}.  To derive a modal solution
with these material parameters, it is assumed that when written in spherical
coordinates, only the radial displacement is non-zero.  The radial displacement
is taken to be of the form
\begin{align}
  u_{r} = \cos(t) \phi(r),
\end{align}
e.g., the temporal and spatial dependence are seperable. With this, in spherical
coordinates the components of the stress tensor are
\begin{align}
  \sigma_{rr} &= \left(\lambda + 2\mu\right) \pd{u_{r}}{r} + 2\mu \frac{u_{r}}{r},
  &
  \sigma_{\theta\theta} &= \sigma_{\phi\phi} =
  2\left(\lambda + \mu\right) \frac{u_{r}}{r} + \lambda \pd{u_{r}}{r},
\end{align}
with $\sigma_{r\theta} = \sigma_{r\phi} = \sigma_{\theta\phi} = 0$. To solve
this the MATLAB package Chebfun~\cite{DriscollHaleTrefethen2014} is used, and
$R_{1}$ and $R_{2}$ are chosen so that $\sigma_{rr} = 0$ is zero at the inner
and outer radii of the shell; see \aref{app:shell} for details.  The right panel
of
\fref{fig:shell:flds} shows $u_{r}$, $\sigma_{rr}$, $\sigma_{\theta\theta}$, and
$\sigma_{\phi\phi}$ as functions of $r$ at time $t = 0$. To implement the
solution, the displacements are first converted to Cartesian coordinates,
\begin{align}
  u_{i} &= \frac{x_{i}}{r}u_{r},
\end{align}
where $r = \sqrt{x_{i}x_{i}}$ is the radial distance, and then the velocities
and stresses are computed using~\eref{eqn:u2vS}.

\begin{table}
  \centering
  \begin{tabular}{rllll}
    \toprule
    & \multicolumn{1}{c}{$N=4$}
    & \multicolumn{1}{c}{$N=5$}
    & \multicolumn{1}{c}{$N=6$}
    & \multicolumn{1}{c}{$N=7$}\\
    {$|\EE|$}
    & \multicolumn{1}{c}{error (rate)}
    & \multicolumn{1}{c}{error (rate)}
    & \multicolumn{1}{c}{error (rate)}
    & \multicolumn{1}{c}{error (rate)}\\
    \midrule
    & \multicolumn{4}{c}{SFIM, full-side mortar elements}\\
    $   745$ & $8.3\times10^{-0}$ $(0.0)$ & $3.2\times10^{-0}$ $(0.0)$ & $8.7\times10^{-1}$ $(0.0)$ & $2.3\times10^{-1}$ $(0.0)$\\
    $  5960$ & $3.2\times10^{-1}$ $(4.7)$ & $4.9\times10^{-2}$ $(6.0)$ & $7.6\times10^{-3}$ $(6.8)$ & $1.1\times10^{-3}$ $(7.6)$\\
    $ 47680$ & $1.1\times10^{-2}$ $(4.9)$ & $9.1\times10^{-4}$ $(5.7)$ & $7.5\times10^{-5}$ $(6.7)$ & $5.5\times10^{-6}$ $(7.7)$\\
    $381440$ & $4.1\times10^{-4}$ $(4.7)$ & $1.8\times10^{-5}$ $(5.7)$ & $7.5\times10^{-7}$ $(6.6)$ & $2.7\times10^{-8}$ $(7.6)$\\
    \midrule
    & \multicolumn{4}{c}{SFIM, split-side mortar elements}\\
    $   745$ & $8.3\times10^{-0}$ $(0.0)$ & $3.1\times10^{-0}$ $(0.0)$ & $8.5\times10^{-1}$ $(0.0)$ & $2.3\times10^{-1}$ $(0.0)$\\
    $  5960$ & $3.2\times10^{-1}$ $(4.7)$ & $4.8\times10^{-2}$ $(6.0)$ & $7.5\times10^{-3}$ $(6.8)$ & $1.1\times10^{-3}$ $(7.6)$\\
    $ 47680$ & $1.0\times10^{-2}$ $(4.9)$ & $8.9\times10^{-4}$ $(5.7)$ & $7.2\times10^{-5}$ $(6.7)$ & $5.3\times10^{-6}$ $(7.7)$\\
    $381440$ & $3.8\times10^{-4}$ $(4.8)$ & $1.7\times10^{-5}$ $(5.7)$ & $6.9\times10^{-7}$ $(6.7)$ & $2.5\times10^{-8}$ $(7.7)$\\
    \bottomrule

  \end{tabular}
  \caption{Error and estimated convergence rates for the spherical
  shell using SFIM with both full-side and split-side mortar elements using an
  upwind flux. A log-log plot of these errors is shown in
  \fref{fig:shell}.\label{tab:shell}}
\end{table}
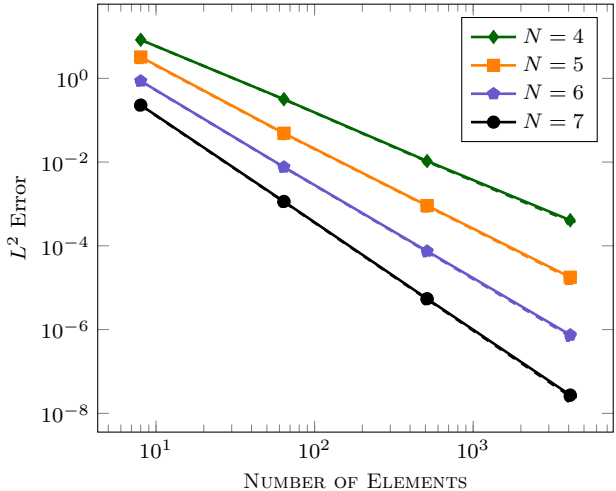
\begin{figure}[tb]
  \centering
  \begin{tikzpicture}
    \begin{loglogaxis}[
        xlabel=\textsc{Number of Elements},
        ylabel=L$^2$ Error,
        legend pos= north east
      ]

      \addplot[color=grn,mark=diamond*,line width=1] plot coordinates {%
        (    8, 8.3097811963463979e+00)
        (   64, 3.1759817459010359e-01)
        (  512, 1.0618880916478252e-02)
        ( 4096, 4.0963699066878958e-04)
      };

      \addplot[color=org,mark=square*,line width=1] plot coordinates {%
        (    8, 3.2373392382145081e+00)
        (   64, 4.8908539604001385e-02)
        (  512, 9.1445257309106522e-04)
        ( 4096, 1.7834423747545586e-05)
      };

      \addplot[color=prp,mark=pentagon*,line width=1] plot coordinates {%
        (    8, 8.6634295643591197e-01)
        (   64, 7.6249214645476683e-03)
        (  512, 7.4935892463975362e-05)
        ( 4096, 7.4884333220607252e-07)
      };

      \addplot[color=black,mark=*,line width=1] plot coordinates {%
        (    8, 2.3004952427599742e-01)
        (   64, 1.1484346683710820e-03)
        (  512, 5.4733919542406406e-06)
        ( 4096, 2.7353708707155175e-08)
      };

      \addplot[color=grn,mark=diamond*,line width=1,dashed] plot coordinates {%
        (    8, 8.3002608920145740e+00)
        (   64, 3.1515498712647122e-01)
        (  512, 1.0282655372037455e-02)
        ( 4096, 3.7929815294197698e-04)
      };

      \addplot[color=org,mark=square*,line width=1,dashed] plot coordinates {%
        (    8, 3.1347531386733452e+00)
        (   64, 4.7949650776244329e-02)
        (  512, 8.9164303550534239e-04)
        ( 4096, 1.6639561279908413e-05)
      };

      \addplot[color=prp,mark=pentagon*,line width=1,dashed] plot coordinates {%
        (    8, 8.5424308847515917e-01)
        (   64, 7.5342636935113918e-03)
        (  512, 7.2028092288188724e-05)
        ( 4096, 6.8980607256601298e-07)
      };

      \addplot[color=black,mark=*,line width=1,dashed] plot coordinates {%
        (    8, 2.2558881637570674e-01)
        (   64, 1.1242386724799329e-03)
        (  512, 5.2877797806373952e-06)
        ( 4096, 2.5338984156190091e-08)
      };

      \legend{$N=4$\\$N=5$\\$N=6$\\$N=7$\\};
    \end{loglogaxis}
  \end{tikzpicture}
  \caption{Log-log plot of $|\EE|$ (number of elements) versus the L$^{2}$ error
  (measured with the energy norm) for a mesh the spherical shell test problem
  for SFIM with full-side mortar elements (solid lines) and split-side mortar
  elements (dashed lines) with the upwind flux. Numerical values for the error
  and rates are given in Table~\ref{tab:shell}.\label{fig:shell}}
\end{figure}
In Table~\ref{tab:shell} and \fref{fig:shell} error and convergence results are
given for this problem for varying polynomial orders for SFIM with both
full-side and split-side mortar elements using the upwind flux. The final time
for the simulation is $t = 6\pi$, and thus three oscillations of the solution
have been considered. As can be seen in the table and figure, the method
performs similarly to the previous planewave test when the elements are curved
and the material properties are heterogeneous.

\section{Conclusions}
In this work we have presented a skew-symmetric, semi-discretely energy stable
discontinuous Galerkin method on nonconforming, non-affine meshes. The two key
ideas that lead to semi-discrete energy stability were:
\begin{itemize}
  \item writing the variational problem in skew-symmetric form so that
    integration-by-parts was not needed in the energy stability analysis; and
  \item evaluating all the surface integrals on the mortar elements
    (as opposed to volume elements faces).
\end{itemize}
Both of these steps ensure that variational crimes, namely inexact integration, do
not affect the solution in a manner that results in energy growth. Importantly,
skew-symmetry is beneficial for nonconforming meshes even for constant
coefficient problems on affine meshes when the integration is inexact, e.g.,
Legendre-Gauss-Lobatto quadrature on tensor product quadrilateral and hexahedral
elements (DG-SEM); for non-affine meshes and/or nonconstant coefficients it is
now well-known that skew-symmetry is needed even for conforming
meshes.

We have also shown how the SFIM formulation can be made both conservative and
constant preserving by adding a set of accuracy and consistency constraints on
the stiffness and projection operators. Most of the additional requirements
will be satisfied naturally by any higher-order method (such as the exact
representation and differentiation of constants), but the existence of a
discrete divergence property~\eref{eqn:disc:div} requires careful handling of
the metric terms on non-affine meshes; in \aref{app:div} we give one approach to
satisfying the discrete divergence theorem for DG-SEM\@.

In the results we presented two mortar approaches for handling nonconforming
interfaces which we called the full-side and split-side mortar; see
\fref{fig:mortars}. When full-side mortars are used, mortar elements on
nonconforming interfaces conform to the larger element face, thus multiple
small faces connect to a single mortar element. When split-side mortars are
used, the smaller faces conform to the mortar elements, thus the larger element
faces are connected to multiple mortar elements.
Both the split-side and full-side mortars are energy stable with SFIM and can be
made conservative and constant preserving.
In terms of accuracy observed, the two methods performed similarly with
the full-side mortar being slightly better in the affine, constant coefficient
test and the split-side mortar performing marginally better for the non-affine,
variable coefficient problem.
The one big difference between the two mortar approaches was the eigenvalue
spectrum, with the full-side mortar having a spectral radius more than twice as
large as the split-side for the problem considered.
Given these results, and the fact that the split-side mortar (in some sense) is
the most natural approach for discontinuous Galerkin methods, at this time we
would recommend the split-side mortar over the full-side. That said, the tests
here represent only a few out of the infinite number of problems that could be
considered. Amongst the important problems not considered here are the impact
the mortar elements have on spurious solution modes and whether either approach
is more accurate for interface waves.

\begin{acknowledgements}
  Jeremy~E.~Kozdon was partially supported by National Science Foundation Award
  EAR-1547596 and Office of Naval Research Award N0001416WX01290. Computational
  resources were provided in part by National Science Foundation Award
  OAC-1403203.

  The authors would like to thank the two anonymous reviewers for their
  feedback, particularly the recommendation to consider the conservation and
  constant state preserving properties of the method.
\end{acknowledgements}

\appendix

\section{Construction of L$^{2}$-Projection Operators}\label{app:project}
Here the construction of the mortar projection operators is discussed, namely
the construction of $\mathcal{P}^{m,e}$ for some $m \in \MM$ and $e \in
\EE^{m}$. Of particular interest is when the mortar space does not support all
the functions that the connected element faces support. For example, in the
center panel of \fref{fig:mortars} if mortar $m^2$ uses polynomials up to
degree $N$ as the approximation space then it cannot support the discontinuous
functions defined by functions on the connected faces of $e^{2}$ and $e^{3}$.

For an $m \in \MM$ and $e \in \EE^{m}$, let $\hat{m}^{e}$ be the portion of the
reference mortar element $\hat{m}$ that corresponds to the intersection in
physical space of $m$ and $\partial e$, that is $\hat{m}^{e} = \vec{r}^{m}(m
\cap \partial e)$. The projection operator in $\mathcal{P}^{m,e}$ is then
defined so that for a given $q^{e} \in \hat{V}^{e}$ the following holds for all
$\phi^{m} \in \hat{U}^{m}$:
\begin{align}
  \label{eqn:L2proj}
  \int_{\hat{m}} \phi^{m} q^{m,e} = \int_{\hat{m}^{e}} \phi^{m} q^{e},
\end{align}
where $q^{m,e} = \mathcal{P}^{m,e} q^{e}$. In the right-hand side integrand
$q^{e}$ is evaluated on the reference mortar, not the element boundary, and thus
$q^{e} =
q^{e}\left(\vec{x}^{e}\left(\vec{r}^{e}\left(\vec{x}^{m}\left(\vec{r}\right)\right)\right)\right)$
with $\vec{r}$ being the integration variable. Note the surface Jacobian
determinant is not
included in the definition of the projection operator given here, and thus the
projection operator is defined on the straight-sided reference element.

In the results, exact integration is used to
construct $\mat{P}^{m,e}$, so that there are no variational crimes in the
discrete representation of $\mathcal{P}^{m,e}$ defined from~\eref{eqn:L2proj}
(i.e., Legendre-Gauss-Lobatto (LGL) quadrature is not used).
Moreover, since the basis functions are tensor product, the L$^{2}$-projection
operators can also be constructed as tensor product operators, and only
one-dimensional projection operators are needed. Thus, since the nonconforming
interfaces in the results are two-to-one, only projection operators from the
bottom, $\mathcal{P}_{b}$, and top, $\mathcal{P}_{t}$, halves of a 1-D element
are required.
The bottom and top projection operators then satisfy
\begin{align}
  \label{eqn:L21dproj}
  \int_{-1}^{1} \phi\; \mathcal{P}_{b} q &= \int_{-1}^{0} \phi\; q,
  &
  \int_{-1}^{1} \phi\; \mathcal{P}_{t} q &= \int_{0}^{1} \phi\; q,
\end{align}
for all polynomials $q$ and $\phi$ of degree $N$.

To define the discrete operators $\mat{P}_{b}$ and $\mat{P}_{t}$, let
$\vec{\phi}$ be the evaluation of $\phi$ at the LGL quadrature nodes on the
interval $[-1,1]$. Similarly, let $\vec{q}_{b}$ be the evaluation of $q$ at
LGL quadrature nodes once they have been scaled to be between $[-1,0]$ and
$\mat{I}_{b}$ interpolate $\vec{\phi}$ to these same nodes; similar
definitions are used for $\vec{q}_{t}$ and $\mat{I}_{t}$ except on the interval
$[0, 1]$. Then if $\mat{M}$ is the exact mass matrix for integrating two
polynomials of degree $N$ evaluated at the LGL quadrature nodes in $[-1,1]$, then~\eref{eqn:L21dproj} is discretely
\begin{align}
  \vec{\phi}^{T}\mat{M}\mat{P}_{b}\vec{q}_{b} &=
  \frac{1}{2}\vec{\phi}^{T}\mat{I}_{b}^{T}\mat{M}\vec{q}_{b},&
  \vec{\phi}^{T}\mat{M}\mat{P}_{t}\vec{q}_{t} &=
  \frac{1}{2}\vec{\phi}^{T}\mat{I}_{t}^{T}\mat{M}\vec{q}_{t},
\end{align}
which holds for all $\vec{\phi}$, $\vec{q}_{t}$, and $\vec{q}_{b}$. Thus,
\begin{align}
  \mat{P}_{b} &= \frac{1}{2} \mat{M}^{-1} \mat{I}_{b}^{T} \mat{M},&
  \mat{P}_{t} &= \frac{1}{2} \mat{M}^{-1} \mat{I}_{t}^{T} \mat{M}.
\end{align}

For clarity, the derivations above are given for two-to-one nonconforming interfaces.
The method generalizes to the many-to-one nonconforming interfaces as well and
the derivation of the operators is similar to the two-to-one case.

\section{Construction of Modal Solution on the Heterogeneous Spherical Shell}\label{app:shell}
To construct a modal solution in the radial heterogeneous spherical
shell, the equations of isotropic elasticity are considered in spherical
coordinates. Since only the radial displacement is non-zero and the
displacement is of the form $u_{r} = \cos(t)\,\phi(r)$, these reduce to solving
the following boundary value problem:
\begin{align}
  \label{eqn:app:bvp}
  0 &= \rho\,\phi + \fd{\sigma_{rr}}{r} + \frac{1}{r}
  \left(2\sigma_{rr}-\sigma_{\theta\theta}-\sigma_{\varphi\varphi}\right),\\
  \sigma_{rr} &= \left(\lambda + 2\mu\right) \fd{\phi}{r} + \frac{2}{r}\lambda\,
  \phi,
  \\
  \sigma_{\theta\theta} &=
  \sigma_{\varphi\varphi} =
  \lambda \fd{\phi}{r} + \frac{2}{r}\left(\lambda+\mu\right) \phi,
\end{align}
with $\sigma_{rr} = 0$ at $r = R_{1}$ and $r = R_{2}$. We note that the solution
to this problem is not unique and for the test problem we just need to find a
particular solution to the equation.

To solve~\eref{eqn:app:bvp} the MATLAB Chebfun
package~\cite{DriscollHaleTrefethen2014} was used.  To find a nontrivial
solution, the initial value problem with values for $\phi(0) = 0$ and
$\fd{\phi}{r}(0) = -\frac{1}{3}$ is solved and then the domain of the solution is
chosen to be between $R_{1}$ and $R_{2}$, the first two roots of the resulting
$\sigma_{rr}$ field. The Chebfun script along with the data necessary for
constructing the high-order polynomial interpolant
are available at the GitHub repository
\url{https://github.com/bfam/spherical_shell}.

\section{Discrete Divergence Theorem for Tensor Product Hexahedral Elements with
LGL Quadrature}\label{app:div}

In this appendix stiffness matrices that satisfy~\eref{eqn:disc:div} are
constructed  for the tensor product hexahedral elements using Legendre-Gauss-Lobatto
(LGL) quadrature. The key step in ensuring this for curvilinear meshes is the
construction of the metric terms. Here, we assume that there is a single
polynomial order for the entire mesh $N$, so that $\hat{V}^{e} = \QQ^{N,3}$ for
all $e\in\EE$ and only the three-dimensional case is considered. The mortar
space is taken to be tensor product polynomial with $\hat{U}^{m} = \QQ^{N,2}$
for all $m \in \MM$ and the quadrature rule on the mortar is the tensor product
LGL quadrature.

Across nonconforming mesh interfaces (and edges) the larger face (or edge) is
referred to as the full face (or edge) and the smaller faces (or edges) as the
hanging faces (or edges).  The mortar element is always
assumed to conform to the face on the minus side of the mortar. For
mortar elements between conforming faces the minus side is arbitrary. When
full-side mortars are used the minus sides of nonconforming mortar elements are
the full faces and when split-side mortars are used minus sides are the hanging
faces. Since mortar elements conform to one of the volume faces, the metric
terms on the mortar (surface Jacobian determinant and unit normal vector) will
match the element face to which a mortar conforms.

The four key step we use to satisfy~\eref{eqn:disc:div} are:
\begin{itemize}
  \item the mesh is made discretely watertight;
  \item the metric terms are computed using a curl invariant
    form~\cite{Kopriva2006jsc} which avoids the need for a discrete product
    rule;
  \item aliasing errors along nonconforming faces and edges are incurred in a
    similar manner on all connected elements; and
  \item metric products are decreased in one polynomial order to ensure that
    certain face integrals are exact under LGL quadrature.
\end{itemize}

\subsection{Properties of LGL Quadrature}
Before discussing the construction of the metric terms and stiffness matrices,
we review some properties of LGL quadrature. For $q,p,J \in \QQ^{N,3}$ the
LGL approximation of the inner product is
\begin{align}
  \notag
  \int_{-1}^{1}\int_{-1}^{1}\int_{-1}^{1} Jpq &\approx
  \sum_{i=0}^{N}\sum_{j=0}^{N}\sum_{k=0}^{N}
  \omega_{i}\omega_{j}\omega_{k}
  J\left(\xi_{i}, \xi_{j}, \xi_{k}\right)
  p\left(\xi_{i}, \xi_{j}, \xi_{k}\right)
  q\left(\xi_{i}, \xi_{j}, \xi_{k}\right)\\
  \notag
  &=
  \vec{p}^{T} \mat{J} \left(\mat{\hat{W}}\otimes\mat{\hat{W}}\otimes\mat{\hat{W}}\right)\vec{q}\\
  \label{eqn:LGL}
  &=
  \vec{p}^{T} \mat{J} \mat{M} \vec{q}.
\end{align}
Here, $\omega_{i}$ are the quadrature weights and $\xi_{k}$ the quadrature
nodes which have the property that $\xi_{0} = -1$ and $\xi_{N} = 1$.
The vectors $\vec{p}$ and $\vec{q}$ are the stacking of the $p$ and $q$ at the
nodal locations; below we assume that the stacking has been done so that
$p\left(\xi_{i}, \xi_{j}, \xi_{k}\right)$ is element $i + 1 + jN + kN^2$ of
$\vec{p}$ (i.e., the first dimension is the fastest). Similarly, the diagonal
matrix $\mat{J}$ approximates $J$ at the quadrature nodes (ordering is the same
as the vectors $\vec{p}$ and $\vec{q}$). The quadrature weights have been
assembled into the diagonal matrix $\mat{\hat{W}}$ with element $i,i$
being $\omega_{i-1}$ and $\otimes$ being the matrix Kronecker product. We note
that the LGL integral approximation~\eref{eqn:LGL} will only be exact if the
product $Jpq\in\QQ^{2N-1,3}$.

Let $\mat{D}_{j}$ be the exact derivative matrix in the $r_{j}$ direction:
if $q \in \QQ^{N,3}$ and $p = \partial q/\partial r_{j}$ then
\begin{align}
  \mat{D}_{j} \vec{q} = \vec{p}.
\end{align}
If $\mat{\hat{D}}$ is the one-dimensional LGL derivative matrix then
\begin{align}
  \mat{D}_{1} &= \mat{I} \otimes \mat{I} \otimes \mat{\hat{D}}, &
  \mat{D}_{2} &= \mat{I} \otimes \mat{\hat{D}} \otimes \mat{I}, &
  \mat{D}_{3} &= \mat{\hat{D}} \otimes \mat{I} \otimes \mat{I},
\end{align}
with $\mat{I}$ being the $(N+1)\times(N+1)$ identity matrix. Since
differentiation decreases the polynomial order by one, the LGL differentiation
and quadrature matrices satisfy the one-dimensional summation-by-parts property~\cite{KoprivaGassner2010}:
\begin{align}
  \label{eqn:SBP}
  \mat{\hat{W}} \mat{\hat{D}} + \mat{\hat{D}}^{T} \mat{\hat{W}}
  &= \vec{e}_{N}\vec{e}_{N}^{T} - \vec{e}_{0}\vec{e}_{0}^{T},&
  \vec{e}_{0} &=
  \begin{bmatrix}
    1\\
    0\\
    \vdots\\
    0
  \end{bmatrix},&
  \vec{e}_{N} &=
  \begin{bmatrix}
    0\\
    \vdots\\
    0\\
    1
  \end{bmatrix}.
\end{align}
For the three-dimensional operators the following approximation of the
divergence theorem holds on the reference element:
\begin{align}
  \label{eqn:div1}
  \vec{p}^{T}
  \mat{M}
  \mat{D}_{1}
  \vec{q}
  +
  \vec{p}^{T}
  \mat{D}_{1}^{T}
  \mat{M}
  \vec{q}
  &=
  \vec{p}^{T}
  \left(\mat{\hat{W}}\otimes\mat{\hat{W}}\otimes\left(\vec{e}_{N}\vec{e}_{N}^{T}
  - \vec{e}_{0}\vec{e}_{0}^{T}\right)\right)
  \vec{q},\\
  \label{eqn:div2}
  \vec{p}^{T}
  \mat{M}
  \mat{D}_{2}
  \vec{q}
  +
  \vec{p}^{T}
  \mat{D}_{2}^{T}
  \mat{M}
  \vec{q}
  &=
  \vec{p}^{T}
  \left(\mat{\hat{W}}\otimes\left(\vec{e}_{N}\vec{e}_{N}^{T} -
  \vec{e}_{0}\vec{e}_{0}^{T}\right)\otimes\mat{\hat{W}}\right)
  \vec{q},\\
  \label{eqn:div3}
  \vec{p}^{T}
  \mat{M}
  \mat{D}_{3}
  \vec{q}
  +
  \vec{p}^{T}
  \mat{D}_{3}^{T}
  \mat{M}
  \vec{q}
  &=
  \vec{p}^{T}
  \left(\left(\vec{e}_{N}\vec{e}_{N}^{T} - \vec{e}_{0}\vec{e}_{0}^{T}\right)\otimes\mat{\hat{W}}\otimes\mat{\hat{W}}\right)
  \vec{q}.
\end{align}
The right-hand sides of the above equations are approximations of surface
integrals along the six faces of the reference element. In order to make this
explicit we number the faces of the element as:
\begin{center}
         face 1 is the face with $r_{1} = -1$,~
         face 2 is the face with $r_{1} = +1$,\\
         face 3 is the face with $r_{2} = -1$,~
         face 4 is the face with $r_{2} = +1$,\\
         face 5 is the face with $r_{3} = -1$,~
         face 6 is the face with $r_{3} = +1$,
\end{center}
and define the value of the volume vector on the faces as
\begin{align}
  \vec{p}^{e,1} &=
  \left(\mat{\hat{I}}\otimes\mat{\hat{I}}\otimes\vec{e}_{0}^{T}\right) \vec{p}^{e}
  =
  \mat{L}^{1}\vec{p}^{e}, &
  \vec{p}^{e,2} &=
  \left(\mat{\hat{I}}\otimes\mat{\hat{I}}\otimes\vec{e}_{N}^{T}\right) \vec{p}^{e}
  =
  \mat{L}^{2}\vec{p}^{e},\\
  \vec{p}^{e,3} &=
  \left(\mat{\hat{I}}\otimes\vec{e}_{0}^{T}\otimes\mat{\hat{I}}\right) \vec{p}^{e}
  =
  \mat{L}^{3}\vec{p}^{e}, &
  \vec{p}^{e,4} &=
  \left(\mat{\hat{I}}\otimes\vec{e}_{N}^{T}\otimes\mat{\hat{I}}\right) \vec{p}^{e}
  =
  \mat{L}^{4}\vec{p}^{e},\\
  \vec{p}^{e,5} &=
  \left(\vec{e}_{0}^{T}\otimes\mat{\hat{I}}\otimes\mat{\hat{I}}\right) \vec{p}^{e}
  =
  \mat{L}^{5}\vec{p}^{e}, &
  \vec{p}^{e,6} &=
  \left(\vec{e}_{N}^{T}\otimes\mat{\hat{I}}\otimes\mat{\hat{I}}\right) \vec{p}^{e}
  =
  \mat{L}^{6}\vec{p}^{e}.
\end{align}
With this notation~\eref{eqn:div1}--\eref{eqn:div3} can be rewritten as
\begin{align}
  {\left( \vec{p}^{e} \right)}^{T}
  \mat{M}
  \mat{D}_{1}
  \vec{q}^{e}
  +
  {\left( \vec{p}^{e} \right)}^{T}
  \mat{D}_{1}^{T}
  \mat{M}
  \vec{q}^{e}
  &=
  {\left( \vec{p}^{e,2} \right)}^{T} \mat{\bar{W}} \vec{q}^{e,2}
  -
  {\left( \vec{p}^{e,1} \right)}^{T} \mat{\bar{W}} \vec{q}^{e,1},\\
  {\left( \vec{p}^{e} \right)}^{T}
  \mat{M}
  \mat{D}_{2}
  \vec{q}^{e}
  +
  {\left( \vec{p}^{e} \right)}^{T}
  \mat{D}_{2}^{T}
  \mat{M}
  \vec{q}^{e}
  &=
  {\left( \vec{p}^{e,4} \right)}^{T} \mat{\bar{W}} \vec{q}^{e,4}
  -
  {\left( \vec{p}^{e,3} \right)}^{T} \mat{\bar{W}} \vec{q}^{e,3},\\
  {\left( \vec{p}^{e} \right)}^{T}
  \mat{M}
  \mat{D}_{3}
  \vec{q}^{e}
  +
  {\left( \vec{p}^{e} \right)}^{T}
  \mat{D}_{3}^{T}
  \mat{M}
  \vec{q}^{e}
  &=
  {\left( \vec{p}^{e,6} \right)}^{T} \mat{\bar{W}} \vec{q}^{e,6}
  -
  {\left( \vec{p}^{e,5} \right)}^{T} \mat{\bar{W}} \vec{q}^{e,5},
\end{align}
with $\mat{\bar{W}} = \mat{\hat{W}} \otimes \mat{\hat{W}}$ being the surface
quadrature rule.

\subsection{Stiffness Matrices and Element Based Relationships}
We now define the stiffness matrices for an $e\in\EE$ as
\begin{align}
  \mat{S}_{j}^{e} &= \mat{J}^{e}\mat{r}_{k,j}^{e} \mat{M}  \mat{D}_{k},
\end{align}
where the diagonal matrices $\mat{J}^{e}$ and $\mat{r}^{e}_{k,j}$ are approximations
of the Jacobian determinant and the metric derivative $\partial r_{k}/\partial
x_{j}$, respectively, at the nodal degrees of freedom. As discussed below,
the product $J\pd{r_{k}}{x_{j}}$ is computed from the of the metric derivatives
$\partial{x_{j}}/\partial{r_{k}}$. From~\eref{eqn:div1}--\eref{eqn:div3} it
follows that
\begin{align}
  \label{eqn:disc:div:elm:pre}
  \notag
  \vec{1}^{T} \mat{S}_{j}^{e} \vec{q}^{e} +
  \vec{1}^{T}\mat{J}^{e}\mat{r}_{k,j}^{e} \mat{D}_{k}^{T} \mat{M} \vec{q}^{e}
  =\;&
  \vec{1}^{T} \mat{J}^{e}\mat{r}_{k,j}^{e} \mat{M} \mat{D}_{k}\vec{q}^{e} +
  \vec{1}^{T}\mat{J}^{e}\mat{r}_{k,j}^{e} \mat{D}_{k}^{T} \mat{M} \vec{q}^{e}\\ \notag
  =\;
  & -\vec{1}^{T}\mat{J}^{e}\mat{r}_{1,j}^{e} {\left(\mat{L}^{1}\right)}^{T} \mat{\bar{W}} \vec{q}^{e,1}+\vec{1}^{T}\mat{J}^{e}\mat{r}_{1,j}^{e} {\left(\mat{L}^{2}\right)}^{T} \mat{\bar{W}} \vec{q}^{e,2}\\
  & -\vec{1}^{T}\mat{J}^{e}\mat{r}_{2,j}^{e} {\left(\mat{L}^{3}\right)}^{T} \mat{\bar{W}} \vec{q}^{e,3}+\vec{1}^{T}\mat{J}^{e}\mat{r}_{2,j}^{e} {\left(\mat{L}^{4}\right)}^{T} \mat{\bar{W}} \vec{q}^{e,4}\\ \notag
  & -\vec{1}^{T}\mat{J}^{e}\mat{r}_{3,j}^{e} {\left(\mat{L}^{5}\right)}^{T} \mat{\bar{W}} \vec{q}^{e,5}+\vec{1}^{T}\mat{J}^{e}\mat{r}_{3,j}^{e} {\left(\mat{L}^{6}\right)}^{T} \mat{\bar{W}} \vec{q}^{e,6}.
\end{align}
If we define the surface Jacobian determinant and outward normal on the faces as
\begin{align}
  \label{eqn:surf:norm12}
  \mat{S}_{J}^{e,1} \vec{n}^{e,1}_{j} &= -\mat{L}^{1} \mat{J}^{e}\mat{r}_{1,j}^{e}\vec{1}, &
  \mat{S}_{J}^{e,2} \vec{n}^{e,2}_{j} &=  \mat{L}^{2} \mat{J}^{e}\mat{r}_{1,j}^{e}\vec{1}, \\
  \label{eqn:surf:norm34}
  \mat{S}_{J}^{e,3} \vec{n}^{e,3}_{j} &= -\mat{L}^{3} \mat{J}^{e}\mat{r}_{2,j}^{e}\vec{1}, &
  \mat{S}_{J}^{e,4} \vec{n}^{e,4}_{j} &=  \mat{L}^{4} \mat{J}^{e}\mat{r}_{2,j}^{e}\vec{1}, \\
  \label{eqn:surf:norm56}
  \mat{S}_{J}^{e,5} \vec{n}^{e,5}_{j} &= -\mat{L}^{5} \mat{J}^{e}\mat{r}_{3,j}^{e}\vec{1}, &
  \mat{S}_{J}^{e,6} \vec{n}^{e,6}_{j} &=  \mat{L}^{6} \mat{J}^{e}\mat{r}_{3,j}^{e}\vec{1},
\end{align}
then relation~\eref{eqn:disc:div:elm:pre} simplifies to
\begin{align}
  \label{eqn:disc:div:elm:pre:surf}
  \vec{1}^{T} \mat{S}_{j}^{e} \vec{q}^{e} +
  \vec{1}^{T}\mat{J}^{e}\mat{r}_{k,j}^{e} \mat{D}_{k}^{T} \mat{M} \vec{q}^{e}
  =\;&
  \sum_{f=1}^{6}
  {\left(\mat{S}_{J}^{e,f} \vec{n}^{e,f}\right)}^{T} \mat{\bar{W}} \vec{q}^{e,f}.
\end{align}
The vector $\vec{n}^{e,f}_{j}$ is the outward unit normal to the element face at
the quadrature nodes and the diagonal matrix of surface Jacobian determinants
$\mat{S}_{J}^{e,f}$ is the normalization so that $\vec{n}^{e,f}_{j} \circ
\vec{n}^{e,f}_{j} = \vec{1}$ with $\circ$ being the componentwise
(Hadamard) product of two vectors.

With exact math, e.g., no aliasing errors, the following metric
identity~\cite{ThompsonWarsiMastin1985}
\begin{align}
  \label{eqn:met:id:bad}
  J^{e}\pd{r^{e}_{k}}{x_{j}} &=
  \pd{x^{e}_{j+1}}{r_{k+1}} \pd{x^{e}_{j-1}}{r_{k-1}} -
  \pd{x^{e}_{j+1}}{r_{k-1}} \pd{x^{e}_{j-1}}{r_{k+1}}
  &\mbox{(no summation over $j$ and $k$)},
\end{align}
with plus and minus for subscripts $j$ and $k$ defined cyclically on the set
$\{1,2,3\}$, can be used to show that
\begin{align}
  \label{eqn:jac:sum}
  \pd{}{r_{k}}\left(J^{e}\pd{r^{e}_{k}}{x_{j}}\right) = 0.
\end{align}
Hence, if we can show that discretely that
\begin{align}
  \label{eqn:jac:sum:disc}
  \mat{D}_{k}\left(\mat{J}^{e}\mat{r}_{k,j}^{e}\vec{1}\right) = \vec{0},
\end{align}
then we will have shown that discretely the divergence theorem holds at the
element level; this is not quite~\eref{eqn:disc:div} since it does not
involve the mortar elements.  Unfortunately, using the metric identities in the
form of~\eref{eqn:met:id:bad} to show~\eref{eqn:jac:sum} requires the use of the
product rule, and in general this does not hold discretely for all
$p^{e},q^{e}\in\QQ^{N,3}$:
\begin{align}
  \mat{D}_{k}\left(\vec{p}^{e}\circ\vec{q}^{e}\right)
  \neq
  \left(\mat{D}_{k}\vec{p}^{e}\right)
  \circ
  \vec{q}^{e}
  +
  \vec{p}^{e}
  \circ
  \left(\mat{D}_{k}\vec{q}^{e}\right).
\end{align}
Fortunately, if instead of using~\eref{eqn:met:id:bad} the equivalent curl
invariant form~\cite{Kopriva2006jsc} is used,
\begin{align}
  \label{eqn:met:id:good}
  J^{e}\pd{r^{e}_{k}}{x_{j}} =\;&
  \frac{1}{2}
  \pd{}{r_{k+1}}\left(
  x^{e}_{j+1} \pd{x^{e}_{j-1}}{r_{k-1}} -
  \pd{x^{e}_{j+1}}{r_{k-1}} x^{e}_{j-1}
  \right)\\\notag&
  -
  \frac{1}{2}
  \pd{}{r_{k-1}}\left(
  x^{e}_{j+1} \pd{x^{e}_{j-1}}{r_{k+1}} -
  \pd{x^{e}_{j+1}}{r_{k+1}} x^{e}_{j-1}
  \right),
  &\mbox{(no summation over $j$ and $k$)},
\end{align}
it can be shown that~\eref{eqn:jac:sum} holds without invoking the product rule.
To show that this form is beneficial discretely, we define
\begin{align}
  \label{eqn:bar:zeta}
  \bar{\zeta}^{e}_{j,k} &=
  x^{e}_{j+1} \pd{x^{e}_{j-1}}{r_{k}} -
  \pd{x^{e}_{j+1}}{r_{k}} x^{e}_{j-1}
  &\mbox{(no summation over $j$)},
\end{align}
so that the curl invariant metric identities~\eref{eqn:met:id:good} can be
rewritten as
\begin{align}
  \label{eqn:met:id:good2}
  J^{e}\pd{r^{e}_{k}}{x_{j}} =\;&
  \frac{1}{2} \left(\pd{\bar{\zeta}^{e}_{j,k-1}}{r_{k+1}}
  -
  \pd{\bar{\zeta}^{e}_{j,k+1}}{r_{k-1}}\right)
  &\mbox{(no summation over $k$)};
\end{align}
here we have added the overbar accent so that once the final discrete mesh is
defined it can be unaccented since in general $\bar{\zeta}^{e}_{j,k} \notin
\QQ^{N,3}$.
Let $\zeta^{e}_{j,k}\in\QQ^{N,d}$ be some approximation of $\bar{\zeta}^{e}_{j,k}$ at the
nodal degrees of freedom (the details of how we approximate these terms will be
given below). The discrete approximation of
$J^{e}\pd{r^{e}_{k}}{x_{j}}$ can be defined directly from~\eref{eqn:met:id:good2} as
\begin{align}
  \label{eqn:Jr}
  \mat{J}^{e}\mat{r}^{e}_{k,j} &= \mbox{diag}\left(\frac{1}{2}\left(
  \mat{D}_{k+1}\vec{\zeta}^{e}_{j,k-1}
  -
  \mat{D}_{k-1}\vec{\zeta}^{e}_{j,k+1}\right)\right),
  &\mbox{(no summation over $k$)},
\end{align}
where the operator $\mbox{diag}\left(\cdot\right)$ makes a diagonal matrix from
a vector (with component $n$ of the vector being the $n,n$ element of the
matrix). Direct computation of~\eref{eqn:jac:sum:disc} then gives
\begin{align}
  \notag
  \mat{D}_{k}\left(\mat{J}^{e}\mat{r}_{k,j}^{e}\vec{1}\right)
  &=
  \sum_{k=1}^{3}
  \frac{1}{2}\left(
  \mat{D}_{k}\mat{D}_{k+1}\vec{\zeta}^{e}_{j,k-1}
  -
  \mat{D}_{k}\mat{D}_{k-1}\vec{\zeta}^{e}_{j,k+1}\right)\\
  \notag
  &=
  \sum_{k=1}^{3}
  \frac{1}{2}\left(
  \mat{D}_{k}\mat{D}_{k+1}\vec{\zeta}^{e}_{j,k-1}
  -
  \mat{D}_{k+1}\mat{D}_{k}\vec{\zeta}^{e}_{j,k-1}\right)\\
  \notag
  &=
  \sum_{k=1}^{3}
  \frac{1}{2}\left(
  \mat{D}_{k}\mat{D}_{k+1}\vec{\zeta}^{e}_{j,k-1}
  -
  \mat{D}_{k}\mat{D}_{k+1}\vec{\zeta}^{e}_{j,k-1}\right)\\
  &=\vec{0},
\end{align}
where the second equality is a shifting of the summation index in the second
term by $+1$ (which is permissible since all values of $k$ are covered by the
sum) and the third equality follows from the commutative property of the
derivative operators: $\mat{D}_{k}\mat{D}_{j} = \mat{D}_{j}\mat{D}_{k}$.  Thus,~\eref{eqn:disc:div:elm:pre} becomes
\begin{align}
  \label{eqn:disc:div:elm}
  \vec{1}^{T} \mat{S}_{j}^{e} \vec{q}^{e} &=
  \sum_{f=1}^{6}
  {\left(\mat{S}_{J}^{e,f} \vec{n}^{e,f}\right)}^{T} \mat{\bar{W}}
  \vec{q}^{e,f},
\end{align}
and the divergence theorem holds at the element level once we have specified how
$\zeta^{e}_{j,k}$ is approximated from $\bar{\zeta}^{e}_{j,k}$. Importantly, the
relationship~\eref{eqn:disc:div:elm} holds for any approximation of
$\zeta^{e}_{j,k}$ and this fact will be exploited to ensure that the
discrete divergence relation~\eref{eqn:disc:div} holds. Namely if we can show
that the metric terms can be constructed so that
\begin{align}
  \label{eqn:surf:elm:mort1}
  {\left(\mat{S}_{J}^{e,f} \vec{n}^{e,f}_{j}\right)}^{T} \mat{\bar{W}} \vec{q}^{e,f}
  =
  \sum_{m\in\MM^{e,f}} {\left(\vec{n}_{j}^{m[e]}\right)}^{T} \mat{W}^{m}
  \mat{P}^{m,e} \vec{q}^{e},
\end{align}
then we will have shown that~\eref{eqn:disc:div} holds; here the set $\MM^{e,f}
\subset \MM^{e}$ is the set of mortar elements connected to face $f$ of element
$e$. Recall that $\mat{W}^{m}$ includes the surface Jacobian determinants, and
since the mortar quadrature rule is the tensor product LGL quadrature
\begin{align}
  \mat{W}^{m} = \mat{S}_{J}^{m} \mat{\bar{W}},
\end{align}
where $\mat{S}_{J}^{m}$ is a diagonal matrix of mortar-based surface Jacobian
determinants. For an $m\in\MM^{e,f}$ we define the projection operator
$\mat{P}^{m,e}$ to first interpolate to face $f$ followed by projection from
face $f$ to $m$, which allows us to write $\mat{P}^{m,e} =
\mat{P}^{m,e[f]}\mat{L}^{f}$ where $\mat{P}^{m,e[f]}$ is the projection from
face $f$ of element $e$ to mortar $m$.  With these definitions, relation~\eref{eqn:surf:elm:mort1} becomes
\begin{align}
  \label{eqn:surf:elm:mort}
  {\left(\mat{S}_{J}^{e,f} \vec{n}_{j}^{e,f}\right)}^{T} \mat{\bar{W}} \vec{q}^{e,f}
  =
  \sum_{m\in\MM^{e,f}} {\left(\mat{S}_{J}^{m} \vec{n}_{j}^{m[e]}\right)}^{T}
  \mat{\bar{W}}
  \mat{P}^{m,e[f]} \vec{q}^{e,f},
\end{align}
which is a discrete statement on the accuracy of the projection from $e$ to
$m$ when geometry is included. To ensure that~\eref{eqn:surf:elm:mort} is satisfied, we require that
$S_{J}^{m}n_{j}^{m[e]}, S_{J}^{e,f}n_{j}^{e,f}\in\QQ^{N-1,2}$ and that for
connected faces and mortars these are computed in a consistent manner. With
$S_{J}^{m}n_{j}^{m[e]}, S_{J}^{e,f}n_{j}^{e,f}\in\QQ^{N-1,2}$
the quadrature approximations in~\eref{eqn:surf:elm:mort} are exact since the
products in the integral are in $\QQ^{2N-1,2}$ which are integrated exactly by
the tensor product LGL quadrature rule.

Here it is worth noting that one cannot simply interpolate the surface Jacobian
determinant and normal vectors from the full faces to the hanging faces. Though
this would result in~\eref{eqn:surf:elm:mort}, it would cause the volume and
surface metrics to be inconsistent so that~\eref{eqn:disc:div:elm:pre} and~\eref{eqn:disc:div:elm:pre:surf} are no longer equivalent. Thus instead of
modifying $S_{J}^{e,f} n_{j}^{e,f}$ directly. We modify $\zeta^{e}_{j,k}$ so
that the surface, volume, and mortar metrics are consistently calculated.

\subsection{Approximation of Metric Terms}
Before proceeding to the discussion of the approximation of $\zeta^{e}_{j,k}$ it is
necessary to define how the geometry transform is handled. Let $\tilde{x}_{j}^{e}
\in \QQ^{N,3}$ be the evaluation of the coordinate mapping $\bar{x}^{e}_{j}$;
the tilde and bar accents are added here so that the final coordinate points can
be unaccented.  As noted in the text, across nonconforming faces
$\tilde{x}_{j}^{e}$ may not be continuous, thus we let $x_{j}^{e} \in \QQ^{N,3}$
be a discretely watertight version of the geometry. In our code, we make the
mesh watertight by interpolating $\tilde{x}^{e}_{j}$ across nonconforming faces
and edges from the full faces and edges to hanging faces and edges\footnote{In
three-dimensions interpolation across nonconforming edges is necessary so that
gaps and overlaps do not occur between conforming faces connected to
nonconforming edges.}; degrees of freedom interior to the element are not
modified in the procedure used to make the mesh watertight. It is from
$x_{j}^{e}$ that we construct the metric terms.

By~\eref{eqn:surf:norm12}--\eref{eqn:surf:norm56} it follows that:
\begin{align}
  &\mbox{$\zeta^{e}_{j,1}$ affects $S_{J}^{e,f}n_{j}^{e,f}$ for faces $f=3,4,5,6$,}\notag\\
  &\mbox{$\zeta^{e}_{j,2}$ affects $S_{J}^{e,f}n_{j}^{e,f}$ for faces $f=1,2,5,6$,}\notag\\
  &\mbox{$\zeta^{e}_{j,3}$ affects $S_{J}^{e,f}n_{j}^{e,f}$ for faces $f=1,2,3,4$,}\notag
\end{align}
which, along with~\eref{eqn:Jr}, implies that
$S_{J}^{e,f}n_{j}^{e,f}\in\QQ^{N-1,2}$ if
\begin{align}
  \label{eqn:zeta:space1}
  \zeta^{e}_{j,1} &\in \QQ^{N-1} \otimes \QQ^{N} \otimes \QQ^{N} \mbox{ along faces $f=3,4,5,6$},\\
  \label{eqn:zeta:space2}
  \zeta^{e}_{j,2} &\in \QQ^{N} \otimes \QQ^{N-1} \otimes \QQ^{N} \mbox{ along faces $f=1,2,5,6$},\\
  \label{eqn:zeta:space3}
  \zeta^{e}_{j,3} &\in \QQ^{N} \otimes \QQ^{N} \otimes \QQ^{N-1} \mbox{ along faces $f=1,2,3,4$};
\end{align}
namely $\zeta^{e}_{j,k}$ is one polynomial order lower when interpolated to
faces that it affects.  We initially calculate in each element $e$
\begin{align}
  \vec{\tilde{\zeta}}^{e}_{j,k} =
  \mat{P}^{N \rightarrow N-1}_{k}
  \left(
  \vec{x}^{e}_{j+1}
  \circ
  \left(\mat{D}_{k}\vec{x}^{e}_{j-1}\right)
  -
  \left(\mat{D}_{k}\vec{x}^{e}_{j+1}\right)
  \circ
  \vec{x}^{e}_{j-1}
  \right).
\end{align}
Here the matrix $\mat{P}^{N \rightarrow N-1}_{k}$, and its continuous
counterpart $\mathcal{P}^{N\rightarrow N-1}_{k}$, is an L$^{2}$-projection
operator which removes the highest mode in the $r_{k}$-direction along faces
which are parallel to $r_{k}$. Thus, $\mathcal{P}^{N \rightarrow N-1}_{1}$ modifies
$\tilde{\zeta}^{e}_{j,1}$ along faces 3, 4, 5, and 6 so that with respect to
$r_{1}$ it is a polynomial of degree $N-1$; the dependence of
$\tilde{\zeta}^{e}_{j,1}$ on $r_{2}$ and $r_{3}$ is unmodified by
$\mathcal{P}^{N \rightarrow N-1}_{1}$.

The quantities $\zeta^{e}_{j,k}$ are then the interpolation and scaling of
$\tilde{\zeta}^{e}_{j,k}$ from the full faces and edges to the hanging faces and
edges. To make this concrete, suppose that face $1$ of element $e$ is hanging,
we modify $\tilde{\zeta}^{e}_{j,2}$ and $\tilde{\zeta}^{e}_{j,3}$
on face $1$ so that the surface Jacobian determinant times the normal vector
will be consistent. Let face $f'$ of element $e'$ be the full face for this
nonconforming interface. We replace values of $\tilde{\zeta}^{e}_{j,2}$ on
degrees of freedom along face $1$ of $e$ by the values interpolated from face
$f'$ of element $e'$ as
\begin{align}
  \label{eqn:zeta2}
  \vec{\zeta}^{e}_{j,2} =
  \left(\mat{I} - {\left(\mat{L}^{1}\right)}^{T} \mat{L}^{1}\right)
  \vec{\tilde{\zeta}}^{e}_{j,2}
  +
  \frac{1}{2} {\left(\mat{L}^{1}\right)}^{T} \mat{I}^{e' \rightarrow e}
  \mat{L}^{f'} \vec{\tilde{\zeta}}^{e'}_{j,k'_{2}}.
\end{align}
Here $k'_{2}$ is the reference direction of element $e'$ that corresponds to
$r_{2}$ of element $e$. The matrix $\mat{I}^{e' \rightarrow e}$ interpolates
from the full face of $e'$ to the hanging face of $e$ as well as taking into
account any flips or rotations needed to transform between the two elements
reference directions. For instance, if face $1$ of $e$ was in the lower-left
quadrant of the full face $\mat{I}^{e' \rightarrow e}$ would interpolate from
tensor product LGL points over $[-1,1]\times[-1,1]$ to points over
$[-1,0]\times[-1,0]$. The scaling by $1/2$ in~\eref{eqn:zeta2} is needed in
order to transform the derivative from $\hat{e}'$ to $\hat{e}$.
The same procedure can be applied to calculate $\vec{\zeta}^{e}_{j,3}$ from
$\tilde{\zeta}^{e}_{j,3}$ and $\tilde{\zeta}^{e'}_{j,k_{3}'}$. Since
$\zeta^{e}_{j,1}$ does not affect $S_{J}^{e,1}n_{j}^{e,1}$ we set
$\zeta^{e}_{j,1} = \tilde{\zeta}^{e}_{j,1}$.

For hanging edges, a similar procedure is used, except that only one of the
$\zeta^{e}_{j,k}$ values needs to be updated. For instance if
the edge between faces $1$ and $3$ is hanging, $\zeta^{e}_{j,3}$ needs to be
defined from $\tilde{\zeta}^{e}_{j,3}$ and $\tilde{\zeta}^{e'}_{j,k'_{3}}$ where
element $e'$ has a full edge corresponding to hanging edge of element $e$.

When multiple faces and edges are hanging, the procedure outlined above can be
applied iteratively to form $\zeta^{e}_{j,k}$. In our implementation we found it
convenient to have element $e$ mimic the calculation that $e'$ performs along
the corresponding full faces and edges. This alternate interpretation makes it
clear that our approach ensures that aliasing errors in the calculation of
$\zeta^{e}_{j,k}$ are incurred in a similar manner across hanging faces and edges,
and that the projection out of the highest modes is done over the same domain.

We now show that the above approach to the computation of the metric
terms results in~\eref{eqn:surf:elm:mort}. We first consider the minus side of
a mortar element. Namely, if $m\in\MM^{e,f}$ with $e \in
\EE^{-m}$, by definition $S_{J}^{m} n_{j}^{m[e]} = S_{J}^{e,f} n_{j}^{e,f}$
(since mortar elements always conform to the minus side faces). It
follows that $\mathcal{P}^{m,e[f]}$ is an identity operation and that $\MM^{e,f}
= \{m\}$, so~\eref{eqn:surf:elm:mort} follows immediately. Thus, all that
remains is to show that~\eref{eqn:surf:elm:mort} holds for elements on the plus
side of the mortar.

In the remaining we are going to assume that face $1$ of element $e$ is on the
plus side of all mortars in $\MM^{e,1}$. Additionally, without loss of
generality we assume that for each $m \in \MM^{e,1}$ that the minus side element
$e' \in \EE^{-m}$ is connected to $m$ through face $2$; by definition $e'$ is
the only element on the minus side of $m$.  We also assume that reference
directions are the same for both $e$ and $e'$, so that the ordering of there
degrees of freedom is the same.  These assumptions mean that
$S_{J}^{e,1}n_{j}^{e,1}$ is calculated from $\zeta_{j,2}^{e}$ and
$\zeta_{j,3}^{e}$ and that for each $m \in \MM^{e,1}$ and $e' \in \EE^{-m}$ that
the product $S_{J}^{e',1}n_{j}^{e',1} = S_{J}^{m}n_{j}^{m}$ is calculated from
$\zeta_{j,2}^{e'}$ and $\zeta_{j,3}^{e'}$.

Let face $1$ of element $e$ be on a conforming interface. Let $m \in \MM^{e,1}$
be the mortar element that face $1$ of $e$ is connected to and let face $2$ of
$e'$ be the element face on the minus side of $m$. By the orientation assumption
it follows that
\begin{align}
  \mat{L}^{1}\vec{\zeta}^{e}_{j,2} &= \mat{L}^{2}\vec{\zeta}^{e'}_{j,2},&
  \mat{L}^{1}\vec{\zeta}^{e}_{j,3} &= \mat{L}^{2}\vec{\zeta}^{e'}_{j,3};
\end{align}
namely these quantities are continuous across the interface.
Note that $S_{J}^{e,1}n_{j}^{e,1}$ and $S_{J}^{e',2}n_{j}^{e',2}$ are calculated
from $J^{e}r^{e}_{1,j}$ and $J^{e'}r^{e'}_{1,j}$, respectively;
see~\eref{eqn:surf:norm12}. This means that $S_{J}^{e,1}n_{j}^{e,1} =
-S_{J}^{e',2}n_{j}^{e',2}$ since $J^{e'}r^{e'}_{1,j}$ and $J^{e'}r^{e'}_{1,j}$
are calculated using derivatives in the $r_{2}$ and $r_{3}$ directions (i.e.,
directions along faces $1$ and $2$), and since $S_{J}^{e',2}n_{j}^{e',2} =
S_{J}^{m}n_{j}^{m}$ it follows that $S_{J}^{m}n_{j}^{m[e]} =
-S_{J}^{m}n_{j}^{m[e]} = S_{J}^{e,1}n_{j}^{e,1}$ and
\begin{align}
  {\left(\mat{S}_{J}^{e,1} \vec{n}_{j}^{e,1}\right)}^{T} \mat{\bar{W}} \vec{q}^{e,1}
  =
  {\left(\mat{S}_{J}^{m} \vec{n}_{j}^{m[e]}\right)}^{T}
  \mat{\bar{W}}
  \vec{q}^{e,1};
\end{align}
in this case the projection from face $1$ of $e$ to the mortar is an identity
operation.

We now move on to the case that face $1$ of $e$ is connected to a nonconforming
interface. Before proceeding we note that since the approximations
$\zeta^{e}_{j,k}$ are in the spaces indicated by~\eref{eqn:zeta:space1}--\eref{eqn:zeta:space3}, it follows that on each
$S_{J}^{e,1}n_{j}^{e,1} \in \QQ^{N-1, 2}$ and that $S_{J}^{m}n_{j}^{m} \in
\QQ^{N-1,2}$ for all $m \in \MM$. Given this, and the accuracy of the
quadrature rule, it follows that the left- and right-hand sides of~\eref{eqn:surf:elm:mort} can be replaced with integrals:
\begin{align}
  {\left(\mat{S}_{J}^{e,1} \vec{n}_{j}^{e,1}\right)}^{T} \mat{\bar{W}} \vec{q}^{e,1}
  &=
  \int_{-1}^{1}\int_{-1}^{1} S_{J}^{e,1} n_{j}^{e,1} q^{e,1},\\
  {\left(\mat{S}_{J}^{m} \vec{n}_{j}^{m[e]}\right)}^{T}
  \mat{\bar{W}}
  \mat{P}^{m,e[1]} \vec{q}^{e,1}
  &=
  \int_{-1}^{1}\int_{-1}^{1} S_{J}^{m} n_{j}^{m[e]} \mathcal{P}^{m,e[1]}
  q^{e,1},
\end{align}
where $\mathcal{P}^{m,e[1]}$ is the projection operator from face $1$ of element
$e$ to the mortar (which corresponds with the matrix operator
$\mat{P}^{m,e[1]}$).  We now consider the case of the full-side and split-side
mortars separately.

When full-side mortars are being used, face $1$ of element $e$ is a hanging face
(since a minus side of a full-side mortar is the full face). Also, it follows
that there is only one mortar element in $\MM^{e,1}$ which we denote as $m$; as
before $e'$ is the element connected to the minus side of $m$ through face $2$.
Since face $2$ of element $e'$ is the full face, we have that by~\eref{eqn:zeta2}:
\begin{align}
  \vec{\zeta}^{e}_{j,2} &=
  \left(\mat{I} - {\left(\mat{L}^{1}\right)}^{T} \mat{L}^{1}\right)
  \vec{\tilde{\zeta}}^{e}_{j,2}
  +
  \frac{1}{2} {\left(\mat{L}^{1}\right)}^{T} \mat{I}^{e' \rightarrow e}
  \mat{L}^{2} \vec{\tilde{\zeta}}^{e'}_{j,2},\\
  \vec{\zeta}^{e}_{j,3} &=
  \left(\mat{I} - {\left(\mat{L}^{1}\right)}^{T} \mat{L}^{1}\right)
  \vec{\tilde{\zeta}}^{e}_{j,3}
  +
  \frac{1}{2} {\left(\mat{L}^{1}\right)}^{T} \mat{I}^{e' \rightarrow e}
  \mat{L}^{2} \vec{\tilde{\zeta}}^{e'}_{j,3}.
\end{align}
As in the conforming case, $S_{J}^{e,1}n_{j}^{e,1}$ and
$S_{J}^{e',2}n_{j}^{e',2}$ are calculated using only derivatives in the $r_{2}$
and $r_{3}$ directions, thus it follows that
\begin{align}
  \mat{L}^{1} \left(\mat{J}^{e}\mat{r}^{e}_{1,j}\vec{1}\right)
  =
  \frac{1}{4}\mat{I}^{e' \rightarrow e}
  \mat{L}^{2} \left(\mat{J}^{e}\mat{r}^{e'}_{1,j}\vec{1}\right);
\end{align}
here the factor of $1/4$ arises because of the $1/2$ in~\eref{eqn:zeta2} along with an additional derivative taken on
reference elements. This means that
\begin{align}
  \mat{S}_{J}^{e,1}\vec{n}_{j}^{e,1}
  =
  -
  \frac{1}{4}\mat{I}^{e' \rightarrow e}
  \mat{S}_{J}^{m}\vec{n}_{j}^{m}
  =
  \frac{1}{4}\mat{I}^{e' \rightarrow e}
  \mat{S}_{J}^{m}\vec{n}_{j}^{m[e]},
\end{align}
and $S_{J}^{e,1}n_{j}^{e,1}$ and $S_{J}^{m}n_{j}^{m[e]}$ are
the same polynomial. Hence, since we
use exact L$^2$ projection to go from face $1$ of $e$ to $m$ it follows that
\begin{align}
  \int_{-1}^{1}\int_{-1}^{1} S_{J}^{e,1} n_{j}^{e,1} q^{e,1}
  &=
  \int_{-1}^{1}\int_{-1}^{1} S_{J}^{m} n_{j}^{m[e]} \mathcal{P}^{m,e[1]}
  q^{e,1},
\end{align}
which is the desired result.

When split-side mortars are used, face $1$ of element $e$ is a full face (with
the minus side of the mortar being the hanging face). Since $e$ is the full-side, it follows that $\MM^{e,1}$ contains four mortar elements which we denote
with $m_{l}$ for $l = 1,2,3,4$. Let the minus side of $m_{l}$ be face $2$ of
element $e'_{l}$. Additionally, since mortar $m_{l}$ conforms face $2$ of
$e'_{l}$ it follows from~\eref{eqn:surf:norm12} and~\eref{eqn:zeta2} that
\begin{align}
  \mat{S}_{J}^{m'_{l}}\vec{n}_{j}^{m'_{l}[e]}
  =
  -
  \mat{S}_{J}^{m'_{l}}\vec{n}_{j}^{m'_{l}}
  =
  -
  \mat{S}_{J}^{e'_{l},2}\vec{n}_{j}^{e'_{l},2}
  =
  \frac{1}{4}\mat{I}^{e \rightarrow e'_{l}}
  \mat{S}_{J}^{e,1}\vec{n}_{j}^{e,1},
\end{align}
namely the surface Jacobian determinant times normal on the mortar is the scaled
interpolant of the product on face $1$ of $e$.  Since face $1$ of element $e$
is the full face, it follows that $\mathcal{P}^{m_{l},e[1]} =
\mathcal{I}^{e\rightarrow e'_{l}}$, namely the projection from face $1$ of $e$
to $m_{l}$ is the same as interpolation to the face $2$ of element $e'_{l}$.
Thus we have that
\begin{align}
  \notag
  \int_{-1}^{1}\int_{-1}^{1} S_{J}^{e,1} n_{j}^{e,1} q^{e,1}
  =&
  \int_{0}^{1}\int_{0}^{1} S_{J}^{e,1} n_{j}^{e,1} q^{e,1}
  +
  \int_{-1}^{0}\int_{0}^{1} S_{J}^{e,1} n_{j}^{e,1} q^{e,1}\\
  \notag
  &+
  \int_{0}^{1}\int_{-1}^{0} S_{J}^{e,1} n_{j}^{e,1} q^{e,1}
  +
  \int_{-1}^{0}\int_{-1}^{0} S_{J}^{e,1} n_{j}^{e,1} q^{e,1}\\
  =&
  \sum_{m_{l}\in\MM^{e,1}}
  \int_{-1}^{1}\int_{-1}^{1} S_{J}^{m_{l}} n_{j}^{m_{l}[e]}
  \mathcal{I}^{e\rightarrow e'_{l}} q^{e,1}.
\end{align}

For conforming mortars and the two types of nonconforming mortars (split- and
full-side) considered in this work we have shown that~\eref{eqn:surf:elm:mort}
holds. As noted above the properties of LGL quadrature imply that if the volume
and surface metric terms are consistently calculated and if~\eref{eqn:surf:elm:mort} holds, then the discrete divergence relation~\eref{eqn:disc:div} holds. To ensure that the metric terms are consistent across
the mortars it is necessary to make sure that certain intermediate calculations
are consistent across nonconforming faces and edges.  Making the mesh discretely
watertight alone is not sufficient since differences in the aliasing errors,
which result from interpolating high-order polynomials to a lower-order spaces,
are different across the nonconforming faces. Finally, in the above construction
(and our implementation), the polynomial order of $S_{J}^{e,f}n_{j}^{e,f}$ is
lowered for all faces (conforming and nonconforming). This is a bit more
aggressive than needed, and~\eref{eqn:surf:elm:mort} would be satisfied if the
reduction of polynomial order was done only across the nonconforming faces and
edges.

\end{document}